%% file: main.tex
\documentclass[review]{elsarticle}
\usepackage{hyperref}
\usepackage{amsmath,graphicx}
\usepackage{float}
\usepackage{mathtools}
\usepackage{amssymb,amsthm}
\usepackage{algorithm}
\usepackage[noend]{algpseudocode}
\usepackage[utf8]{inputenc}
\usepackage[english]{babel}
\usepackage{color}
\usepackage{hyperref}
\usepackage{booktabs}
\usepackage{geometry}
\usepackage{array}
\usepackage{siunitx}
\usepackage{multirow}

\newcommand{\y}{\textit{\textbf{y}}}

\newcommand{\cblue}{\textcolor{black}}

\newtheorem{remark}{\em Remark}

\newtheorem{assumption}{\em Assumption}
\newtheorem{lemma}{\em Lemma}
\newtheorem{theorem}{\em Theorem}
\newtheorem{corollary}{\em Corollary}
 
\newcommand{\appropto}{\mathrel{\vcenter{
  \offinterlineskip\halign{\hfil$##$\cr
    \propto\cr\noalign{\kern2pt}\sim\cr\noalign	{\kern-2pt}}}}}
\newcommand{\normalized}{\tilde}

\newcommand{\Ic}{\mathcal{I}}

\newcommand{\Zc}{\mathcal{Z}}
\newcommand{\Var}{\mathrm{Var}}

\newcommand{\bR}{\mathbb{R}}  
\newcommand{\bN}{\mathbb{N}} 
\newcommand{\Real}{\mathbb{R}} 
\def\x{{\bf x}}

\def\w{{\bf w}}
\def\u{{\bf u}}

\def\t{{\bf t}}

\def\A{{\bf A}}
\def\B{{\bf B}}
\def\T{{\bf T}}
\def\W{{\bf W}}
\makeatletter
\def\BState{\State\hskip-\ALG@thistlm}
\makeatother

\makeatletter
\newcommand{\skipitems}[1]{%
	\addtocounter{\@enumctr}{#1}%
}
\makeatother

\journal{Journal of The Franklin Institute}

\begin{document}

\begin{frontmatter}

\title{A novel sequential method for building upper and lower bounds of moments of distributions}


\author[mymainaddress]{Solal Martin\corref{mycorrespondingauthor}}
\cortext[mycorrespondingauthor]{Corresponding author}
\ead{solal.martin@tse-fr.eu}

\author[mysecondaryaddress]{Emilie Chouzenoux}

\author[mythirdadress]{Víctor Elvira}

\address[mymainaddress]{Toulouse School of Economics, France}
\address[mysecondaryaddress]{
Universit\'e Paris-Saclay, Inria, CentraleSupélec, CVN, France}
\address[mythirdadress]{School of Mathematics, University of Edinburgh, UK}
 
\begin{abstract}
Approximating integrals is a fundamental task in probability theory and statistical inference, and their applied fields of signal processing, and Bayesian learning, as soon as expectations over probability distributions must be computed efficiently and accurately. 
When these integrals lack closed-form expressions, numerical methods must be used, from the Newton-Cotes formulas and Gaussian quadrature, to Monte Carlo and variational approximation techniques. Despite these numerous tools, few are guaranteed to preserve majoration/minoration inequalities, while this feature is fundamental in certain applications in statistics. %For instance, when evaluating the variance of a statistical estimator, one wants to have an upper bound on such variance, to perform accurate worst-case analysis and retro-active hyper-parameter tuning. 

%which that possess both a controlled precision, and a . Such approximations are central in various applications, including signal processing and Bayesian learning, where expectations over probability distributions must be computed efficiently and accurately.
%While classical numerical methods—like Newton-Cotes formulas and Gaussian quadrature—work well for low-dimensional, smooth integrands, high-dimensional settings require stochastic techniques. Monte Carlo methods, along with modern developments such as importance sampling and variational approaches, have greatly improved our ability to handle complex integration problems.

In this paper, we focus on the integration problem arising in the estimation of moments of scalar unnormalized distributions. We introduce a sequential method for constructing upper and lower bounds on the sought integral. Our approach leverages the majorization-minimization framework to iteratively refine these bounds  using an envelope principle. The method has proven convergence and controlled accuracy under mild conditions. \cblue{We then generalize the method to the multi-dimensional setting, along with an effective implementation strategy based on power diagrams.} We demonstrate \cblue{the effectiveness of the proposed approach} through a detailed numerical example of the estimation of a Monte Carlo sampler variance in a Bayesian inference problem, \cblue{in one- and two-dimensional     cases}. 
\end{abstract}
 

\end{frontmatter}

%\linenumbers
\newpage

\section{Introduction and Problem Statement}
\label{sec_problem}
\input{01_Problem_Statement}

\section{A first construction of integral bounds}
\label{sec_gaussbounds}
\input{02a_intro_gaussbound}
\input{02b_minorant_gaussbound}
\input{02c_majorant_gaussbound}
\input{02d_discussion_gaussbound}

\section{Sequential refinement of the bounds}
\label{sec_compound}
\input{03a_piecewisebound_compound}
\input{03b_algo_compound}

\section{Convergence Analysis}
\label{sec_convergence_analysis}
\input{04_convergence_proofs}

\color{black}

\section{Application to importance sampler variance estimation}
\label{sec_IS}
\input{05a_IS_numapp}

\section{Numerical Results}
\label{sec_logistic}
\input{06a_experimental_numapp}
\input{06b_description_experiments}

\cblue{
\section{Generalisation to the multidimensional setting}
\label{sec_multidim}
\input{07_multidimensionnal_case}
}
\section{Conclusion}
\label{sec_conclusion}
In this article, we introduced a new algorithm for computing accurate bounds on the moments of one-dimensional distributions, \cblue{and its generalization to the multi-dimensional case}. An important application is the estimation of the variance of an importance sampling estimator, allowing a facilitated tuning of proposal hyper-parameters. Our method achieves excellent precision, matching the performance of the competing approach from Evans \emph{et al.} on a unidimensional example. \cblue{An efficient implementation for coping with higher-dimension problems is proposed, based on power diagrams.}
Furthermore, we provided theoretical guarantees of the iterative computation of the bounds, and a quality control on their values. There are several avenues for future improvements. One potential direction is to explore other proposal families in the numerical application, such as Student's t     distributions, \cblue{which will require to extend the calculations of Appendix~B}. \cblue{We also plan to explore local majorant/minorant constructions \cite{ChouzenouxLocal}, to handle non log-concave and possibly multimodal targets.} %Another, more challenging, avenue is to generalize the method to approximate moments of high-dimensional distributions. 

\section*{Appendix A. Closed-form expressions for Gaussian and truncated-Gaussian monomial moments}
\label{app:gaussian_moments}
\input{Appendix_truncated_gaussian}

\cblue{
\section*{Appendix B. Computations of Gaussian majorant/minorant for $p(x)^2/q(x)$ from Sec.~\ref{sec_logistic}}
\label{app:bounds_p^2/q}
\input{Appendix_bounds_p2q}

}

\cblue{
\section*{Appendix C. Technical lemma for Sec.~\ref{sec_convergence_analysis}}
\label{app:lemmas}
\input{Appendix_useful_lemmas}
}

\section*{Acknowledgments}

S.M. and E.C. acknowledge support from the European Research Council
under Starting Grant MAJORIS ERC-2019-STG-850925.

\bibliography{biblio}

\end{document}

%% file: 01_Problem_Statement.tex
\subsection{Introduction}

\cblue{The approximation of intractable integrals has long been} a fundamental problem in mathematical sciences, leading to extensive research and the development of various methods \cite{Robert04,liu2001monte,krylov2006approximate}. Classical approaches include Newton-Cotes formulas, such as the trapezoidal and Simpson’s rules, which approximate the integrand using polynomial interpolation over equally spaced points \cite{DavisRabinowitz1984}. While these methods are easy to implement, they can have low accuracy for highly oscillatory functions due to Runge’s phenomenon \cite{Demailly2006}. \cblue{A more advanced approach relies on Gaussian quadrature, which selects integration nodes and weights so that the rule is exact for polynomials up to the highest possible degree given a fixed number of function evaluations \cite{brass2011quadrature,elvira2020importance}.} MATLAB’s integrator \cite{MATLAB_integral}, for example, extends this idea using the Gauss-Kronrod quadrature rule, which enhances accuracy by adding additional nodes to an existing Gaussian quadrature formula. For integrals over infinite domains or with singularities, adaptive quadrature methods dynamically refine the integration mesh based on local error estimates, improving efficiency and precision \cite{KahanerMolertNash1989}.

When dealing with multi-dimensional settings, traditional quadrature methods rapidly become impractical, and Monte Carlo methods provide an alternative approach by estimating integrals through random sampling \cite{Robert04,mcbook,liu2001monte}. 
\cblue{Markov chain Monte Carlo (MCMC) methods provide another widely used approach, generating dependent samples from the target distribution through a Markov chain mechanism \cite{Robert04}.}
\cblue{Importance sampling methods, an important family of Monte Carlo algorithms, improve efficiency by favoring regions that contribute most to the integral's value \cite{Robert04,elvira2021advances}. Hybrid approaches combining sampling and quadrature, such as importance Gaussian quadrature \cite{elvira2020importance}, have also been proposed to improve efficiency.}
Another powerful approach in high-dimensional integration problems is sparse grid integration, which reduces the curse of dimensionality by combining lower-order quadrature rules adaptively \cite{NovakRitter1996}. \cblue{Quasi-Monte Carlo methods and number-theoretic constructions further improve integration accuracy by using low-discrepancy deterministic point sets, reducing variance compared to standard Monte Carlo approaches \cite{caflisch1998monte,naor2004number}.} Each of these methods has its advantages and limitations, making their selection highly dependent on the problem at hand. A more detailed review of these numerical integration techniques can be found in standard numerical analysis texts, such as the reference book \cite{EvansSwartz2000}, for instance.

\cblue{An important class of integration problems is the computation of a one-dimensional moment of a given target distribution.} Beyond that, the assessment of moment-based methods often require the evaluation of the variance of the computed estimator, to evaluate the method's efficiency. One widely used diagnostic in this context is the Effective Sample Size (ESS), which provides a quantitative measure of how well the weighted samples represent the target distribution. The ESS is relevant in importance sampling and sequential Monte Carlo methods, where weight degeneracy can significantly affect estimator quality \cite{martino2017effective,elvira2022rethinking}. Other techniques include variance reduction strategies such as control variates, stratified sampling, and resampling methods, which aim to improve estimator stability and convergence \cite{owen2000safe,elvira2019generalized}. 
\cblue{The variance of the estimator is often expressed as an integral that is as difficult as the original problem. Moreover, lower and upper bounds for the variance also involve intractable integrals.} However, most state-of-the-art integration methods come with significant approximation errors, and do not guarantee the preservation of inequalities. This paper tackles this challenge by providing an original sequential approach to compute tight bounds for moment integrals. \cblue{In contrast to standard numerical integration methods that aim at producing accurate point estimates, our objective is to construct certified upper and lower bounds that hold deterministically for a given computational budget. This distinction is particularly important in settings where reliability is critical. For instance, in importance sampling, the variance of the estimator is itself an intractable integral, and its Monte Carlo estimation is well known to be unstable and dominated by rare events (see for instance \cite{koehler2009assessment,Owen13}). As a result, point estimates of performance may lack reliable uncertainty quantification. In contrast, deterministic bounds provide explicit guarantees that can be directly used to assess estimator quality \cite{agarwal2022principled}, guide the tuning of proposal distributions \cite{bugallo2017adaptive,elvira2022optimized}, or determine the number of samples required to achieve a target level of accuracy \cite{elvira2021performance}.}

%Since direct sampling from these distributions is challenging, standard Monte Carlo techniques are inefficient. Importance sampling, on the other hand, allows for variance reduction by drawing samples from an auxiliary distribution that is easier to manipulate while correcting the introduced bias with appropriate weighting \cite{Robert04}.

\subsection{Problem Statement}

Let $f: \mathbb{R} \mapsto \mathbb{R}$ and $\phi: \mathbb{R} \mapsto \mathbb{R}$, differentiable. We consider in this paper the following moment integral:
\begin{equation}
\label{integral_interest}
\Ic \equiv \int_{-\infty}^{+\infty} f(x) \pi(x) dx, \quad \text{with} \quad \pi(x) \equiv \exp(-\phi(x)).
\end{equation}
No normalization assumption is made on $\pi$, that is $Z \equiv \int_{-\infty}^{+\infty} \pi(x) dx$ might be different from one. For most choices of $f$ and $\phi$, $\Ic$ does not have a closed-form expression and should be approximated by a numerical method. The goal of this paper is to provide an efficient iterative strategy for computing tight lower and upper bounds of $\Ic$, that we define as
\begin{equation}
\underline{\Ic} \leq \Ic \leq \overline{\Ic}. \label{eq:boundsgene}
\end{equation}

\cblue{Part of our theoretical results in Section~\ref{sec_convergence_analysis} are made under the assumption that}
function $f$ is a monomial function, i.e., 
\begin{equation}
    (\forall x \in \mathbb{R}) \quad f(x) = x^k, \, \text{for some} \; k \in  \mathbb{N}. \label{eq:monomial}
\end{equation}
If $k=0$, the problem amounts to computing the normalization constant of $\pi$, since $\mathcal{I} = \mathcal{Z}$. \cblue{The extension of the results to more general test functions $f$ is discussed in the Remark \ref{remarkmonomial}, at the end of Section~\ref{sec_convergence_analysis}}. \cblue{The extension to the multi-dimensional setting, with $\x \in \mathbb{R}^d$, $f: \mathbb{R}^d \mapsto \mathbb{R}$ and $\pi: \mathbb{R}^d \mapsto \mathbb{R}$ will be presented in Section \ref{sec_multidim}.} Throughout the paper, we assume that the problem is well-posed, that is $f$ and $\pi$ are such that $\mathcal{I} \in \mathbb{R}$. 

\subsection{Paper Contributions and Outline}

This paper presents a new sequential method to compute the bounds $(\underline{\Ic},\overline{\Ic})$ in \eqref{eq:boundsgene}
with controlled precision. Our approach first stands on the construction of piecewise Gaussian upper and lower approximations for $\pi$, by relying on the powerful principle of majoration-minimization \cite{chouzenoux:hal-00789962,chouzenoux:hal-04250055,Sun2017}. We then design a recursive method to compute upper and lower envelopes, from the set of piecewise approximations. Our approach hence yields the construction of two sequences of upper and lower bounds for the integral of interest \eqref{integral_interest}, with increasing and controlled accuracy, and are both proved to converge to the exact integral value $\Ic$. Although integral $\mathcal{I}$ could be computed, in special cases of $\pi$, using known results such as the moments of truncated Gaussian distributions, our method applies more generally to any unnormalized distribution satisfying the assumptions detailed in Section~\ref{sec_compound}. We then specialize our method to the computation of bounds on the variance of an importance sampling estimator. Finally, we illustrate the practical effectiveness of our method, and compare it to the state-of-the-art envelope-based integration technique from \cite[Sec.5.4]{EvansSwartz2000}, in an application of statistical inference, namely a Bayesian classification problem. \cblue{The extension of the results to the multi-dimensional case is finally explored, relying on the properties of convex polyhedral cells and power diagrams.}
%through an application in . Specifically, we study the computation of the bounds for the variance of an importance sampling estimator of the second-order moment of the posterior distribution arising from a penalized logistic regression problem. 

The outline of the paper is as follows. In Section \ref{sec_gaussbounds}, we introduce our construction of upper and lower Gaussian approximations for $\pi$, under reasonable assumptions on $\phi$. We deduce a preliminary formula of rough lower and upper bounds $\underline{\Ic}$ and $\overline{\Ic}$ for $\Ic$. Then, in Section \ref{sec_compound}, we propose an iterative algorithm which aims at producing $(\underline{\Ic}^{(n)})_{n \in \bN}$ and $(\overline{\Ic}^{(n)})_{n \in \bN}$, that are sequences of refined lower and upper bounds for the integral of interest, $\Ic$. We prove in Section \ref{sec_convergence_analysis} that both sequences converge to $\Ic$ as $n$ approaches infinity, and we provide a quantitative control of the error along iterations. \cblue{In Section \ref{sec_IS}, we formulate the problem of variance estimation for an importance sampling estimator and the specialization of our method to this case.} In Section \ref{sec_logistic}, we illustrate the performance of our method, and compare it to the state-of-the-art, in the use-case context motivated in Sec.~\ref{sec_IS}. \cblue{In Section \ref{sec_multidim}, we present the formulation and implementation of our approach for the multi-dimensional case, and further illustrate its performance on a 2D problem completing the scenario of Section~\ref{sec_logistic}.}

%We will also compare our results with the method that is most similar to ours in the state of the art, as presented by Michael Evans and Tim Swartz \cite{EvansSwartz2000}. Indeed, our method is not directly rely on numerical quadrature but instead leverages upper and lower bounding techniques. Specifically, our method shares similarities with the envelope-based techniques described in Section 5.4 of \cite{EvansSwartz2000}, which is why, in Section \ref{sec_logistic}, we compare our approach with the one presented in their work.

%% file: 02a_intro_gaussbound.tex
\subsection{Introduction}

% Let us denote
% \begin{equation}
% (\forall x \in \bR) \quad \pi(x) \equiv \exp(-\phi(x)).
% \end{equation}
Let us first define the notion of tangent majorant (resp. minorant) functions \cite{chouzenoux:hal-04250055}. Let $t \in \bR$. Function $x \mapsto \overline{b}(x,t)$ is said to be a \emph{tangent minorant} to function $\pi$ at $t$ if %\victor{First, I would call it $t_0$ instead of $t$ (to not give the impression that it is an independent variable). Second, I think it would be easier if we denoted $\underline{b}_{t_0}(x)$ instead of $\underline{b}(x,t_0)$, to not give the impression that it is a function of two independent variables ($t_0$ is a parameter, right?)}
%\emilie{OK for $t_0$, i will change. But the notation, with the 2 entries for $b$ is standard, in all the works on majorizing approximations, so i do not want to modify it}

\begin{equation}
(\forall x \in \bR) \quad \pi(x) \geq \underline{b}(x;t) \quad \text{and} \quad \pi(t) = \underline{b}(t;t).
\label{eq_minor}
\end{equation}
Respectively, function $x \mapsto \overline{b}(x;t)$ is said to be a \emph{tangent majorant} to function $\pi$ at $t$ if the following conditions hold:
\begin{equation}
(\forall x \in \bR) \quad \pi(x) \leq \overline{b}(x;t) \quad \text{and} \quad \pi(t) = \overline{b}(t;t).
\label{eq_major}
\end{equation}
In this section, we present how to construct tangent minorant (resp. majorant) functions for $\pi$, under structural assumptions on $\phi$. The approximations take the form of unnormalized Gaussian densities. From these, we deduce a simple strategy to derive rough lower and upper bounds for integral $\Ic$ in \eqref{integral_interest}, when $f$ is the monomial given by \eqref{eq:monomial}.

In the following, we will make use of the following shorter notation, for the Gaussian probability density function (pdf):
\begin{equation}
(\forall x \in \bR) \quad g(x;\mu,\sigma) \equiv \frac{1}{\sqrt{2 \pi \sigma^2}}\exp\left(-(x - \mu)^2/(2 \sigma^2) \right),
\end{equation}
with mean $\mu \in \bR$ and standard deviation $\sigma > \bR^+$. 

%% file: 02b_minorant_gaussbound.tex
\subsection{Gaussian tangent minorants}
In order to build a tangent minorant with Gaussian shape for $\pi$, we rely on the following assumption:

\begin{assumption}
\label{assump1}
For every $t \in \mathbb{R}$, there exists $\beta(t) > 0$ such that 
\begin{equation}
(\forall x \in \mathbb{R}) \quad \phi(x) \leq \phi(t) + \dot{\phi}(t) (x-t) + \frac{\beta(t)}{2} (x-t)^2,
\end{equation}
where $\dot{\phi}$ denotes the first order derivative of $\phi$. 
\end{assumption}

Assumption \ref{assump1} is satisfied by a wide class of differentiable functions. For instance, assume that $\phi$ is convex (i.e, $\pi$ is log-concave), twice differentiable on $\bR$ with bounded second order derivative $\ddot{\phi}$ such that $\max_{x \in \bR} \ddot{\phi}(x) = \beta_{\max} > 0$. Then, Assumption~\ref{assump1} holds with $\beta(t) = \beta_{\max}$ for every $t \in \bR$ \cite{Bauschke2017}. More generally, by the descent lemma, Assumption~\ref{assump1} is satisfied for any $L$-Lipschitz differentiable $\phi$ as soon as $\beta(t) \geq L$ for every $t \in \bR$ \cite{Bauschke2017}. Examples of functions $\phi$ fulfilling Assumption~\ref{assump1} are listed in Tab.~\ref{tab:Assumption1}. Other examples can be found for instance in \cite{Sun2017,chouzenoux:hal-00789962}. 

Let us emphasize that Assumption~\ref{assump1} is stable by summation: if $(\phi_j)_{1 \leq j \leq J}$ satisfy Assumption~\ref{assump1} for some $(\beta_j(\cdot))_{1 \leq j \leq J}$, then $\sum_{j=1}^J \phi_j$ satisfies Assumption~\ref{assump1} with $t \mapsto \beta(t) = \sum_{j=1}^J \beta_j(t)$. We will make use of this property to build our final approximation algorithm.

\begin{table}
\renewcommand*{\arraystretch}{1.5}
\centering
\begin{tabular}{|c||c|c|c|}
\hline
{Type} & $x \mapsto \phi(x)$ & $x \mapsto\dot{\phi}(x)$ & $t \mapsto \beta(t)$ \\
\hline
\hline
Quadratic & $\frac{1}{2} x^2$ & $x $ & $1$\\
\hline
Hyperbolic & $(1 + x^2/\delta^2)^{1/2}$ & $(x/\delta^2) (1 + x^2/\delta^2)^{-3/2}$ & $\delta^{-2} (1 + t^2/\delta^2)^{-3/2}$\\
\hline                           
Huber & $\begin{cases} x^2, & \text{if} \;|x|<\delta,\\ 2 \delta |x|- \delta^2, & \text{otherwise} \end{cases}$ & 
$\begin{cases} 2x, & \text{if} \;|x|<\delta,\\ 2 \delta \text{sign}(x), & \text{otherwise}  \end{cases}$ & 
$\begin{cases} 2, & \text{if} \;|t|<\delta,\\ 2 \delta |t|^{-1}, & \text{otherwise}  \end{cases}$\\
\hline                                                                
Logistic & $\log( 1 + e^x)$ & $e^x / (1 + e^x)$ & $\begin{cases} 1/4, & \text{if} \; t = 0,\\ \frac{1}{t}(\frac{1}{1 + e^{-t}}-\frac{1}{2}), & \text{elsewhere} \end{cases} $\\
\hline
Cauchy & $\log(1 + x^2/\delta^2)$ & $2 x  (x^2 + \delta^2)^{-1}$ & $2 (t^2 + \delta^2)^{-1}$\\
\hline
\end{tabular}
\caption{Examples of functions, defined on $\mathbb{R}$, satisfying Assumption \ref{assump1}. We provide the expression of $\phi$, of its derivative $\dot{\phi}$ and of the associated majorant curvature $\beta$. We assume that $\delta>0$.}
\label{tab:Assumption1}
\end{table}

\begin{lemma}
\label{lem:ass1}
Under Assumption~\ref{assump1}, the minoration conditions \eqref{eq_minor} hold with
\begin{equation}
(\forall (x,t) \in \bR^2) \quad \underline{b}(x;t) = 
\underline{C}(t) g(x; \underline{u}(t), \underline{\sigma}(t))
\end{equation}
using the following variance, mean and scale parameters:
\begin{equation}
(\forall t \in \bR) \quad
\begin{cases}
\underline{\sigma}(t) & = 1/\sqrt{\beta(t)}, \\
\underline{u}(t) & = t -\underline{\sigma}(t)^2 \dot{\phi}(t), \\
\underline{C}(t) & = \sqrt{2 \pi \underline{\sigma}(t)^2} \exp\left(- \phi(t) + (\underline{\sigma}(t)^2/2) (\dot{\phi}(t))^2\right). \\
\end{cases}
\end{equation}
\end{lemma}

\begin{proof}  
Let $t \mapsto \beta(t)$ defined as in Assumption \ref{assump1}. We have:
\begin{align}
(\forall (x,t) \in \bR^2) \quad \phi(t) &+ \dot{\phi}(t) (x - t) + \frac{\beta(t)}{2} (x - t)^2 \\
& = \frac{\beta(t)}{2} x^2 + (\dot{\phi}(t) - \beta t) x + \phi(t) - t \dot{\phi}(t) + \frac{\beta(t)}{2} (t)^2\\
%& = \frac{\beta(t)}{2} \left(x^2 - (2 t - \frac{2}{\beta(t)} \dot{\phi}(t)) x\right) + \phi(t) - t \dot{\phi}(t) + \frac{\beta(t)}{2} (t)^2\\
%& =\frac{\beta}{2} \left(x - (t - \frac{1}{\beta(t)} \dot{\phi}(t)) \right)^2 + \nonumber \\
%& \qquad  \phi(t) - t \dot{\phi}(t) + \frac{\beta(t)}{2} (t)^2 - \frac{\beta(t)}{2} (t - \frac{1}{\beta(t)} \dot{\phi}(t) )^2\\
& = \frac{\beta(t)}{2} \left(x - (t - \frac{1}{\beta(t)} \dot{\phi}(t)) \right)^2 + \phi(t) - \frac{1}{2 \beta(t)} (\dot{\phi}(t))^2
\end{align}
Thus, according to Assumption \ref{assump1}, for every $t \in \bR$,
\begin{equation}
(\forall x \in \bR) \quad
\phi(x) \leq \frac{\beta(t)}{2} (x - u(t))^2 + \phi(t) - \frac{1}{2 \beta(t)} (\dot{\phi}(t))^2
\end{equation}
where we set
\begin{equation}
(\forall t \in \bR) \quad u(t) = t - \beta(t)^{-1} \dot{\phi}(t).
\end{equation}
Hence the result, using $\phi =- \log \pi$.
\end{proof}

%% file: 02c_majorant_gaussbound.tex
\subsection{Gaussian tangent majorants}

The majoration step requires an additional assumption on $\phi$, as follows: 
\begin{assumption}
\label{assump2}
For every $t \in \mathbb{R}$, there exists $\nu(t)>0$ such that
\begin{equation}
(\forall x \in \mathbb{R}) \quad \phi(x) \geq \phi(t) + \dot{\phi}(t) (x-t) + \frac{\nu(t)}{2} (x-t)^2
\end{equation}
\end{assumption}

Assumption~\ref{assump2} is related to the notion of \emph{strong convexity} (also sometimes called \emph{uniform convexity}) \cite{Bauschke2017}. It holds for instance for twice differentiable convex functions $\phi$ such that $\min_{x \in \bR} \ddot{\phi}(x) = \nu > 0$, by setting $\nu(t) = \nu$ for every $t \in \bR$. It also holds for $\phi: x \to \tilde{\phi}(x) + a x^2$, with any $a>0$, and $\tilde{\phi}$ convex on $\bR$. More generally, if $(\phi_j)_{1 \leq j \leq J}$ are convex functions, and for some $j \in \left\{1,\ldots,J\right\}$, $\phi_j$ 
satisfies Assumption \ref{assump2} with $\nu_j(\cdot)$, then $\sum_{j=1}^J \phi_j$ also does, with $t \mapsto \nu(t) = \nu_j(t)$. 

%%Other examples are listed in Tab.~\ref{tab:Assumption2}. 
%%\begin{table}
%%\renewcommand*{\arraystretch}{1.5}
%%\centering
%%\begin{tabular}{|c||c|c|c|}
%%\hline
%%Name & $\phi(x)$ & $\dot{\phi}(x)$ & $\nu(t)$ \\
%%\hline
%%\hline
%%Quadratic & $\frac{1}{2} x^2$ & $x $ & $1$\\
%%\hline
%%\end{tabular}
%%\caption{Examples of functions satisfying \ref{assump2}. We provide the expression of $\phi$, of its derivative $\dot{\phi}$ and of the associated $t \mapsto \nu(t)$.}
%%\label{tab:Assumption2}
%%\end{table}

\begin{lemma}
\label{lem:ass2}
Under Assumption~\ref{assump2}, the majoration conditions \eqref{eq_major} are satisfied, as soon as
\begin{equation}
(\forall (x,t) \in \bR^2) \quad
\overline{b}(x;t)    = \overline{C}(t) g(x; \overline{u}(t), \overline{\sigma}(t)),
\end{equation}
with variance, mean and scale parameters given by
\begin{equation}
\label{eq:uppermean_variance_constant_gauss}
(\forall t \in \bR) \quad
\begin{cases}
\overline{\sigma}(t) & = 1/\sqrt{\nu(t)},\\
\overline{u}(t) & = t -\overline{\sigma}(t)^2 \dot{\phi}(t), \\
\overline{C}(t) & = \sqrt{2 \pi \overline{\sigma}(t)^2} \exp\left(- \phi(t) + (\overline{\sigma}(t)^2/2) (\dot{\phi}(t))^2\right). \\
\end{cases}
\end{equation}
\end{lemma}

\proof{The proof follows similar structure than the one of Lemma~\ref{lem:ass1}, reverting inequalities.}

%% file: 02d_discussion_gaussbound.tex
\subsection{First bounds and discussion}

According to what precedes, we are able to construct, for every $t \in \bR$, unnormalized Gaussian densities $\overline{b}(\cdot;t)$ and $\underline{b}(\cdot;t)$ that are tangent majorant and minorant functions for $\pi$ at $t$ i.e.,
\begin{equation}
(\forall x \in \bR) \quad \underline{b}(x;t) \leq \pi(x) \leq \overline{b}(x;t).
% \quad \text{and} \quad \underline{b}(t,t) = \pi(t) = \overline{b}(t,t)
\label{eq_majmin0}
\end{equation}
An example is displayed in Fig.~\ref{fig:minmaj}, where the target function is displayed in blue line, the minorant functions are in black thin lines (left plot) and the majorant functions are in red thin lines (right plot).
\begin{figure}[h]
\centering
\begin{tabular}{@{}c@{}c@{}}
\includegraphics[width = 5cm]{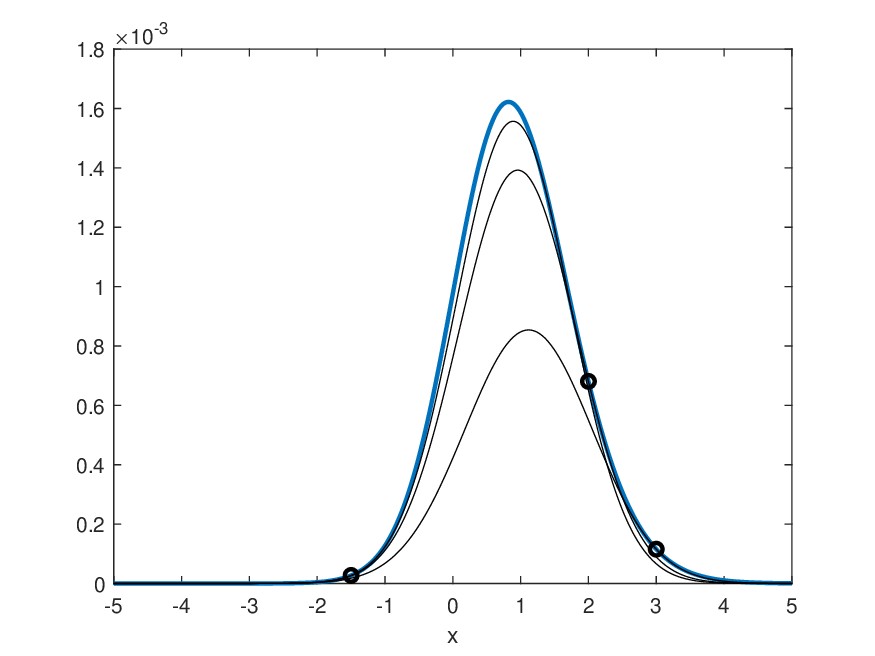}
&
\includegraphics[width = 5cm]{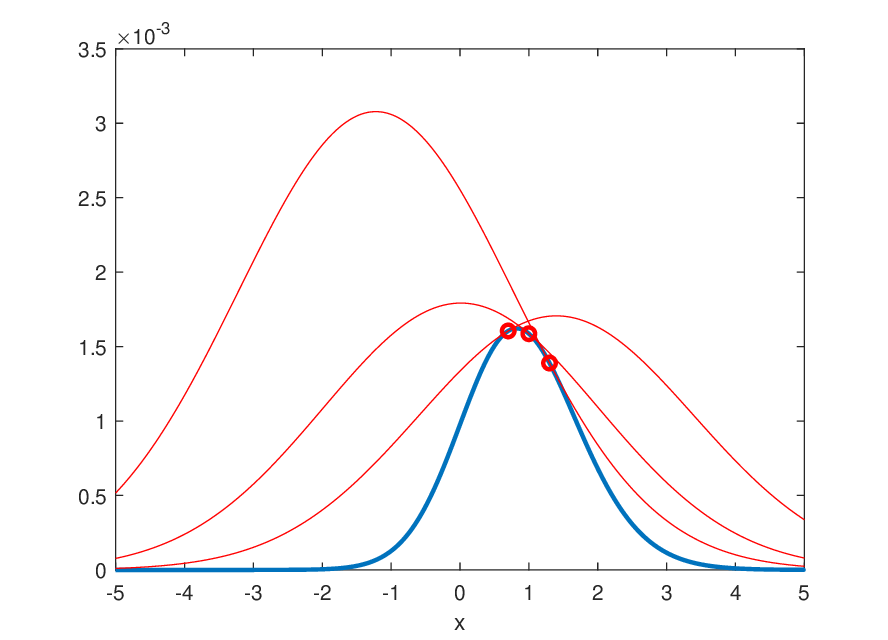}
\end{tabular}
\caption{Example of function $\pi$ (blue thick line) and minorant/majorant functions (black/red thin lines) at various tangency points (black/red circles)}
\label{fig:minmaj}
\end{figure}
 
We define the \emph{positive} and \emph{negative} parts of a function $f$ by
\begin{equation}
\forall x \in \mathbb{R}, \quad f^+(x) = \max\{f(x), 0\}, \quad f^-(x) = \max\{-f(x), 0\}.
\end{equation}
With these definitions, we have the following identities:
\begin{equation}
\forall x \in \mathbb{R}, \quad f(x) = f^+(x) - f^-(x), \quad |f(x)| = f^+(x) + f^-(x),
\end{equation}
\cblue{respectively.}
 %\emilie{Solal, please put a more formal, mathematical definition for these notations, using max, min symbols.}
Let us denote, for every $t \in \mathbb{R}$,
\begin{align}
\underline{\Ic}(t) & = \int_{\mathbb{R}} f^+(x) \, \underline{b}(x;t) \, dx - \int_{\mathbb{R}} f^-(x) \, \overline{b}(x;t) \, dx \nonumber \\
& = \underline{C}(t) \int_{\mathbb{R}}  f^+(x) \, g(x; \underline{u}(t), \underline{\sigma}(t)) \, dx - \overline{C}(t) \int_{\mathbb{R}}  f^-(x) \, g(x; \overline{u}(t), \overline{\sigma}(t)) \, dx,
\end{align}
and
\begin{align}
\overline{\Ic}(t) & = \int_{\mathbb{R}} f^+(x) \, \overline{b}(x;t) \, dx - \int_{\mathbb{R}} f^-(x) \, \underline{b}(x;t) \, dx \nonumber \\
& = \overline{C}(t) \int_{\mathbb{R}} f^+(x) \, g(x; \overline{u}(t), \overline{\sigma}(t)) \, dx - \underline{C}(t) \int_{\mathbb{R}} f^-(x) \, g(x; \underline{u}(t), \underline{\sigma}(t)) \, dx.
\end{align}

We deduce from \eqref{integral_interest}, and \eqref{eq_majmin0} the following bounds,
\begin{equation}
\label{eq:bound_I}
(\forall t \in \bR) \quad \underline{\Ic}(t)\leq \Ic \leq \overline{\Ic}(t).
\end{equation}
Hence, the following inequality holds:
\begin{equation}
\max_{t \in \bR} \underline{\Ic}(t) \leq \Ic \leq \min_{t \in \bR} \overline{\Ic}(t).
\label{eq:bound1}
\end{equation}

\cblue{When the test function $f$ reads as a monomial, as in \eqref{eq:monomial}, or as a sum of monomials}, the terms $\underline{\Ic}(t)$ and $\overline{\Ic}(t)$ for every $t \in \mathbb{R}$, involve integrals that can be read as moments of (possibly truncated) Gaussian distributions. Such moments are well-defined, and can be computed easily, with high precision, based on recursive relations relying on the error function \texttt{erf}. These expressions are recalled in Appendix~A.

Still, computing the optimal bounds in \eqref{eq:bound1} requires solving two optimization problems over the set $t \in \mathbb{R}$, whose resolution may not be straightforward. Relying on the preliminary results established in Lemmas~\ref{lem:ass1} and~\ref{lem:ass2}, we propose in the remainder of the paper an iterative approximation strategy to refine the bounds, featuring two main ingredients: (i) a recursive and adaptive selection of a finite set of representative tangency points, and (ii) the computation of integral bounds using piecewise-Gaussian upper and lower approximations of $\pi$, delimited by the retained tangency points.

Thanks to this approach presented in Section \ref{sec_compound}, no explicit optimization is needed, and all numerical computations are straightforward, under the mild assumption that \cblue{truncated Gaussian moments associated to $f$}, are computable.% which, in our case, always holds by construction.

%\solal{I rewrite the previous paragraph like this and I also write an appendix for the computation of the moments of the gaussian truncated moment, it's better ?}

%% file: 03a_piecewisebound_compound.tex
\subsection{Introduction}

In Section \ref{sec_gaussbounds}, we have derived a method to construct tangent minorant and majorant functions for $\pi$, for any given value of tangency point $t \in \bR$. As we can observe in Figure~\ref{fig:minmaj}, the accuracy of such approximations highly depends on the distance to the tangency point. When $x = t$, the approximations are exact (because of the tangency property), and, as $x$ gets further from $t$, a degradation in terms of quality of approximation can be expected. We propose an improved approximation strategy, described in Figure~\ref{fig:compound}: Given a list of $M \geq 2$ tangency points, we construct the associated tangent minorant (thin black lines) and majorant (thin red lines) functions for $\pi$ (blue line) at these points. Then, from this set of $M$ functions, improved minorant (thick blue line) and majorant (thick red line) approximations of $\pi$ are deduced. These envelopes are defined in a piecewise manner, as the upper (resp. lower) bound of the available tangent minorant (resp. majorant) functions. This idea is at the core of the proposed Algorithm \ref{algo_compund}. In order to improve again the performance of the method, we then propose an iterative selection of the tangency points, through our Algorithm 2. By construction, such approach guarantees an increasing precision of the bounds along its iterations, and, as we will show in the next section, it produces a sequence of bounds that asymptotical converge to the sought integral value.

%for integral approximation. 
%\victor{The red sentence is a bit fuzzy. I think it is much easier to first introduce the nice Fig. 2, and then explain on the figures.}
\begin{figure}[H]
\centering
\begin{tabular}{@{}c@{}c@{}}
\includegraphics[width = 5cm]{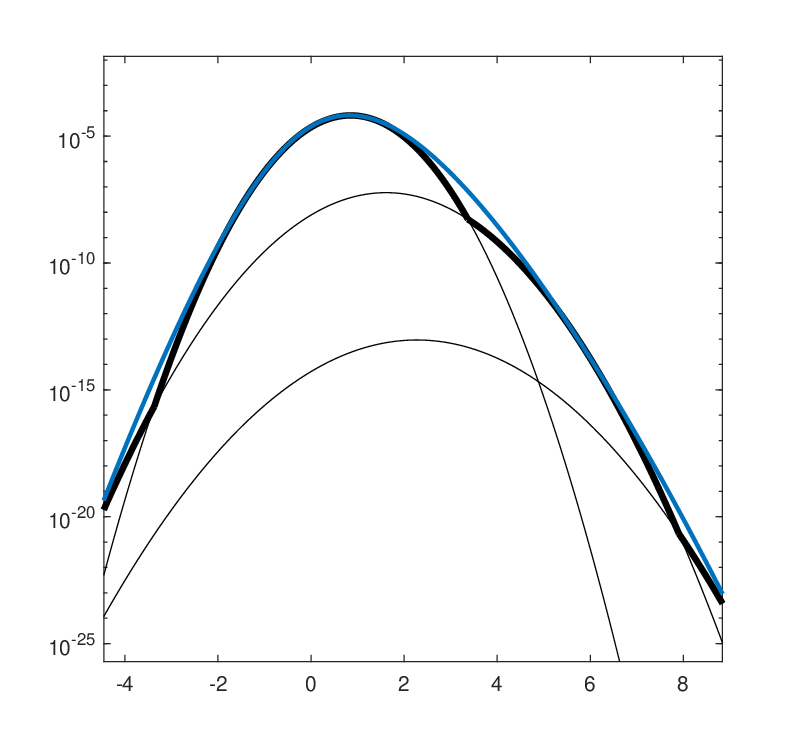}
&
\includegraphics[width = 5cm]{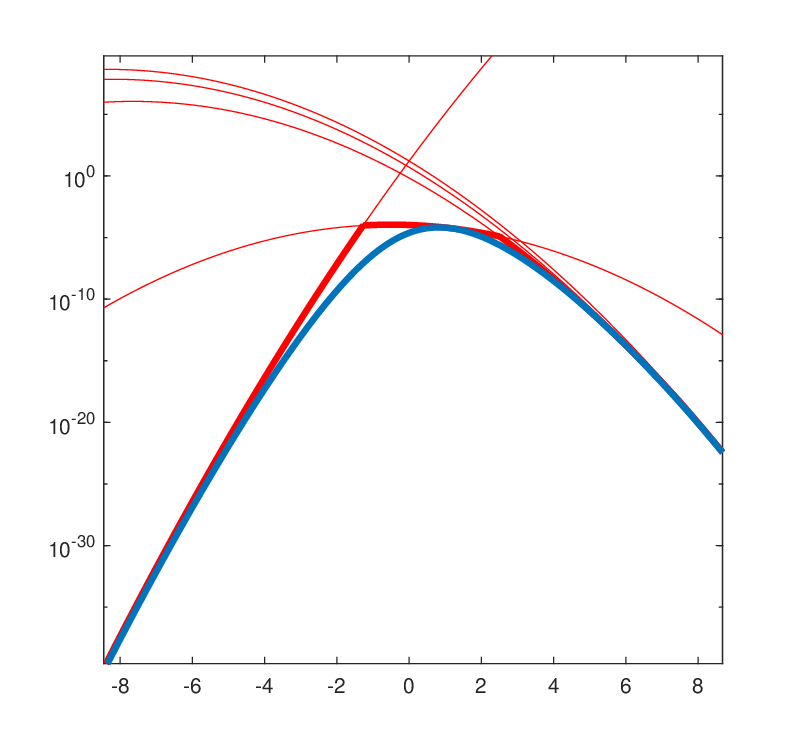}
\end{tabular}
\caption{Example of function $\pi$ (blue thick line), set of tangent minorant/majorant functions (black/red thin lines) at various tangency points, and associated piecewise minorant/majorant function (black/red thick line).}
\label{fig:compound}
\end{figure}

%\emilie{to complete: definition and sequential construction}

\subsection{Minorant/majorant envelopes}

Let us introduce the notation that will be required below. We set $M \geq 1$ tangency points, and $\mathbf{T} = \left[t_1, \cdots, t_m, \cdots t_M\right] \in \mathbb{R}^M$ as the vector that contains the $M$ values of the tangency points, sorted in increasing order. 
%%We denote   
%%\begin{equation}
%%\mathcal{\underline{B}}(\mathbf{T}) = \left\{\underline{b}(\cdot;t_m)\right\}_{1 \leq m \leq M},
%%\end{equation}
%%the set of $M$ tangent minorant functions with tangency points $(t_m)_{1 \leq m \leq M}$, and
%%\begin{equation}
%%\mathcal{\overline{B}}(\mathbf{T}) = \left\{\overline{b}(\cdot;t_m)\right\}_{1 \leq m \leq M},
%%\end{equation}
%%the set of $M$ tangent majorant functions with tangency points $(t_m)_{1 \leq m \leq M}$. 
The piecewise minorant function $\overline{C}(\cdot;\mathbf{T})$ and piecewise majorant function $\underline{C}(\cdot;\mathbf{T})$, associated to $\mathbf{T}$ are defined by
\begin{equation}
(\forall x \in \bR) \quad \underline{C}(x;\mathbf{T}) = \max_{m \in \left\{1,\ldots,M\right\}} \underline{b}(x;t_m),
\label{eq:compound_min}
\end{equation}
with $\left\{\underline{b}(\cdot;t_m)\right\}_{1 \leq m \leq M}$ the set of $M$ tangent minorant functions with tangency points $(t_m)_{1 \leq m \leq M}$, and
\begin{equation}
(\forall x \in \bR) \quad \overline{C}(x;\mathbf{T}) = \min_{m \in \left\{1,\ldots,M\right\}} \overline{b}(x;t_m),
\label{eq:compound_maj}
\end{equation}
with $\left\{\overline{b}(\cdot;t_m)\right\}_{1 \leq m \leq M}$ the set of $M$ tangent majorant functions with tangency points $(t_m)_{1 \leq m \leq M}$. Therefore, by construction,
\begin{equation}
(\forall x \in \bR) \quad \underline{C}(x;\mathbf{T}) \leq \pi(x) \leq \overline{C}(x;\mathbf{T}). \label{eq:enveloppepi}
\end{equation}
Note that, due to the tangency condition, and the optimality conditions \eqref{eq:compound_min}-\eqref{eq:compound_maj}, the function and the piecewise majorant and minorant functions take the same value at each tangent point, i.e., 
\begin{equation}
(\forall m \in \left\{1,\ldots,M\right\}) \quad \overline{C}(t_m;\mathbf{T}) = \underline{C}(t_m;\mathbf{T}) = \pi(t_m).
\end{equation}
%
%\emilie{Is it obvious or should we prove it?}
%\victor{It is obvious but also good to leave it (without proving it, I would say.)}

The optimization problems \eqref{eq:compound_min} and \eqref{eq:compound_maj} are known in the literature of combinatorial algorithmics as the search for \emph{upper enveloppe} and \emph{lower enveloppe} of a collection of univariate functions. Such problems can be solved in an efficient manner, using \cblue{the recursive algorithm from \cite{AGARWAL20001}, relying on Davenport–Schinzel Sequences}. We now describe this computational algorithm applied to our specific problem, i.e., with majorant/minorant functions that take the form of unnormalized Gaussian densities. 

Let us focus on problem~\eqref{eq:compound_min}, which amounts to determining the upper envelope of the set of minorant functions. We introduce the short notation $\psi_{t_m}(x) \triangleq \underline{b}(x; t_m)$ the minorant built from the tangency point $t_m$, for every $m \in \{1, \ldots, M\}$. %In other words, each $\psi_{t_m}$ corresponds to the lower bound evaluated at the point $t_m$.
The upper envelope is obtained by calling the recursive Algorithm~\ref{algo_compund}, starting with input $\mathbf{T} = [t_1, \dots,  t_M]$ and the associated family $\{\psi_{t_m}\}_{1 \leq m \leq M}$. \cblue{The algorithm returns $\mathbf{v} = (v_1, \dots, v_i, \ldots, v_{N})$, the ordered sequence of indices realizing the upper envelope in \eqref{eq:compound_min} on consecutive maximal intervals. The integer \(N\) is exactly the number of such intervals, hence the number of \emph{breakpoints} (i.e., points where the maximizer changes) is $N-1$. The tangent points associated to each breakpoints are gathered in $\u \in \mathbb{R}^{N-1}$. The envelope can be written in piecewise form once the breakpoints are known. Then the solution of the maximization problem \eqref{eq:compound_min} is given explicitly by
\[
(\forall x \in \mathbb{R}) \quad \underline{C}(x;\T) = \psi_{u_i}(x), \quad x \in [v_{i}, v_{i+1}],
\]
for \(i=1,\dots,N-1\) (note that we authorize $v_1 = - \infty$ and $v_{N} = + \infty$). Each breakpoint \(v_i\) is a solution of an equation
\[
\psi_{u_i}(v_i) = \psi_{u_{i+1}}(v_i), \quad \text{with} \quad \psi_{u_i}(x) \ge \psi_{t_m}(x),\;\text{for every $t_m$ in a neighborhood of } v_i.
\]
Denote $s>0$ such that any pair of functions $(\psi_{t_m},\psi_{t_m'})$
intersects at most 
$s$ times. The combinatorial complexity of the constructed envelope (and thus the number of intervals and breakpoints handled by Algorithm \ref{algo_compund}) is bounded by 
$\lambda_s(M)$, the Davenport–Schinzel extremal function. This factor grows nearly linearly in $M$ for fixed $s$~\cite{AGARWAL20001}.
} %Namely, the solution of~\eqref{eq:compound_min} reads:

\cblue{
\begin{algorithm}[H]
\caption{Recursive Compound Minimization Algorithm}
\label{algo_compund}
\begin{algorithmic}[1]
\Require Vector $\mathbf{T}$ of size $L$, and functions $\{\psi_{t_i}\}_{1 \leq i \leq L}$
\Ensure $\mathbf{u}$ of size $N-1$, and $\mathbf{v}$ of size $N$
\State Partition $\mathbf{T}$ into $\mathbf{T}^{(1)}$ and $\mathbf{T}^{(2)}$ of sizes $L_1$ and $L_2$, with $L_1 + L_2 = L$ and $L_1, L_2 \leq \lceil \tfrac{L}{2} \rceil$
\If{$L_1 > 1$}
    \State Call Algorithm~\ref{algo_compund} on $\mathbf{T}^{(1)}$ and $\{\psi_{t_i}\}_{1 \leq i \leq L_1}$, returning $\mathbf{u}^{(1)}$ and $\mathbf{v}^{(1)}$
\Else
    \State Initialize $\mathbf{u}^{(1)} = (t_1)$ and $\mathbf{v}^{(1)} = (-\infty, +\infty)$
\EndIf
\If{$L_2 > 1$}
    \State Call Algorithm~\ref{algo_compund} on $\mathbf{T}^{(2)}$ and $\{\psi_{t_{L_1+i}}\}_{1 \leq i \leq L_2}$, returning $\mathbf{u}^{(2)}$ and $\mathbf{v}^{(2)}$
\Else
    \State Initialize $\mathbf{u}^{(2)} = (t_L)$ and $\mathbf{v}^{(2)} = (-\infty, +\infty)$
\EndIf
\State Merge and sort $\mathbf{v}^{(1)}$ and $\mathbf{v}^{(2)}$ into
$
\mathbf{v}' = (v'_1 < v'_2 < \cdots < v'_{N'})
$
\State Initialize an empty list $\tilde{\mathbf{v}}$
\For{$i = 1$ to $N'-1$}
    \State Find $u_i^{*(1)} \in \mathbf{u}^{(1)}$ such that:
    \begin{equation}
        (\forall x \in (v'_i, v'_{i+1}))(\forall j \in \mathbf{u}^{(1)}) \quad \psi_{u_i^{*(1)}}(x) \geq \psi_{t_j}(x) \label{eq_upper_env_1}
    \end{equation}
    \State Find $u_i^{*(2)} \in \mathbf{u}^{(2)}$ such that:
    \begin{equation}
        (\forall x \in (v'_i, v'_{i+1}))(\forall j \in \mathbf{u}^{(2)}) \quad \psi_{u_i^{*(2)}}(x) \geq \psi_{t_j}(x) \label{eq_upper_env_2}
    \end{equation}
    \State Compute all intersection points of $\psi_{u_i^{*(1)}}$ and $\psi_{u_i^{*(2)}}$ in $(v'_i, v'_{i+1})$
    \State Add these points to $\tilde{\mathbf{v}}$
\EndFor
\State Define $\mathbf{v}$ as the sorted union of $\mathbf{v}'$ and $\tilde{\mathbf{v}}$, as
$
\mathbf{v} = (v_1 < v_2 < \cdots < v_N)
$
\State Initialize $\mathbf{u} \leftarrow ()$
\For{$i = 1$ to $N-1$}
    \State Select any $x \in (v_i, v_{i+1})$
    \State Compute
    $
    u_i = \arg\max_{m \in \mathbf{u}^{(1)} \cup \mathbf{u}^{(2)}} \psi_{t_m}(x)
    $
    \State Append $u_i$ to $\mathbf{u}$
\EndFor
\end{algorithmic}
\end{algorithm}
}

%\solal{I change, the name of $f_m$ to $\psi_m$}

%\victor{Describe it in a recursive way, but batch or sequential? In batch, we have $M$ points, and I split it in two sets of $M/2$, etc. It is how it is described in that paper. In sequential, I assume that I have $M-1$ point for which I know the compounds, and the problem is how to update them (requires more thinking, haha). Regarding the code, if I remember correctly, it is always batch, and we never exploit the fact that we know the compounds for $M-1$, but I should check (do you remember?).}

%\emilie{I confirm we always used the batch version in our code.}

The lower envelope of the set of majorant functions, defined as the solution of Eq. \eqref{eq:compound_maj} can be obtained similarly as in Alg. \ref{algo_compund}. More precisely, the inequalities order should be reverted in both Eqs. \eqref{eq_upper_env_1} and \eqref{eq_upper_env_2} in lines 13 and 14 of Alg. \ref{algo_compund}, while in line 21, the argmin (instead of argmax) in each interval must be selected.

%% file: 03b_algo_compound.tex
\subsection{Integral approximation}

Once the upper and lower envelopes in \eqref{eq:enveloppepi} have been obtained, using the same reasoning that we use to find the equation \ref{eq:bound_I}, we are now able to express new bounds for the integral $\Ic$ as
\begin{equation}
\int_{\bR} f^+(x) \underline{C}(x;\mathbf{T} ) dx -  \int_{\bR} f^-(x) \overline{C}(x;\mathbf{T} ) dx\leq \Ic \leq \int_{\bR} f^+(x) \overline{C}(x;\mathbf{T} ) dx -\int_{\bR} f^-(x) \underline{C}(x;\mathbf{T} ) dx.
\label{eq:compoundbnd}
\end{equation}
If $f$ reads as in \eqref{eq:monomial}, the integrals appearing in the bounds \eqref{eq:compoundbnd} involve products of monomials with piecewise unnormalized Gaussian densities, as defined in \eqref{eq:compound_min} and \eqref{eq:compound_maj}. These integrals can be computed explicitly piecewise, provided that the moments of the (possibly truncated) Gaussian distribution are available (see our Appendix~A). 
%\emilie{Solal, again the two previous sentences must be reformulated, if we state $f(x) = x^k$ from the beginning}
%\solal{I rewrite the previous paragraph, I hope it's better. I recall that we assume f monomial, is it too much or not ?}.

\subsection{Iterative tangency point selection}
In order to refine again the precision of the bounds on $\Ic$, we now define an iterative, and adaptive, strategy for the selection of the tangency points, provided in Algorithm~\ref{algo:tangency}. At each iteration $n \in \bN^*$ of Alg.~\ref{algo:tangency}, minorant and majorant functions $\underline{C}(x;\mathbf{T}^{(n)})$ and $\overline{C}(x;\mathbf{T}^{(n)})$ are constructed, \cblue{in a piecewise manner}, for the current set of tangency points $\mathbf{T}^{(n)}$, by Alg.~\ref{algo_compund}. Lower and upper bounds, \( \underline{\Ic}^{(n)} \) and \( \overline{\Ic}^{(n)} \), for \( \Ic \), \cblue{i.e., such that $\underline{\Ic}^{(n)} \leq \Ic \leq \overline{\Ic}^{(n)}$,} are computed following the framework outlined in Eq.~\eqref{eq:compoundbnd}:

\begin{equation}
    \forall n \in \mathbb{N}^*, \quad \overline{\Ic}^{(n)} = \int_{\mathbb{R}} f^+(x) \overline{C}(x; \mathbf{T}^{(n)}) \, dx-\int_{\mathbb{R}} f^-(x) \underline{C}(x; \mathbf{T}^{(n)}) \, dx,
    \label{eq:ub_I}
\end{equation}

\begin{equation}
    \forall n \in \mathbb{N}^*, \quad \underline{\Ic}^{(n)} = \int_{\mathbb{R}} f^+(x) \underline{C}(x; \mathbf{T}^{(n)}) \, dx -\int_{\mathbb{R}} f^-(x) \overline{C}(x; \mathbf{T}^{(n)}).
    \label{eq:lb_I}
\end{equation}

The maximization of the difference between the {integral bounds provided by} the majorant and the minorant functions in all intervals between consecutive tangent points allows to determine a new tangency point, that yields the new set $\mathbf{T}^{(n+1)}$. 

\newpage
\begin{algorithm}[H]
\caption{Adaptive Tangency Point Selection}
\label{algo:tangency}
\begin{algorithmic}[1]
\Require $\mathbf{T}^{(1)} = [t_1^{(1)}, \dots, t_{M_1}^{(1)}] \in \mathbb{R}^{M_1}$ with $M_1 \geq 1$, sorted increasingly; precision value $\epsilon > 0$
\Ensure $\underline{\Ic}$ and $\overline{\Ic}$, lower and upper bounds on $\Ic$

\For{$n = 1, 2, \ldots$}
    \State Construct $\overline{C}(\cdot; \mathbf{T}^{(n)})$ and $\underline{C}(\cdot; \mathbf{T}^{(n)})$ using Alg. \ref{algo_compund}
    \State Define subintervals $S_1^{(n)} = (-\infty, t_1^{(n)}]$, $S_2^{(n)} = [t_1^{(n)}, t_2^{(n)}]$, ..., $S_{M_n+1}^{(n)} = [t_{M_n}^{(n)}, +\infty)$
    \For{$i = 1$ to $M_n+1$}
        \State Compute:
        \begin{equation}
        \begin{cases}
        \underline{\Ic}_i^{(n)} = \int_{ S_i^{(n)}} f^+(x)\, \underline{C}(x; \mathbf{T}^{(n)}) \, dx -\int_{S_i^{(n)}} f^-(x)\, \overline{C}(x; \mathbf{T}^{(n)}) \, dx \\
        \overline{\Ic}_i^{(n)} = \int_{S_i^{(n)}} f^+(x)\, \overline{C}(x; \mathbf{T}^{(n)}) \, dx - \int_{ S_i^{(n)}} f^-(x)\, \underline{C}(x; \mathbf{T}^{(n)}) \, dx
        \end{cases}
        \end{equation}
    \EndFor
    \State Compute bounds:
    \begin{equation}
    \begin{cases}
    \underline{\Ic}^{(n)} = \sum_{i=1}^{M_n+1} \underline{\Ic}_i^{(n)} \\
    \overline{\Ic}^{(n)} = \sum_{i=1}^{M_n+1} \overline{\Ic}_i^{(n)}
    \end{cases}
    \end{equation}
    \If{$\overline{\Ic}^{(n)} - \underline{\Ic}^{(n)} \leq \epsilon$}
        \State \Return $(\underline{\Ic}^{(n)}, \overline{\Ic}^{(n)})$
    \Else
        \For{$i = 1$ to $M_n+1$}
            \State Compute:
            \begin{equation}
            G_i^{(n)} = \overline{\Ic}_i^{(n)} - \underline{\Ic}_i^{(n)}
            \end{equation}
        \EndFor
        \State Select $\hat{\iota}$:
        \begin{equation}
        \hat{\iota} = \operatorname{argmax}_{i \in \left\{1,\ldots,M_n+1\right\}} G_i^{(n)}
        \label{eq:iota}
        \end{equation}
        \State Choose new point $\hat{t} \in S_{\hat{\iota}}^{(n)}$
        \State Update:
        \begin{equation}
        \mathbf{T}^{(n+1)} = [t_1^{(n)}, \dots, t_{\hat{\iota}-1}^{(n)}, \hat{t}, t_{\hat{\iota}}^{(n)}, \dots, t_{M_n}^{(n)}]
        \end{equation}
    \EndIf
\EndFor
\end{algorithmic}
\end{algorithm}

\newpage

	%\victor{Is this necessary? Discuss about these un-ordered sets of fucntions when $\mathbf{T}^{(n+1)}$ seems to be ordered/vector. }
	
A key feature in our method is the addition of a new tangency points in the intervals where the approximation gap is maximum (step 13 of Alg.~\ref{algo:tangency}). As we will show in our convergence analysis in Sec.~\ref{sec_convergence_analysis}, the implementation of step 13, coupled with the definition of $\mathbf{T}^{(1)}$, has an impact in the theoretical guarantees of the proposed approach. We provide, in Sec.~\ref{sec:algopractical}, a detailed practical solution to select new tangency points, that allows to met the convergence requirements and to reach excellent practical performance. 

% In practice, we propose the following rule, which was observed to lead to satisfying results in our experiments in terms of convergence speed.
% %
% %$Let $n \in \bN$ and $\hat{\iota}$ defined as in \eqref{eq:iota}. 
% If $\hat{\iota} \in (1,T_n + 1)$ (i.e., the interval is not at one extreme), we set $\hat{t}$ as the middle of $S_{\hat{\iota}}^{(n)}$, i.e.,
% \begin{equation}
% 	\hat{t} = \frac{t_{\hat{\iota}-1}^{(n)} + t_{\hat{\iota}}^{(n)}}{2}.
% \end{equation}
% In the case when the maximum gap is detected on the border, $S_{\hat{\iota}}^{(n)}$ is not bounded, we propose to set 
% \begin{equation}
% 	\hat{t} = t_1^{(n)} - 
% 	\frac{\sum_{m=1}^{M_n-1} t_{m+1}^{(n)} - t_m^{(n)}}{M_n-1}, \quad \text{if} \quad \hat{\iota} = 1,
% \end{equation}
% and 
% \begin{equation}
% 	\hat{t} = t_{M_n}^{(n)} + \frac{\sum_{m=1}^{M_n-1} t_{m+1}^{(n)} - t_m^{(n)}}{M_n - 1}, \quad \text{if} \quad \hat{\iota} = M_n + 1.
% \end{equation}
%The convergence of Algorithm \ref{algo:tangency} is analyzed in the next Section.
	%$$
	%\mathbf{T}^{(n+1)} = \left[\hat{t}, \; \mathbf{T}^{(n)} \right].
	%$$	
		%and
	%$$
	%\mathbf{T}^{(n+1)} = \left[\mathbf{T}^{(n)}, \; \hat{t} \right].
	%$$

%% file: 04_convergence_proofs.tex
\cblue{
 Throughout the section, we suppose that Assumptions \ref{assump1} and \ref{assump2} hold, for continuous functions $\beta(.)$ and $\nu(.)$.}
The main idea of our analysis relies on the control of the following gap between the upper and lower envelopes, defined, for every \( n \in \mathbb{N}^* \) and \( x \in \mathbb{R} \), as:
\begin{equation}
\label{eq:h_n}
h^{(n)}(x) = \overline{C}(x; \mathbf{T}^{(n)}) - \underline{C}(x; \mathbf{T}^{(n)}),
\end{equation}
where \( \overline{C}(x; \mathbf{T}^{(n)}) \) and \( \underline{C}(x; \mathbf{T}^{(n)}) \) are defined in Eqs.~\ref{eq:compound_maj} and \ref{eq:compound_min}, respectively.

We start by establishing two preliminary lemmas. Lemma~\ref{ref:lemmadecrease_i} gives fundamental properties of the sequence $(h^{(n)}(x))_{n \in \mathbb{N}^*}$. These properties are then used in Lemma~\ref{lemm:compact} to construct Gaussian bounds for the products $f^+ h^{(n)}$ and $f^- h^{(n)}$, and to obtain explicit control of their integrals outside compact sets.

The convergence analysis below is carried out for any sequence of tangency sets $(T^{(n)})_{n\ge1}$ satisfying the refinement condition stated in Theorem~\ref{th1:convergence}. This property is not enforced by Algorithm~\ref{algo:tangency} alone. Its practical enforcement is discussed afterward in Section~\ref{sec:algopractical}. Under this assumption, we show that, for $n$ large enough, the absolute difference between $\overline{\mathcal{I}}^{(n)}$ and $\underline{\mathcal{I}}^{(n)}$ can be made arbitrarily small, which yields the desired convergence result.

\cblue{Part of the analysis is made under the assumption of a monomial $f$ (i.e., \eqref{eq:monomial} holds). Relaxation of this assumption is discussed at the end of the section.}

\subsection{Preliminary lemmas}

\begin{lemma}
Let $(h^{(n)}(x))_{n \in \mathbb{N}^*}$ defined in \eqref{eq:h_n}. The following properties hold:
\begin{enumerate}
    \item[(i)] For every \( x \in \mathbb{R} \), the sequence \( (h^{(n)}(x))_{n \in \mathbb{N}^*} \) is positive-valued and monotonically decreasing.
    \label{ref:lemmadecrease_i}
    \item[(ii)] For every \( n \in \mathbb{N}^* \), \( h^{(n)} \) is continuous on \( \mathbb{R} \).
    \label{ref:lemmadecrease_ii}
\end{enumerate}
\end{lemma}

\begin{proof}
(i)
Let $n \in \mathbb{N}^*$. We have $\{t_m^{(n)} \}_{1 \leq m \leq T_n} \subset \{t_m^{(n+1)} \}_{1 \leq m \leq T_{n+1} }$, \cblue{that we denote shortly as $\mathbf{T}^{(n)} \subset \mathbf{T}^{(n+1)}$, with a slight abuse of notation}. Then, by definition of the piecewise bounds, we have
\begin{equation}
\forall x \in \mathbb{R}, \quad \pi(x) \leq \overline{C}(x;\mathbf{T}^{(n+1)})  \leq \overline{C}(x;\mathbf{T}^{(n)}),
\end{equation}
and
\begin{equation}
\forall x \in \mathbb{R}, \quad 
\underline{C}(x;\mathbf{T}^{(n)})  \leq \underline{C}(x;\mathbf{T}^{(n+1)}) \leq \pi(x).
\end{equation}
Therefore,
\begin{equation}
0 \leq h^{(n+1)}(x)  \leq h^{(n)}(x).
\end{equation}
Hence, we have the result. \\
(ii) 
We have that for every $n \in \mathbb{N}^*$, $\overline{C}(\cdot;\mathbf{T}^{(n)})$ and $\underline{C}(\cdot;\mathbf{T}^{(n)})$ are continuous in $\mathbb{R}$. Since they are respectively the max and the min of Gaussian functions, and we know that the max and the min are operators that maintain continuity.\\
Hence we have the result.
\end{proof}

\begin{corollary}
Let, for every \( n \in \mathbb{N} \),
\begin{equation}
    G^{(n)} : = \overline{\mathcal{I}}^{(n)} - \underline{\mathcal{I}}^{(n)} = \int_{\bR}|f(x)|h^{(n)}(x)dx
    \label{eq:G_n}
\end{equation}
where \( \overline{\mathcal{I}}^{(n)} \) and \( \underline{\mathcal{I}}^{(n)} \) are defined in Eqs.~\ref{eq:ub_I} and \ref{eq:lb_I}, respectively. The following properties hold:
\begin{enumerate}
    \item[(i)] The sequence \( (G^{(n)})_{n \in \mathbb{N}^*} \) is positive-valued and monotonically decreasing.
    \item[(ii)] There exists a function \( h : \mathbb{R} \to \mathbb{R}^+ \) such that the sequence \( (h^{(n)})_{n \in \mathbb{N}^*} \), defined as in Eq.~\ref{eq:h_n}, converges pointwise to \( h \).
    \label{corollary_limit_h(n)}
\end{enumerate}
\end{corollary}

\begin{proof}
(i) Let $n \in \mathbb{N}^*$. By Lemma \ref{ref:lemmadecrease_i},
\begin{equation}
0 \leq \int_{\mathbb{R}} |f(x)| \left(\overline{C}(x;\mathbf{T}^{(n+1)}) - \underline{C}(x;\mathbf{T}^{(n+1)})\right) dx \leq 
\int_{\mathbb{R}} |f(x)| \left(\overline{C}(x;\mathbf{T}^{(n)}) - \underline{C}(x;\mathbf{T}^{(n)})\right) dx,
\end{equation}
hence the result. \\

(ii) By Lemma \ref{ref:lemmadecrease_i}, we know that, for every $x \in \mathbb{R}$, sequence $(h^{(n)}(x))_{n \in \mathbb{N}^*}$ is positive valued and monotonically decreasing. Consequently, there exists $h : \mathbb{R} \mapsto \mathbb{R}^+$ such that, for every $x \in \mathbb{R}$, $\lim_{n \rightarrow \infty} h^{(n)}(x) = h(x)$, hence the result.  
\end{proof}

\begin{lemma} \label{lemm:compact}
Set $\mathbf{T}^{(1)} = \{t_1\}$ where $t_1 \in \bR$. \cblue{Let $f$ read as in \eqref{eq:monomial}, with monomial power $k \in \mathbb{N}$.} %\solal{I remove the fact that we suppose f monomial and the definition of $f^+$ and $f^-$ cause it's already define earlier}
We have the following statements. 
\begin{enumerate}
    \item[(i)] For every $n \in \mathbb{N}^*$:   
    \begin{equation}
    \label{eq_fplus}
         (\forall x \in \mathbb{R}) \ (\forall s \in \mathbb{R}^+) \quad f^+(x) h^{(n)}(x) \leq \widetilde{C}(s) g(x; \widetilde{\mu}(s), \widetilde{\sigma}(s)),
    \end{equation}
    \begin{equation}
    \label{eq_fmoins}
        (\forall x \in \mathbb{R}) \ (\forall s \in \mathbb{R}^-) \quad f^-(x) h^{(n)}(x) \leq \widehat{C}(s) g(x; \widehat{\mu}(s), \widehat{\sigma}(s)),
    \end{equation}  
with $(h^{(n)})_{n\in\mathbb{N}^*}$ defined in~\eqref{eq:h_n} and

   \begin{equation}
(\forall s > 0) \quad
\begin{cases}
\label{eq:widetildesigma}
\widetilde{\sigma}(s) & = \overline{\sigma}(t_1),\\
\widetilde{\mu}(s) & = \overline{\mu}(t_1)+\frac{\sigma^2(t_1)k}{s} \\
\widetilde{C}(s) & = \overline{C}(t_1)\, \frac{\widetilde{\sigma}(s)}{\overline{\sigma}(t_1)}\,
\exp\!\left(
- \frac{(s-\overline{\mu}(t_1))^2}{2\,\overline{\sigma}(t_1)^2}
+ k\ln(s)
+ \frac{\widetilde{\sigma}(s)^2}{2}\left(\frac{s-\overline{\mu}(t_1)}{\overline{\sigma}(t_1)^2} - \frac{k}{s}\right)^{\!2}
\right).
\end{cases}
\end{equation}

\begin{equation}
\label{eq:widehatsigma}
(\forall s < 0) \quad
\begin{cases}
\widehat{\sigma}(s) &= \overline{\sigma}(t_1), \\
\widehat{\mu}(s) &= \overline{\mu}(t_1)+\frac{\sigma^2(t_1)k}{s} \\
\widehat{C}(s) &= \overline{C}(t_1)\, \frac{\widehat{\sigma}(s)}{\overline{\sigma}(t_1)}\,
\exp\!\left(
- \frac{(-s+\overline{\mu}(t_1))^2}{2\,\overline{\sigma}(t_1)^2}
+ k \ln(-s)
+ \frac{\widehat{\sigma}(s)^2}{2}\left(\frac{-s+\overline{\mu}(t_1)}{\overline{\sigma}(t_1)^2} + \frac{k}{s}\right)^{\!2}
\right).
\end{cases}
\end{equation}

    where $\overline{\mu}(t_1)$, $\overline{\sigma}(t_1)$ and $\overline{C}(t_1)$ are defined in \eqref{lem:ass2}.

{\color{blue}
    \item[(ii)] Let $\varepsilon>0$.
For every $s>0$, define
\begin{equation}
\label{eq:size_compact_f+}
\widetilde{K}_{\varepsilon}(s)
=
\left[0,\max\left(0,\widetilde{\mu}(s)+\widetilde{\sigma}(s)\sqrt{\frac{2\widetilde{C}(s)}{\varepsilon}}\right)\right].
\end{equation}
Then, for every $n\in\mathbb{N}^*$,
\begin{equation}
\int_{\mathbb{R}\setminus \widetilde{K}_{\varepsilon}(s)} f^+(x)h^{(n)}(x)\,dx
\le \frac{\varepsilon}{2}.
\end{equation}

Similarly, for every $s<0$, define
\begin{equation}
\label{eq:size_compact_f-}
\widehat{K}_{\varepsilon}(s)
=
\left[\min\left(0,\widehat{\mu}(s)-\widehat{\sigma}(s)\sqrt{\frac{2\widehat{C}(s)}{\varepsilon}}\right),0\right].
\end{equation}
Then, for every $n\in\mathbb{N}^*$,
\begin{equation}
\int_{\mathbb{R}\setminus \widehat{K}_{\varepsilon}(s)} f^-(x)h^{(n)}(x)\,dx
\le \frac{\varepsilon}{2}.
\end{equation}
}
\end{enumerate}
 
\begin{proof}
\begin{enumerate}
    \item[(i)]
Let $n \in \mathbb{N}^*$. By Eq.~\eqref{ref:lemmadecrease_i}, $(h^{(n)})_{n\in\mathbb{N}^*}$ is decreasing, hence
\begin{equation}
\begin{aligned}
(\forall x \in \mathbb{R}) \quad f^+(x) h^{(n)}(x) 
&\leq f^+(x) h^{(1)}(x) \\
&\leq f^+(x)\big(\overline{C}(x;\mathbf{T}^{(1)})-\underline{C}(x;\mathbf{T}^{(1)})\big) \\
&\leq f^+(x)\,\overline{C}(t_1)\,g\!\left(x;\overline{\mu}(t_1),\overline{\sigma}(t_1)\right).
\end{aligned}
\end{equation}
Using \eqref{eq:monomial},
\begin{equation}
(\forall x \in \mathbb{R})\quad 
f^+(x) h^{(n)}(x) \leq \max\!\Big(0,\ \overline{C}(t_1)\,x^k\, g\!\left(x;\overline{\mu}(t_1),\overline{\sigma}(t_1)\right)\Big),
\end{equation}
that is, for $x>0$,
\begin{equation}
\overline{C}(t_1)\,x^k\, g\!\left(x;\overline{\mu}(t_1),\overline{\sigma}(t_1)\right)
=\overline{C}(t_1)\,\frac{1}{\sqrt{2\pi}\,\overline{\sigma}(t_1)}\,
\exp\!\left(-\phi_+(x)\right),
\end{equation}
with
\begin{equation}
\phi_+(x)=\frac{1}{2\,\overline{\sigma}(t_1)^2}\big(x-\overline{\mu}(t_1)\big)^2 - k\ln x
\qquad (x>0).
\end{equation}
Since
\begin{equation}
\phi_+''(x)=\frac{1}{\overline{\sigma}(t_1)^2}+\frac{k}{x^2}\ \ge\ \frac{1}{\overline{\sigma}(t_1)^2}\quad (x>0),
\end{equation}
$\phi_+$ satisfies Assumption~\ref{assump2} with the (constant) modulus $\nu\equiv 1/\overline{\sigma}(t_1)^2$. By Lemma~\ref{lem:ass2}, for every $s>0$,
\begin{equation}
(\forall x\in\mathbb{R})\qquad
f^+(x) h^{(n)}(x) \le \widetilde{C}(s)\, g\!\left(x;\widetilde{\mu}(s),\widetilde{\sigma}(s)\right),
\end{equation}
with the parameters given in \eqref{eq:widetildesigma}, namely
\begin{equation}
\widetilde{\sigma}(s)=\overline{\sigma}(t_1),\qquad
\widetilde{\mu}(s)=s-\widetilde{\sigma}(s)^2\!\left(\frac{s-\overline{\mu}(t_1)}{\overline{\sigma}(t_1)^2}-\frac{k}{s}\right)=\overline{\mu}(t_1)+\frac{\sigma^2(t_1)k}{s},
\end{equation}
\begin{equation}
\widetilde{C}(s)=\overline{C}(t_1)\,\frac{\widetilde{\sigma}(s)}{\overline{\sigma}(t_1)}\,
\exp\!\left(-\frac{(s-\overline{\mu}(t_1))^2}{2\,\overline{\sigma}(t_1)^2}
+ k\ln s
+ \frac{\widetilde{\sigma}(s)^2}{2}\left(\frac{s-\overline{\mu}(t_1)}{\overline{\sigma}(t_1)^2}-\frac{k}{s}\right)^{\!2}\right).
\end{equation}

For the negative part, parametrize the negative support by $y=-x>0$. For $x\le 0$, $f^-(x)=(-x)^k$, and
\begin{equation}
\overline{C}(t_1)\,(-x)^k\, g\!\left(x;\overline{\mu}(t_1),\overline{\sigma}(t_1)\right)
=\overline{C}(t_1)\,y^k\, g\!\left(-y;\overline{\mu}(t_1),\overline{\sigma}(t_1)\right)
=\overline{C}(t_1)\,\frac{1}{\sqrt{2\pi}\,\overline{\sigma}(t_1)}\,\exp\!\left(-\phi_-(y)\right),
\end{equation}
with
\begin{equation}
\phi_-(y)=\frac{1}{2\,\overline{\sigma}(t_1)^2}\big(y+\overline{\mu}(t_1)\big)^2 - k\ln y \qquad (y>0).
\end{equation}
Since
\begin{equation}
\phi_-''(y)=\frac{1}{\overline{\sigma}(t_1)^2}+\frac{k}{y^2}\ \ge\ \frac{1}{\overline{\sigma}(t_1)^2}\quad (y>0),
\end{equation}
$\phi_-$ also satisfies Assumption~\ref{assump2} with $\nu\equiv 1/\overline{\sigma}(t_1)^2$. By Lemma~\ref{lem:ass2}, for every $u>0$,
\begin{equation}
y^k\, g\!\left(-y;\overline{\mu}(t_1),\overline{\sigma}(t_1)\right)
\le \widehat{C}(u)\, g\!\left(y;\widehat{\mu}_+(u),\widehat{\sigma}(u)\right),
\end{equation}
with
\begin{equation}
\widehat{\sigma}(u)=\overline{\sigma}(t_1),\qquad
\widehat{\mu}_+(u)=u-\widehat{\sigma}(u)^2\,\phi_-'(u)
= u-\widehat{\sigma}(u)^2\,\!\left(\frac{u+\overline{\mu}(t_1)}{\overline{\sigma}(t_1)^2}-\frac{k}{u}\right)=-\overline{\mu}(t_1)+\frac{\overline{\sigma}(t_1)^2k}{u},
\end{equation}
\begin{equation}
\widehat{C}(u)=\overline{C}(t_1)\,\frac{\widehat{\sigma}(u)}{\overline{\sigma}(t_1)}\,
\exp\!\left(-\frac{(u+\overline{\mu}(t_1))^2}{2\,\overline{\sigma}(t_1)^2}
+ k\ln u
+ \frac{\widehat{\sigma}(u)^2}{2}\left(\frac{u+\overline{\mu}(t_1)}{\overline{\sigma}(t_1)^2}-\frac{k}{u}\right)^{\!2}\right).
\end{equation}
Returning to $x=-y$ and using $g(-x;\mu,\sigma)=g(x;-\mu,\sigma)$, and reindexing by $s:=-u<0$ (so $-s>0$), we obtain
\begin{equation}
(\forall s<0)\ (\forall x\in\mathbb{R})\qquad
f^-(x) h^{(n)}(x) \le \widehat{C}(s)\, g\!\left(x;\widehat{\mu}(s),\widehat{\sigma}(s)\right),
\end{equation}
with the parameters given in \eqref{eq:widehatsigma}, namely
\begin{equation}
\widehat{\sigma}(s)=\overline{\sigma}(t_1),\qquad
\widehat{\mu}(s)=-\,\widehat{\mu}_+(-s)
= \overline{\mu}(t_1)+\frac{\overline{\sigma}(t_1)^2k}{s},
\end{equation}
\begin{equation}
\widehat{C}(s)=\overline{C}(t_1)\,\frac{\widehat{\sigma}(s)}{\overline{\sigma}(t_1)}\,
\exp\!\left(-\frac{(-s+\overline{\mu}(t_1))^2}{2\,\overline{\sigma}(t_1)^2}
+ k\ln(-s)
+ \frac{\widehat{\sigma}(s)^2}{2}\left(\frac{-s+\overline{\mu}(t_1)}{\overline{\sigma}(t_1)^2}+\frac{k}{s}\right)^{\!2}\right).
\end{equation}
This concludes the proof of (i).

    \item[(ii)]
    {\color{blue}
    Let $\varepsilon>0$.
    
    \textbf{Case of $f^+$:} Fix $s>0$. From the proof of item (i), for every $n\in\mathbb N^*$ and every $x\in\mathbb R$,
\begin{equation}
f^+(x)h^{(n)}(x)\le \widetilde C(s)\, g(x;\widetilde\mu(s),\widetilde\sigma(s)).
\end{equation}
    }
%     \begin{equation} \label{eq:f+_2}
%     (\forall n \in \bN) \ (\forall x \in \mathbb{R}) \ (\forall s \in \mathbb{R}^+), \quad f^+(x) h^{(n)}(x) \leq \widetilde{C}(s) g(x; \widetilde{\mu}(s), \widetilde{\sigma}(s)),
%     \end{equation}
% and
%     \begin{equation} \label{eq:f-_1}
%     (\forall n \in \bN) \ (\forall x \in \mathbb{R}) \ (\forall s \in \mathbb{R}^-), \quad f^-(x) h^{(n)}(x) \leq \widehat{C}(s) g(x; \widehat{\mu}(s), \widehat{\sigma}(s)).
%     \end{equation}

    From the Bienaymé-Chebyshev inequality \cite{BertsekasTsitsiklis}, that 
    \begin{equation} \label{eq:Bienayme_Chebyshev_basic}
   (\forall b > 0) \quad \mathrm{P}(|X - E(X)| \geq b) \leq \frac{\mathrm{Var}(X)}{b^2}
    \end{equation}
    with $X$ a random variable. Let us set
    \begin{equation}
     \quad X \sim \mathcal{N}(\widetilde{\mu}(s), \widetilde{\sigma}(s)^2)
    \end{equation}
    Hence \eqref{eq:Bienayme_Chebyshev_basic} yields:
    \begin{equation} \label{eq:Bienayme_Chebyshev_gaussian_f+1}
     \mathrm{P}(|X - \widetilde{\mu}(s)| \geq b) \leq \frac{\widetilde{\sigma}(s)^2}{b^2}
    \end{equation}
Let us now choose $b > 0$ such that $\frac{\varepsilon}{2\widetilde{C}(s)} = \frac{\widetilde{\sigma}(s)^2}{b^2}$, i.e.
    \begin{equation}
    b = \sqrt{\frac{2 \widetilde{\sigma}(s)^2 \widetilde{C}(s)}{\varepsilon}}
    \end{equation}
    By substituting $b$ in \eqref{eq:Bienayme_Chebyshev_gaussian_f+1}, we have:
    \begin{equation} \label{eq:Bienayme_Chebyshev_gaussian_f+2}
     \mathrm{P}(|X - \widetilde{\mu}(s)| \geq b) \leq \frac{\varepsilon}{2\widetilde{C}(s)}.
    \end{equation}    
    Therefore, 
    \begin{equation}
    \label{eq:cheby_principle}
    \int_{\mathbb{R} \setminus [\widetilde{\mu}(s) - b, \widetilde{\mu}(s) + b]}  g(x; \widetilde{\mu}(s), \widetilde{\sigma}(s)) \, dx \leq \frac{\varepsilon}{2\widetilde{C}(s)}.
    \end{equation} 

  As $f^+(x) = 0$ for $x \in \mathbb{R}^-$, we define the compact set $\widetilde{K}_{\varepsilon}(s) = [0, \max(0, \widetilde{\mu}(s) + b)]$. To prove the bound on the integral, we distinguish two cases based on the sign of the upper bound $\widetilde{\mu}(s) + b$:

\begin{itemize}
    \item \textbf{Case 1: $\widetilde{\mu}(s) + b > 0$.} 
    In this case, $\widetilde{K}_{\varepsilon}(s) = [0, \widetilde{\mu}(s) + b]$.
    \begin{itemize}
        \item If $\widetilde{\mu}(s) - b \geq 0$, then $[\widetilde{\mu}(s) - b, \widetilde{\mu}(s) + b] \subset [0, \widetilde{\mu}(s) + b]$. Since the integrand is positive, the integral outside the larger set is smaller than the integral outside the smaller one:
        \begin{equation*}
        \int_{\mathbb{R} \setminus [0, \widetilde{\mu}(s) + b]} f^+(x)h^{(n)}(x) \, dx \leq \int_{\mathbb{R} \setminus [\widetilde{\mu}(s) - b, \widetilde{\mu}(s) + b]} \widetilde{C}(s) g(x; \widetilde{\mu}(s), \widetilde{\sigma}(s)) \, dx \leq \frac{\varepsilon}{2}.
        \end{equation*}
        \item If $\widetilde{\mu}(s) - b < 0$, then $\mathbb{R} \setminus [0, \widetilde{\mu}(s) + b] = (-\infty, 0) \cup (\widetilde{\mu}(s) + b, +\infty)$. On $(-\infty, 0)$, $f^+(x) = 0$. On $(\widetilde{\mu}(s) + b, +\infty)$, the region is contained in the upper tail of the Gaussian outside the Chebyshev interval. Thus, the integral is again bounded by $\varepsilon/2$.
    \end{itemize}

    \item \textbf{Case 2: $\widetilde{\mu}(s) + b \leq 0$.}
    Then $\widetilde{K}_{\varepsilon}(s) = \{0\}$. Since $f^+(x) = 0$ for $x \leq 0$, we only consider the integral over $(0, +\infty)$. Given that $0 \geq \widetilde{\mu}(s) + b$, the interval $(0, +\infty)$ is entirely contained in the upper tail $(\widetilde{\mu}(s) + b, +\infty)$. Thus, the integral is bounded by the mass of the Gaussian majorant outside the Chebyshev interval, which is $\leq \varepsilon/2$.
    
\end{itemize}

In all cases, $\int_{\mathbb{R} \setminus \widetilde{K}_{\varepsilon}(s)} f^+(x)h^{(n)}(x) \, dx \leq \varepsilon/2$.

\textbf{Case of $f^-$:} Fix $s<0$, the same reasoning applies by symmetry. By using the change of variables $y = -x$, the fact that $f^-(x) = 0$ for $x > 0$, and the Gaussian bound from (i), we define $\widehat{K}_{\varepsilon}(s) = [\min(0, \widehat{\mu}(s) - b), 0]$. By proceeding exactly as above for the lower tail of the Gaussian, we obtain $\int_{\mathbb{R} \setminus \widehat{K}_{\varepsilon}(s)} f^-(x)h^{(n)}(x) \, dx \leq \varepsilon/2$.

\end{enumerate}
\end{proof}
\end{lemma}

\subsection{Convergence theorem}
\cblue{
In this section, we prove that the approximation error $G^{(n)}$ vanishes, provided that the sequence of tangency grids $(\mathbf{T}^{(n)})_{n\ge 1}$ satisfies a so-called \emph{compact refinement property}. To do so, we control the error both on the tails and on a suitable compact region. %Lemma \ref{lemm:compact} allows us to define compact sets outside of which the approximation error is arbitrarily small. For a given tolerance $\varepsilon > 0$, and fixed parameters $s_0 > 0$ and $s_1 < 0$, we define, following the notations of Lemma \ref{lemm:compact},
%$$K^+ := \widetilde{K}_\varepsilon(s_0) \quad \text{and} \quad K^- := \widehat{K}_\varepsilon(s_1).$$
%By definition, $K^+$ is a closed interval of the form $[0, b^+]$ and $K^-$ is of the form $[-b^-, 0]$. We can therefore define a single global compact interval covering both supports simply by taking their union $K := K^- \cup K^+$. In practice (as detailed in Section \ref{sec:algopractical}), building list of tangent points within this single interval $K$ ensures that the approximation is uniformly refined across the regions of interest for both $f^+$ and $f^-$.
}

\cblue{
\begin{theorem}[Convergence of the approximation integral]
\label{th1:convergence}
Let $\varepsilon > 0$, and $s_0 > 0$, $s_1 < 0$. \cblue{Let $f$ reads as in \eqref{eq:monomial}, with monomial power $k \in \mathbb{N}$.} Following the notation of Lemma \ref{lemm:compact}, set $K^+= \widetilde{K}_\varepsilon(s_0)$, $K^-=\widehat{K}_\varepsilon(s_1)$, and $K= K^- \cup K^+$. Assume that the sequence of grids $(\mathbf{T}^{(n)})_{n\ge 1}$ is constructed such that for all $n\ge1$, $\mathbf{T}^{(n)} \subset K$, with $\mathbf{T}^{(1)}=\{t_1\}$ for some $t_1\in K$, and that the sequence refines $K$ in the following sense:
\begin{equation}\label{eq:H-net_thm}
\eta_n(K) := \sup_{x\in K}\inf_{t\in\mathbf{T}^{(n)}}|x-t| \xrightarrow[n\to\infty]{} 0.
\end{equation}
Then, for every $n\in\mathbb{N}^*$, the following \emph{outside-compact} bounds hold:
\begin{equation} \label{eq:int_outside_compact_f+}
\int_{\mathbb{R} \setminus K^+} f^+(x)\, h^{(n)}(x) \, dx \leq \frac{\varepsilon}{2},
\end{equation}
\begin{equation} \label{eq:int_outside_compact_f-}
\int_{\mathbb{R} \setminus K^-} f^-(x)\, h^{(n)}(x) \, dx \leq \frac{\varepsilon}{2}.
\end{equation}
Moreover, there exists $n_0 \in \mathbb{N}^*$ such that, for all $n \ge n_0$,
\begin{equation}
\label{eq:global_G_bound}
\left| G^{(n)} \right| \leq \varepsilon \left(1 + \widetilde{\mu}(s_0) + |\widehat{\mu}(s_1)| + \sqrt{\frac{2 \widehat{\sigma}(s_1)^2 \widehat{C}(s_1)}{\varepsilon}} + \sqrt{\frac{2 \widetilde{\sigma}(s_0)^2 \widetilde{C}(s_0)}{\varepsilon}} \right).
\end{equation}
In particular, $\lim_{n \rightarrow \infty} G^{(n)} = 0$,  $\lim_{n \rightarrow \infty} \overline{\mathcal{I}}^{(n)} = \mathcal{I}$, and $\lim_{n \rightarrow \infty} \underline{\mathcal{I}}^{(n)} = \mathcal{I}$.
\end{theorem}
\begin{proof}
Fix $\varepsilon>0$ and $s_0>0$, $s_1<0$, and define $K^+=\widetilde K_\varepsilon(s_0)$, $K^-=\widehat K_\varepsilon(s_1)$, and $K=K^-\cup K^+$ as in the statement. The bounds \eqref{eq:int_outside_compact_f+} and \eqref{eq:int_outside_compact_f-} follow directly from Lemma~\ref{lemm:compact} by construction of $K^+$ and $K^-$. These estimates already provide the desired control on the tails of the integral. It remains to bound the error on the compact core $K$. By hypothesis, the elements of $(\mathbf{T}^{(n)})_{n\ge 1}$ belong to $K$ and refine it (i.e., $\eta_n(K) \to 0$). \\
\smallskip
We now detail the argument to bound the integral of $x \mapsto f^+(x)h^{(n)}(x)$ on $K^+$. The reasoning for $f^-$ on $K^-$ follows exactly the same logic. Recall that for all $x \in \mathbb{R}$,
\begin{equation}
h^{(n)}(x)=\overline C(x;\mathbf T^{(n)})-\underline C(x;\mathbf T^{(n)}),
\end{equation}
with
\begin{equation}
\overline C(x;\mathbf T^{(n)})=\min_{t\in \mathbf T^{(n)}} \overline b(x;t),
\qquad
\underline C(x;\mathbf T^{(n)})=\max_{t\in \mathbf T^{(n)}} \underline b(x;t).
\end{equation}
Let us first show that the integral of $x\mapsto f^+(x)h^{(n)}(x)$ on $K^+$ can be arbitrarily small for $n$ large enough.
 %In assuming for the moment that the family
% \begin{equation}
% \mathcal F:=\{x\mapsto f^+(x)h^{(n)}(x)\}_{n\ge 1}
% \end{equation}
% is uniformly equicontinuous on $K$. 
%
\smallskip
\noindent\textbf{Main argument, assuming uniform equicontinuity.}
Let $\gamma>0$. Let assume that \begin{equation}
\mathcal F^+:=\{x\mapsto f^+(x)h^{(n)}(x)\}_{n\ge 1}
\end{equation}
is uniformly equicontinuous on $K$. Since $K^+\subset K$, the family $\mathcal F^+$ is also uniformly equicontinuous on $K^+$.
Then, there exists $\lambda>0$ such that for all $n\ge 1$ and all $(x,y)\in (K^+)^2$,
\begin{equation}
|x-y|<\lambda
\quad\Longrightarrow\quad
|f^+(x)h^{(n)}(x)-f^+(y)h^{(n)}(y)|<\gamma.
\label{eq:equicont_assumed_thm2}
\end{equation}
By \eqref{eq:H-net_thm}, there exists $n_{+}\in\mathbb N^*$ such that
\begin{equation}
\eta_n(K)<\lambda,
\qquad \forall n\ge n_{+}.
\end{equation}
Since $K^+\subset K$, we have
\begin{equation}
\sup_{x\in K^+}\inf_{t\in\mathbf{T}^{(n)}}|x-t|<\lambda,
\qquad \forall n\ge n_{+}.
\end{equation}
Therefore, for every $n\ge n_{+}$ and every $x\in K^+$, there exists $t\in\mathbf T^{(n)}$ such that
\begin{equation}
|x-t|<\lambda.
\end{equation}
Since $t\in \mathbf T^{(n)}$, the tangency/interpolation property gives
\begin{equation}
h^{(n)}(t)=0.
\end{equation}
Applying \eqref{eq:equicont_assumed_thm2} with $y=t$, we obtain
\begin{equation}
|f^+(x)h^{(n)}(x)|
=
|f^+(x)h^{(n)}(x)-f^+(t)h^{(n)}(t)|
<
\gamma.\end{equation}
Since the above holds for every $x\in K^+$, we deduce
\begin{equation}
\sup_{x\in K^+} f^+(x)h^{(n)}(x)\le \gamma,
\qquad \forall n\ge n_{+}.
\label{eq:sup_inside_compact_gamma_thm2}
\end{equation}
Taking $\gamma=\varepsilon$, we get
\begin{equation}
\sup_{x\in K^+} f^+(x)h^{(n)}(x)\le \varepsilon,
\qquad \forall n\ge n_{+}.
\label{eq:sup_inside_compact_thm2}
\end{equation}
Therefore,
\begin{equation}
\int_{K^+} f^+(x)h^{(n)}(x)\,dx
\le
\varepsilon\,\mathrm{Leb}(K^+),
\qquad \forall n\ge n_{+}.
\label{eq:int_inside_compact_thm2}
\end{equation}
Using Lemma~\ref{lemm:compact} (see \eqref{eq:size_compact_f+}), we have
$$
\mathrm{Leb}(K^+)
=
\widetilde{\mu}(s_0)+\sqrt{\frac{2\widetilde{\sigma}(s_0)^2\widetilde{C}(s_0)}{\varepsilon}},
$$
which yields
\begin{equation}\label{eq:int_f+_thm2}
\int_{\mathbb{R}} f^+(x)\, h^{(n)}(x) \, dx
\le
\frac{\varepsilon}{2}
+
\varepsilon\left(\widetilde{\mu}(s_0)+\sqrt{\frac{2\widetilde{\sigma}(s_0)^2\widetilde{C}(s_0)}{\varepsilon}}\right),
\qquad \forall n\ge n_{+}.
\end{equation}
The same argument applied on $K^-$ gives the following estimate, provided that the family
$$
\mathcal F^-:=\{x\mapsto f^-(x)h^{(n)}(x)\}_{n\ge 1}
$$
is uniformly equicontinuous on $K$. More precisely, there exists $n_-\in\mathbb N^*$ such that, for all $n\ge n_-$,
$$
\int_{\mathbb R} f^-(x)h^{(n)}(x)\,dx
\le
\frac{\varepsilon}{2}
+
\varepsilon\left(
|\widehat{\mu}(s_1)|
+
\sqrt{\frac{2\widehat{\sigma}(s_1)^2\widehat C(s_1)}{\varepsilon}}
\right).
$$
Let $n_0=\max(n_+,n_-)$. Combining the estimates for the positive and negative parts, we obtain, for all $n\ge n_0$,
$$
|G^{(n)}|
\le
\varepsilon
\left(
1
+
\widetilde{\mu}(s_0)
+
|\widehat{\mu}(s_1)|
+
\sqrt{\frac{2\widetilde{\sigma}(s_0)^2\widetilde C(s_0)}{\varepsilon}}
+
\sqrt{\frac{2\widehat{\sigma}(s_1)^2\widehat C(s_1)}{\varepsilon}}
\right).
$$
This proves \eqref{eq:global_G_bound}, assuming the uniform equicontinuity of the two families introduced above. It remains to justify this property.
\smallskip
\noindent\textbf{Verification of the uniform equicontinuity property.} We now prove that the families
$$
\mathcal F^+
:=
\{x\mapsto f^+(x)h^{(n)}(x)\}_{n\ge 1}
\quad\text{and}\quad
\mathcal F^-
:=
\{x\mapsto f^-(x)h^{(n)}(x)\}_{n\ge 1}
$$
are uniformly equicontinuous on $K$. Since $K$ is compact and the functions $\nu(\cdot)$ and $\beta(\cdot)$ are continuous and strictly positive on $K$, there exist constants $0<\sigma_{\min}\le \sigma_{\max}<\infty$ such that
\begin{equation}
\underline\sigma(t),\overline\sigma(t)\in [\sigma_{\min},\sigma_{\max}],
\qquad \forall t\in K.
\end{equation}
Moreover, since $\phi$ and $\dot\phi$ are continuous, the maps
\begin{equation}
t\mapsto \underline\mu(t),\quad
t\mapsto \overline\mu(t),\quad
t\mapsto \underline C(t),\quad
t\mapsto \overline C(t)
\end{equation}
are continuous on $K$, hence bounded. Therefore, the gradient maps
\begin{equation}
(x,t)\mapsto \nabla_x \overline b(x;t)
=
-\overline C(t)\frac{x-\overline\mu(t)}{\overline\sigma(t)^2}\,g(x;\overline\mu(t),\overline\sigma(t))
\end{equation}
and
\begin{equation}
(x,t)\mapsto \nabla_x \underline b(x;t)
=
-\underline C(t)\frac{x-\underline\mu(t)}{\underline\sigma(t)^2}\,g(x;\underline\mu(t),\underline\sigma(t))
\end{equation}
are continuous on $K\times K$. Since $K\times K$ is compact, they are bounded on this set. Hence there exists a constant $L_K>0$, depending only on $K$, such that for all $t\in K$ and all $(x,y)\in K^2$,
\begin{equation}
|\overline b(x;t)-\overline b(y;t)|\le L_K|x-y|,
\qquad
|\underline b(x;t)-\underline b(y;t)|\le L_K|x-y|.
\label{eq:Lip_b_thm2}
\end{equation}
In particular, for every $n\ge 1$, since $\mathbf T^{(n)}\subset K$, the families
\begin{equation}
\{x\mapsto \overline b(x;t)\}_{t\in \mathbf T^{(n)}}
\qquad\text{and}\qquad
\{x\mapsto \underline b(x;t)\}_{t\in \mathbf T^{(n)}}
\end{equation}
consist of functions that are all $L_K$-Lipschitz on $K$, with the same constant independent of $n$. Applying Lemma~\ref{lemm:max_min_lip_function} from Appendix C, we obtain that
$
\overline C(\cdot;\mathbf T^{(n)})
\,\text{and}\,
\underline C(\cdot;\mathbf T^{(n)})
$
are $L_K$-Lipschitz on $K$ for every $n\ge 1$. Consequently,
\begin{equation}
h^{(n)}=\overline C(\cdot;\mathbf T^{(n)})-\underline C(\cdot;\mathbf T^{(n)})
\end{equation}
is $2L_K$-Lipschitz on $K$, namely,
\begin{equation}
|h^{(n)}(x)-h^{(n)}(y)|\le 2L_K|x-y|,
\qquad \forall (x,y)\in K^2,\ \forall n\ge 1.
\label{eq:Lip_h_thm2}
\end{equation}
On the other hand, the maps
\begin{equation}
(x,t)\mapsto \underline b(x;t)
\qquad\text{and}\qquad
(x,t)\mapsto \overline b(x;t)
\end{equation}
are continuous on $K\times K$, hence bounded. Therefore, there exists $M_b>0$ such that
\begin{equation}
|\underline b(x;t)|\le M_b,
\qquad
|\overline b(x;t)|\le M_b,
\qquad \forall (x,t)\in K\times K.
\end{equation}
Since, for each $n\ge 1$, the functions $\underline C(\cdot;\mathbf T^{(n)})$ and $\overline C(\cdot;\mathbf T^{(n)})$ are obtained as pointwise maxima and minima of these families, it follows that
\begin{equation}
\|\underline C(\cdot;\mathbf T^{(n)})\|_{\infty,K}\le M_b,
\qquad
\|\overline C(\cdot;\mathbf T^{(n)})\|_{\infty,K}\le M_b,
\qquad \forall n\ge 1.
\end{equation}
Hence
\begin{equation}
\|h^{(n)}\|_{\infty,K}
\le
\|\underline C(\cdot;\mathbf T^{(n)})\|_{\infty,K}
+
\|\overline C(\cdot;\mathbf T^{(n)})\|_{\infty,K}
\le 2M_b,
\qquad \forall n\ge 1.
\end{equation}
Therefore,
\begin{equation}
M_h:=\sup_{n\ge 1}\|h^{(n)}\|_{\infty,K}<\infty.
\end{equation}
Since $f^+$ is continuous on the compact set $K$, it is uniformly continuous and bounded on $K$. Set
\begin{equation}
M_f:=\|f^+\|_{\infty,K}<\infty.
\end{equation}
Let $\gamma>0$. By uniform continuity of $f^+$ on $K$, there exists $\lambda_1>0$ such that
\begin{equation}
|x-y|<\lambda_1
\quad\Longrightarrow\quad
|f^+(x)-f^+(y)|<\frac{\gamma}{2M_h}.
\end{equation}
On the other hand, by \eqref{eq:Lip_h_thm2}, for all $n\ge 1$ and all $(x,y)\in (K)^2$,
\begin{equation}
|h^{(n)}(x)-h^{(n)}(y)|\le 2L_K|x-y|.
\end{equation}
Hence, choosing
\begin{equation}
\lambda_2:=\frac{\gamma}{4L_KM_f},
\end{equation}
we obtain
\begin{equation}
|x-y|<\lambda_2
\quad\Longrightarrow\quad
|h^{(n)}(x)-h^{(n)}(y)|<\frac{\gamma}{2M_f},
\qquad \forall n\ge 1.
\end{equation}
Let
\begin{equation}
\lambda:=\min(\lambda_1,\lambda_2).
\end{equation}
Then, for all $n\ge 1$ and all $(x,y)\in (K)^2$ such that $|x-y|<\lambda$,
\begin{equation}
\begin{aligned}
|f^+(x)h^{(n)}(x)-f^+(y)h^{(n)}(y)|
&\le |f^+(x)-f^+(y)|\,|h^{(n)}(x)| + |f^+(y)|\,|h^{(n)}(x)-h^{(n)}(y)| \\
&\le M_h\,|f^+(x)-f^+(y)| + M_f\,|h^{(n)}(x)-h^{(n)}(y)| \\
&< \gamma.
\end{aligned}
\end{equation}
Thus the family $\mathcal F^+$ is uniformly equicontinuous on $K$. The same argument, with $f^-$ in place of $f^+$, shows that the family $\mathcal F^-$ is also uniformly equicontinuous on $K$. This completes the proof.\\
\textbf{Convergence of the sequences:} Since \eqref{eq:global_G_bound} holds for any positive $\varepsilon$, the gap $G^{(n)}$ goes to $0$ when $n$ goes to infinity. Since, for every $n \geq 1$,
$$
G^{(n)}
=
\left(\overline{\mathcal{I}}^{(n)}-\mathcal{I}\right)
+
\left(\mathcal{I}-\underline{\mathcal{I}}^{(n)}\right),
$$
with both terms nonnegative, it follows that
$$
\overline{\mathcal{I}}^{(n)}\to \mathcal{I}
\qquad\text{and}\qquad
\underline{\mathcal{I}}^{(n)}\to \mathcal{I}.
$$
\end{proof}
}

\cblue{Let us formulate some important remarks, to clarify the scope of the above theoretical result.}

\cblue{
\begin{remark}
It is important to stress that the uniform equicontinuity property is crucial in the proof. Indeed, the argument on the compact set $K^+$ relies on the fact that, for every $\gamma>0$, one can choose a radius $\lambda>0$, \emph{independant from $n$}, such that
\begin{equation*}
|x-y|<\lambda
\quad\Longrightarrow\quad
|f^+(x)h^{(n)}(x)-f^+(y)h^{(n)}(y)|<\gamma,
\qquad \forall n\ge 1,\ \forall (x,y)\in (K^+)^2.
\end{equation*}
%The key point is that this radius $\lambda$ must be independent of $n$.
\end{remark}
\begin{remark}
\label{rk:dyadic}
The refinement assumption \eqref{eq:H-net_thm} is mild and is satisfied by many standard grid constructions. A canonical example is given by dyadic grids, which are also used in our numerical experiments. Indeed, let $K$ a non-empty compact subset of $\mathbb R$, and, for every $n\geq1$, $\mathbf T^{(n)}$ is chosen as the dyadic grid of mesh size $2^{-n}$ on an interval containing $K$, then \eqref{eq:H-net_thm} holds. More precisely, one has
\begin{equation}
(\forall n \geq 1) \quad \eta_n(K)\le C\,2^{-n},
\end{equation}
for some constant $C>0$ depending only on the embedding of $K$ into the ambient dyadic interval.
Moreover, 
%the argument used in the proof of Theorem~\ref{th1:convergence} yields a quantitative estimate once a uniform Lipschitz bound is available for the family
%\begin{equation}
%\mathcal F_\mathcal{S}:=\{x\mapsto f^+(x)h^{(n)}(x)\}_{n\ge1}.
%\end{equation}
%Namely, 
if there exists $L_K>0$ such that the functions $\{x\mapsto f^+(x)h^{(n)}(x)\}_{n\ge1}$ are $L_K$-Lipschitz on $K$, then
\begin{equation}
(\forall n \geq 1) \quad \sup_{x\in K} f^+(x)h^{(n)}(x)\le L_K\,\eta_n(K).
\end{equation}
Consequently, for dyadic grids,
\begin{equation}
(\forall n \geq 1) \quad \sup_{x\in K} f^+(x)h^{(n)}(x)\le L_K C\,2^{-n},
\end{equation}
so that
\begin{equation}
(\forall n \geq 1) \quad \int_K f^+(x)h^{(n)}(x)\,dx \le L_K C\,\mathrm{Leb}(K)\,2^{-n}.
\end{equation}
The same reasoning applies to the negative part. Thus, on any prescribed compact set, dyadic refinement yields not only convergence but also an explicit geometric rate for the compact contribution to the approximation error.
\end{remark}
}

\begin{remark}
\label{remarkmonomial}
Theorem~\ref{th1:convergence} establishes the convergence of the bounds when $f$ is a monomial, i.e., $f$ takes the form \eqref{eq:monomial}. This result could actually be extended further: since any polynomial function is a finite linear combination of monomials, the same approximation procedure applies term-by-term, yielding convergence of the integral bounds for any polynomial integrand $f$. Interestingly, in such case, the moment computations in Algorithm~\ref{algo:tangency} can still be defined in a closed form, by decomposing the polynomial into a sum of monomials and using the recursive formulas from Appendix A. 
\cblue{
The restriction to monomials in Theorem~\ref{th1:convergence} allows explicit control of the tail contributions and, as such, it is a key ingredient in the proof of Lemma~\ref{lemm:compact}. The monomial structure would however no longer be essential to compute the integral $\int_{\mathcal{S}} f(x) \pi(x) dx$ within a compact set $\mathcal{S}$. In that setting, the proof would only require that, for some reference tangency point $t_0$, the function $x\mapsto \overline b(x;t_0)\,f(x)$ belongs to $L^1(\mathcal{S})$. Under such a local integrability assumption, using the same compact interpolation argument than in our proof of Lemma~\ref{lemm:compact}, we can shows that the integral gap computed on $\mathcal{S}$ vanishes over iterations (i.e., as the grid is refined).
}
\end{remark}

\subsection{From the convergence theorem to a practical algorithm} \label{sec:algopractical}

\cblue{
Theorem~\ref{th1:convergence} guarantees that the gap
$G^{(n)}=\overline{\mathcal I}^{(n)}-\underline{\mathcal I}^{(n)}$ tends to zero \emph{provided} the sequence of lists of tangency points $(\mathbf T^{(n)})_{n\ge 1}$ satisfies the refinement condition.%, for any compact set $K\subset\mathbb R$,
% \begin{equation}
% \eta_n(K):=\sup_{x\in K}\inf_{t\in \mathbf T^{(n)}\cap K}|x-t|\xrightarrow[n\to\infty]{}0.
% \end{equation}
}
Algorithm~\ref{algo:tangency} alone does not, by design, enforce such a refinement property: repeatedly refining the current max--gap interval does not prevent other regions from being neglected, and therefore does not guarantee that $\eta_n(K)\to 0$ on every compact $K$.
To connect the theoretical refinement condition with a practical implementation, we complement Algorithm~\ref{algo:tangency} with an explicit dyadic candidate pool on a compact interval $[a,b]$. This yields a finite-resolution version of the refinement mechanism: the tangency points are selected from a prescribed dyadic set whose granularity can be controlled in advance.
Accordingly, Algorithm~\ref{algo:dyadic_set} constructs a dyadic candidate set used to discretize the search of new tangency points. This practical construction is consistent with the refinement requirement of Theorem~\ref{th1:convergence}, although it should be understood as a finite-resolution implementation rather than as an infinite grid refinement process.

\paragraph{Adapting Algorithm~\ref{algo:tangency} with a candidate pool (no replacement)}
We implement Step~13 of Algorithm~\ref{algo:tangency} in order to control the list of tangency points. Before inserting the suggested point $\hat t$, associated with the index $\hat\iota$, we replace it by the nearest available element of a candidate pool $\mathcal U\subset [a,b]$, and then remove this selected point from $\mathcal U$ (selection without replacement). We denote the remaining candidates at iteration $n$ by
$$
\mathcal U^{(n+1)}=\mathcal U^{(n)}\setminus\{\hat t\},
$$

where $\mathcal U^{(0)}=\mathcal U$.
\paragraph{Compact used to define the pool (quantile construction)}
The theoretical compacts introduced in Lemma~\ref{lemm:compact} may be overly conservative for practical implementation. We therefore define a numerical compact interval $[a,b]$ from \emph{Gaussian quantiles} of the upper tangent bound $\overline b(\cdot;t_1)$ associated with the initial point $t_1$ (see \cite{abramowitz1972handbook, MATLABnorminv2025}). For a prescribed tail level $\gamma>0$, we choose
\begin{multline}
\label{eq_ab}
\int_{\mathbb R\setminus[a,b]} \overline b(x;t_1)\,dx \le \overline C(t_1)\,\gamma,\\
\text{with}\quad
a=\texttt{norminv}(\gamma/2,\mu(t_1),\sigma(t_1)),
\qquad
b=\texttt{norminv}(1-\gamma/2,\mu(t_1),\sigma(t_1)).
\end{multline}
Since $\overline b(\cdot;t_1)$ has Gaussian tails, polynomially weighted tail integrals with respect to $\overline b(\cdot;t_1)$ also become small as $\gamma\to 0$. In particular, when $f$ is a monomial, the quantity
$$
\int_{\mathbb R\setminus[a,b]} |f(x)|\,\overline b(x;t_1)\,dx
$$

is small for sufficiently small $\gamma$. This motivates the use of the quantile-based interval $[a,b]$ in practice: it is easy to compute and typically much less conservative than the theoretical compact sets from Lemma~\ref{lemm:compact}.

\paragraph{Dyadic pool construction}
We build the candidate set $\mathcal U$ from a dyadic decomposition of $[a,b]$, as described in Algorithm~\ref{algo:dyadic_set}.

\begin{algorithm}[H]
\caption{Dyadic Set Construction}
\label{algo:dyadic_set}
\begin{algorithmic}[1]
\Require $\ell>0$ (density parameter), $a$ (start of the interval), $b$ (end of the interval)
\Ensure $\mathcal{U}$ (dyadic candidate set), $M$ (number of candidates)
\State $L \gets \max(1,\lceil b \rceil - \lfloor a \rfloor)$ \Comment{Length of the enclosing integer interval}
\State $d_1 \gets \max\bigl(1,\, \lfloor \ell / L \rfloor \bigr)$ \Comment{Target points per unit interval}
\State $d_2 \gets \lfloor \log_2(d_1) \rfloor$ \Comment{Dyadic depth}
\State $\mathcal{U} \gets \emptyset$
\State $x \gets \lfloor a \rfloor$
\While{$x \le \lceil b \rceil$}
    \State Append $x$ to $\mathcal{U}$
    \State $x \gets x + 2^{-d_2}$ \Comment{Increment by dyadic precision}
\EndWhile
\State $M \gets |\mathcal{U}|$
\State \Return $\mathcal{U}, M$
\end{algorithmic}
\end{algorithm}

\begin{remark}
The input parameter $\ell$ controls the dyadic density within $[a,b]$. The total number of candidates $M$ depends on both $\ell$ and the interval length. As $\ell$ increases, the subdivision becomes finer and $M$ increases. The construction ensures that all integers between $\lfloor a \rfloor$ and $\lceil b \rceil$ are included in $\mathcal U$, which guarantees that an integer tangency point $t_1$ (for instance, in all our experiments, $t_1=1$) belongs to the pool. The dyadic set is therefore aligned with the compact used for integration and offers reproducible control of its granularity.
\end{remark}

\paragraph{Step 13: pool-based insertion rule}
We replace the “choose a new point $\hat t\in S_{\hat\iota}^{(n)}$” by the two-step rule described in Alg.~\ref{algo:two} hereafter; the remaining candidates at iteration $n$ are denoted $\mathcal U^{(n)}\subseteq\mathcal U$ with $\mathcal U^{(0)}=\mathcal U$.

\begin{algorithm}[H]
\caption{Decision rule with dyadic nearest-candidate (without replacement)}
\label{algo:two}
\begin{algorithmic}
\Require Iteration index $n \in \mathbb{N}^*$. Current tangency set $\mathbf{T}^{(n)}$, gaps $\{G_i^{(n)}\}_{i=1}^{M_n+1}$, subintervals $\{S_i^{(n)}\}_{i=1}^{M_n+1}$, remaining dyadic candidates $\mathcal{U}^{(n)}\subseteq\mathcal{U}$, selected index $\hat\iota$ from \eqref{eq:iota}
\Ensure New point $\hat t$, updated sets $\mathbf{T}^{(n+1)}$, $\mathcal{U}^{(n+1)}$
\If{$\mathcal{U}^{(n)}=\varnothing$} \State \Return \Comment{no candidate remains} \EndIf
\If{$\mathcal{U}^{(n)}\cap S_{\hat\iota}^{(n)}=\varnothing$}
    \State $\hat\iota \gets \displaystyle \arg\max\bigl\{\,G_i^{(n)}\ :\ \mathcal{U}^{(n)}\cap S_i^{(n)}\neq\varnothing\,\bigr\}$
    \If{no such $i$} \State \Return \Comment{no candidate remains} \EndIf
\EndIf
\State Compute the target location
$$
\hat{t}^{\star}=
\begin{cases}
\dfrac{t_{\hat{\iota}-1}^{(n)}+t_{\hat{\iota}}^{(n)}}{2}, & 1<\hat{\iota}<M_n+1,\\[0.35cm]
t_1^{(n)}-\dfrac{1}{M_n-1}\sum_{m=1}^{M_n-1}\bigl(t_{m+1}^{(n)}-t_m^{(n)}\bigr), & \hat{\iota}=1,\\[0.45cm]
t_{M_n}^{(n)}+\dfrac{1}{M_n-1}\sum_{m=1}^{M_n-1}\bigl(t_{m+1}^{(n)}-t_m^{(n)}\bigr), & \hat{\iota}=M_n+1.
\end{cases}
$$
\State Pick the nearest available dyadic candidate in $S_{\hat\iota}^{(n)}$ and remove it:
$$
\hat t \in \arg\min_{u\in \mathcal{U}^{(n)}\cap S_{\hat\iota}^{(n)}} |u-\hat t^{\star}|,
\qquad
\mathcal{U}^{(n+1)} \gets \mathcal{U}^{(n)}\setminus\{\hat t\}.
$$
\State Insert $\hat t$ into $\mathbf{T}^{(n)}$ in sorted order to obtain $\mathbf{T}^{(n+1)}$.
\end{algorithmic}
\end{algorithm}

\paragraph{Synthesis of the proposed pipeline:} Let us summarize below the steps of the complete proposed pipeline, satisfying the convergence Theorem~\ref{th1:convergence}:
\begin{enumerate}
    \item Choose initial point $t_1 \in \mathbb{N}$. 
    \item Compute $[a,b]$ using \eqref{eq_ab}.
    \item Build dyadic candidate pool $\mathcal{U}$, with size $M>0$, for a chosen density parameter $\ell>0$, using Alg.~\ref{algo:dyadic_set}.
    \item Run Alg.~\ref{algo:tangency}, using Alg.~\ref{algo:two} for new point selection at each call of step 13 in Alg.~\ref{algo:tangency}.
\end{enumerate}
We terminate Algorithm~\ref{algo:tangency} when the pool $\mathcal{U}$ is exhausted or as soon as the gap $\cblue{G^{(n)} \leq \epsilon}$,
%$\cblue{(\overline{\mathcal I}^{(n)}-\underline{\mathcal I}^{(n)})/\overline{\mathcal I}^{(n)} \ \le\ \tau}$
% $$
% G^{(n)} = \cblue{(\overline{\mathcal I}^{(n)}-\underline{\mathcal I}^{(n)})/\overline{\mathcal I}^{(n)} \ \le\ \tau,}
% $$
for a user-chosen tolerance $\epsilon>0$. The output are the integral bounds $(\underline{\mathcal{I}}^{(n)},\overline{\mathcal{I}}^{(n)})$.

% \paragraph{Patch inside Algorithm~\ref{algo:tangency}: pool-based insertion rule.}
% Replace the line “Choose a new point $\hat t\in S_{\hat\iota}^{(n)}$” by the two-step rule below.

% \medskip
% \noindent\textbf{(i) Suggestion (unchanged spirit).} Compute a suggested point $\tilde t$ in $S_{\hat\iota}^{(n)}$:
% $$
% \tilde t=\begin{cases}
% \displaystyle \frac{t_{\hat\iota-1}^{(n)}+t_{\hat\iota}^{(n)}}{2}, & \text{if } \hat\iota\in\{2,\ldots,M_n\} \ \text{(interior interval)},\\[8pt]
% \displaystyle t_1^{(n)}-\frac{1}{M_n-1}\sum_{m=1}^{M_n-1}\bigl(t_{m+1}^{(n)}-t_m^{(n)}\bigr), & \text{if } \hat\iota=1,\\[10pt]
% \displaystyle t_{M_n}^{(n)}+\frac{1}{M_n-1}\sum_{m=1}^{M_n-1}\bigl(t_{m+1}^{(n)}-t_m^{(n)}\bigr), & \text{if } \hat\iota=M_n+1.
% \end{cases}
% $$

% \noindent\textbf{(ii) Projection onto the dyadic pool (no replacement).}  
% With $\mathcal U^{(n)}$ the set of still-available candidates produced by Algorithm~\ref{algo:dyadic_set},
% $$
% \hat t \in \arg\min_{u\in \mathcal U^{(n)}\cap S_{\hat\iota}^{(n)}} |u-\tilde t|
% \quad\text{or, if empty,}\quad
% \hat t \in \arg\min_{u\in \mathcal U^{(n)}} |u-\tilde t|.
% $$
% Insert $\hat t$ in order into $\mathbf T^{(n)}$ and remove it from the pool: $\mathcal U^{(n+1)}=\mathcal U^{(n)}\setminus\{\hat t\}$.

%Ainsi, le nouvel algorithm avec le patch sur l'algorithme \ref{algo:tangency} peut s'écrire comme l'algorithme suivant :

%% file: 05a_IS_numapp.tex
In this section, we provide the specific implementation of the proposed method to the problem of variance estimation of the importance sampler. %This section serves as a preamble to the experimental Section \ref{sec_logistic}.

%In the following two sections we present a comprehensive numerical analysis to evaluate the performance of our method. First, we detail the variance estimator in importance sampling, along with the theoretical framework behind it in the section \ref{sec_IS}. Then, we conduct experiments using various sets of tangency points, including dyadic sets as in the proof of Th.~\ref{th1:convergence}, sets refined through iterative algorithms as in Alg.~\ref{algo:tangency}, and randomly generated sets. Finally, we compare the results of our method with the established approach in.~\cite{evans1997approximation}, providing insights into the accuracy and computational trade-offs of both methods, in the section \ref{sec_logistic}

\subsection{Importance sampling}

Importance sampling (IS) is a general Monte Carlo technique for the approximation of a probability density function (pdf) of interest by a random measure composed of samples and weights \cite{Robert04}. Consider $\normalized p(x)$ a pdf in $x\in\Real$, and $p(x)$ its unnormalized version, such $\normalized p(x) = \frac{p(x)}{\Zc}$, with $\Zc$ being the (generally) unknown normalizing constant. In an IS method, a set of $N$ samples, $\{x_n\}_{n=1}^N$, is drawn from a {single} proposal pdf, $q(x)$. Each sample, ${x_n}$, is then assigned an importance weight given by
 \begin{equation} 
	w_n= \frac{p(x_n)}{{q(x_n)}}, \quad n=1,\ldots,N.
\label{is_weights_static}
\end{equation} 
The samples and the weights then form the random measure $\chi=\{x_n,w_n\}_{n=1}^N$ that approximates the measure of the target pdf as
\begin{equation}
	{\widehat p}_{\text{IS}} (x)= \frac{1}{N\widehat{\Zc}}  \sum_{n=1}^N w_n \delta_{x_n}(x),
\label{approx_pi}
\end{equation}
where  $\delta_{x_n}(x)$ is the unit delta measure 
concentrated at $x_n$. Given such approximation, it is direct to build approximations of moments of $p(x)$. For instance, $\widehat{\Zc} = \frac{1}{N}\sum_{n=1}^N w_n$ is an unbiased estimator of $\Zc\triangleq\int_{\mathcal{X}} \tilde{p}(x) dx$. More generally, one can approximate $\Ic \triangleq  \int_{\mathcal{X}} {m(x)} \normalized p(x) dx$ by the unbiased IS estimator 
\begin{equation}
\widehat{\Ic} = \frac{1}{N\Zc} \sum_{n=1}^N w_n  {m(x_n)},
%\label{eq_IS_estimator_unnorm}
\label{eq_IS_estimator_unnorm}
\end{equation}
assuming $\Zc$ is known. Otherwise, if the target distribution is only known up to the normalizing constant (i.e., $\Zc$ is unknown), one can use the self-normalized estimator 
\begin{equation}
	\widetilde{\Ic} = \frac{1}{N\widehat{\Zc}}  \sum_{n=1}^N w_n {m(x_n)},
\label{eq_IS_estimator_norm}
\end{equation}
where $\Zc$ is approximated by its IS estimate. Under some mild assumptions,  $\widetilde{\Ic}$ is a consistent estimator of $\Ic$ \cite{Robert04,mcbook,elvira2021advances}.

\subsection{Characterization of the importance sampling estimators}

The variance of the estimator $\widehat{I}$ %and $\widetilde{I}$ 
is directly related to the discrepancy between $\normalized p(x)|{m(x)}|$ and the proposal $q(x)$ \cite{Robert04, kahn1953methods}. Specifically, the variance of the unbiased IS estimator of Eq. \eqref{eq_IS_estimator_unnorm} is given by
\begin{equation}
\Var_q\left(\widehat \Ic\right) = \frac{1}{N\Zc^2} \int_{\mathcal{X}} \frac{m^2(x)  p^2(x) }{q(x)}dx - \frac{\Ic^2}{N}.
\label{var_IS_unnorm}
\end{equation} 
The computation of this variance is usually intractable, and then it should be approximated. An accurate approximation of it is crucial, to assess quantitatively the performance of the IS estimator, and to set adjustment strategies in its hyper-parameters (e.g., parameters of the proposal $q(x)$) to reduce this variance. 
%It can be shown that the optimal proposal $q^*$ that minimizes the variance is $q^*(\x) \propto |f(\x)|  \pi(\x)$. Intuitively, the performance is very bad when the targeted integrand $f(\x)  \pi(\x)$ and the distribution that generates the samples have a large mismatch.

Using the approach introduced in Section \ref{sec_compound}, we propose to construct bounds, for each three integral terms in the variance expression, namely
\begin{align}
&\underline{\Ic} \leq {\Ic} = \int_{\mathcal{X}} {m(x)} p(x) dx \leq \overline{\Ic} \label{eq:bndI1},\\
&\underline{\Zc} \leq {\Zc} = \int_{\mathcal{X}}  p(x) dx \leq \overline{\Zc}\label{eq:bndZ},\\
&\underline{\mathcal{J}} \leq {\mathcal{J}} \triangleq \int_{\mathcal{X}} \frac{m^2(x)  p^2(x) }{q(x)} dx \leq \overline{\mathcal{J}}\label{eq:bndI2}.
\end{align}
%In particular, in Eq. \eqref{var_IS_unnorm} we could majorize (is this the word?) the first function $U_1(\x) \geq  \frac{f^2(\x)  \pi^2(\x) }{q(\x)}  = f^2(\x)  w^2(\x)q(\x)d\x $, $\forall \x\in\mathcal{X}$ 
%\emilie{(I dont understand this equation)}
%\victor{What exactly? $w$ is the weighting function given by Eq. \eqref{is_weights_static} (if you see it as a function)} 
%and minorize the second function $U_2(\x) \leq f(\x)\normalized \pi(\x)$, $\forall \x\in\mathcal{X}$. We could then find an upper found of \eqref{var_IS_unnorm} when $Z$ is known by integrating
%We propose to compute accurate values for the bounds in \eqref{eq:bndI1}, \eqref{eq:bndI2}, \eqref{eq:bndZ}, with our proposed approach. 
Namely, we will apply Alg.~\ref{algo:tangency}, on the three above problems, that take the general form \eqref{integral_interest}, with different roles for functions $\pi(x) = \exp(-\phi(x))$ and $f(x)$, identified as follows:
\begin{itemize}
    \item Computation of the bounds in \eqref{eq:bndI1}, using $\pi \equiv p$ and $f \equiv m$;
     \item Computation of the bounds in \eqref{eq:bndZ}, using $\pi \equiv p$ and $f \equiv 1$;
        \item Computation of the bounds in \eqref{eq:bndI2}, using $\pi \equiv p^2/q$ and $f \equiv m^2$.
\end{itemize}
Finally, combining the obtained bounds, the IS estimator variance will be approximated as:
% \begin{equation}
% U = \frac{1}{NZ^2} \int_{\mathcal{X}}  U_1(\x)d\x - \frac{1}{N} \left(  \int_{\mathcal{X}} U_2(\x)d\x \right)^2 \geq \Var_q\left(\widehat I\right)
% \label{var_IS_unnorm}
% \end{equation}
\begin{equation}
\frac{1}{N\overline{\Zc}^2}  \underline{\mathcal{J}} - \frac{\overline{\Ic}^2}{N} \leq \Var_q\left(\widehat I\right) \leq \frac{1}{N\underline{\Zc}^2}  \overline{\mathcal{J}} - \frac{\underline{\Ic}^2}{N}.
\end{equation}

To illustrate and discuss the practical implementation, and assess the performance of this approach, we detail in the next section a practical example for the target function $\pi$. We also present a numerical comparison with a standard numerical integrator, as well as with the state-of-the-art method from \cite{evans1997approximation} for integral approximation.
%\emilie{(Are you sure of the second term ? What have you done with the squared in $I^2/N$ ?)}\victor{Typo. Now, I have corrected it, but my comment below still applies . It might be true for non-negative $f(\x)$. For the moment, we can focused on $U_1(\x)$.}
%\cred{I think the minorization I have written of the second term is not correct, since the integral can be negative, and $f(\x)$ can take both negative and positive values. It would be easier to upper bound this term by computing $\int |f(\x)|\normalized \pi(\x)\geq \int f(\x)\normalized \pi(\x)$, but it is useless to our purposes. Maybe, we can just upper bound the first integral (although it would be less tight) since  $\frac{I^2}{N}$ is obviously always positive.}

% \subsubsection{Self-normalized estimator}
% The variance of the SNIS is a nightmare (even more!), since it includes a ratio of dependent r.v.'s as
% \begin{equation}
% \Var_q\left(\widetilde I\right) = Z^2 \Var \left( \frac{\widehat I}{\widehat Z} \right).
% \end{equation}
% Usually this ratio is approximated (then, the bound is probably lost) by the delta method. I will not give you more details for the moment, but in the future, giving a bound for $\widetilde I$ would be also of great interest.
% \victor{Any idea for bounding that variance, it is a hard problem I think?}

%% file: 06a_experimental_numapp.tex
\subsection{Target function and proposal}

We now consider an experimental study of the Bayesian logistic regression model, estimating the variance of the IS estimator as   depicted in Sec.~\ref{sec_IS}. Given $J\geq1$ data, associated to $(y_j)_{1 \leq j \leq J} \in \left\{-1,1\right\}$ class labels, and $(w_j)_{1 \leq j \leq J} \in \mathbb{R}$ features, the target function is constructed as the product of the Bayesian logistic likelihood with a Gaussian prior with zero mean and variance $s^2$ (with $s>0$). The unnormarlized target function is %\emilie{Can you add some sentences to motivate this density choice ?}
\begin{equation}
(\forall x \in \bR) \quad
p(x) = A \exp(-x^2/(2 s^2)) \times \prod_{j=1}^J \exp\left(-\log( 1 + \exp(y_j (w_j x)) )\right),
\label{eq:pi}
\end{equation}
with $A > 0$. In all numerical experiments, we set $s=1.2$, $J=10$, and $A=1$. %$J = 20$ data, 

%We will fix also $(\forall \, x \in \mathbb{R})\, \, f(x) = x^2$. 
 %\newline

%This approach is particularly beneficial for avoiding overfitting, especially in situations where the number of dimensions exceeds the number of observations. By utilizing a Gaussian prior, we effectively implement Ridge Regression, which penalizes large values of the coefficients.

Following the notations in \eqref{integral_interest}, we denote 
\begin{equation}
\label{eq:phi_experiment}
(\forall x \in \mathbb{R}) \quad \phi(x) = \frac{x^2}{2 s^2} + \sum_{j=1}^J \log\left(1 + \exp(y_j w_j x)\right) - \log(A),
\end{equation}
so that, for every $x \in \mathbb{R}$, $p(x) = \exp(-\phi(x))$. Using \cite{Bouchard2008,chouzenoux:hal-03793623}, we deduce that function $\phi$ in \eqref{eq:phi_experiment} satisfies Assumptions~\ref{assump1} and \ref{assump2} on $\mathbb{R}$, with the following expressions for $\beta(t)$ and $\nu(t)$:
\begin{equation}
\label{eq:beta_nu_pi}
(\forall t \in \mathbb{R}) \quad
\begin{cases}                      
\beta(t) = \sum_{j=1}^J (y_j w_j)^2 \psi(y_j w_j t) + \frac{1}{s^2},\\
\nu(t) = \frac{1}{s^2},
\end{cases}
\end{equation}
where the function $\psi: \mathbb{R} \to \left(0, \frac{1}{4}\right]$ is defined as
\begin{equation}
\label{eq:psi_def}
(\forall t \in \mathbb{R}) \quad
\psi(t) = 
\begin{cases}                      
\frac{1}{4} & \text{if} \; t = 0,\\
\frac{1}{t} \left(\frac{1}{1 + \exp(-t)} - \frac{1}{2}\right) & \text{otherwise.}
\end{cases}
\end{equation}
%\end{lemma}

We furthermore define, for the proposal of the IS sampler, a Gaussian proposal distribution $q(x)$ with mean equals $\mu = 2$ and \cblue{standard deviation} $\theta>0$. We set $m(x) = x^2$. \cblue{Figure~\ref{fig:pi_vs_piq} illustrates the shape of $p(x)$, $m(x)p(x)$ and $m(x)^2p(x)^2/q(x)$, using $\theta = 1.5$ for the proposal standard deviation. Function $p$ is unimodal, while the two other functions are multimodal.}

\begin{figure}[H]
    \centering
    \begin{tabular}{@{}c@{}c@{}c@{}}
    \includegraphics[width=5cm]{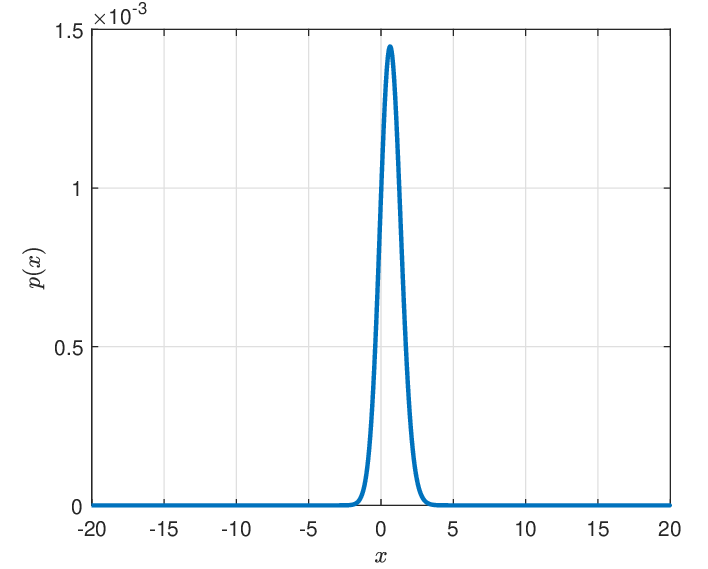} &
 \includegraphics[width=5cm]{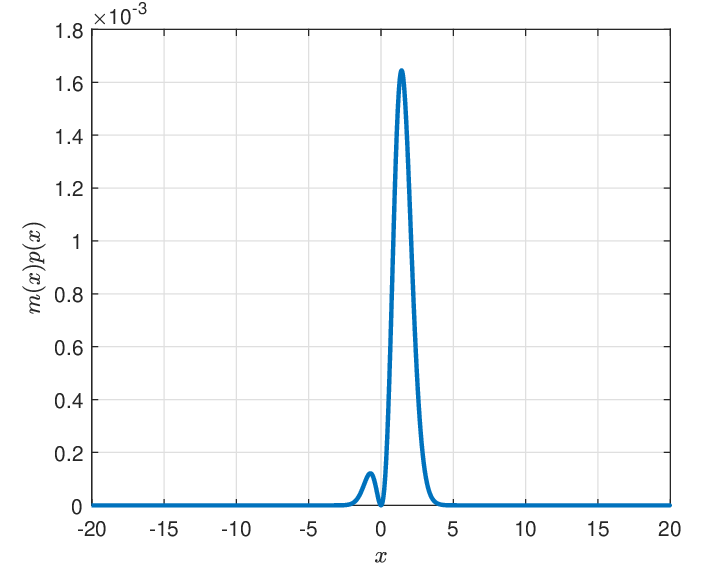}
 &
 \includegraphics[width=5cm]{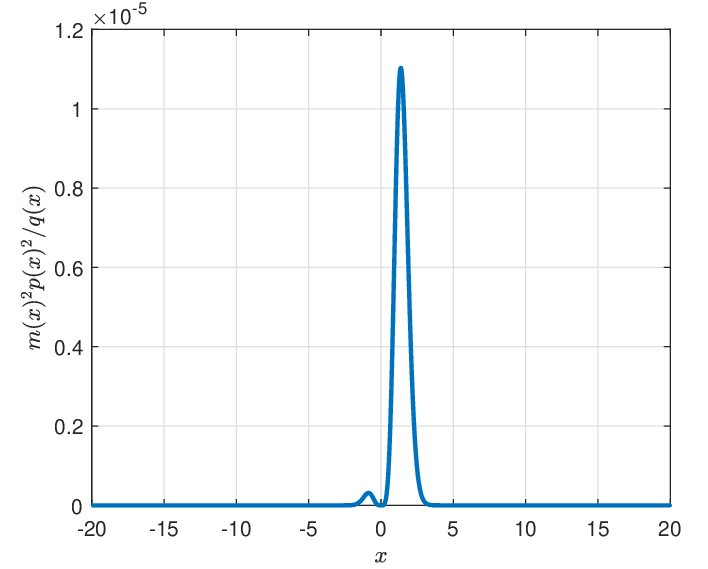}
 \end{tabular}
    \caption{\cblue{Plots of functions $x \to p(x)$, $x \to m(x)p(x)$ and $x \to m(x)^2p(x)^2/q(x)$ over interval $[-20,20]$.}}
\label{fig:pi_vs_piq}
\end{figure}

The final goal is to compute bounds on the variance estimator \eqref{var_IS_unnorm}. We run Alg.~\ref{algo:tangency} to construct bounds for the three intermediary quantities \eqref{eq:bndI1},\eqref{eq:bndZ}, and \eqref{eq:bndI2} \cblue{(i.e., the integrals over $\mathbb{R}$ of the three functions displayed in Fig.~\ref{fig:pi_vs_piq})}, leading to the following bounds for the integral \eqref{var_IS_unnorm}: 
\begin{equation}
\label{eq:variancebounds}
\underline{V} = \frac{1}{N \overline{\mathcal{Z}}^2} \, \underline{\mathcal{J}} - \frac{\overline{\mathcal{I}}^2}{N}, \quad
\overline{V} = \frac{1}{N \underline{\mathcal{Z}}^2} \, \overline{\mathcal{J}} - \frac{\underline{\mathcal{I}}^2}{N}.
\end{equation}
The bounds are compared to empirical variance estimation, obtained by running the IS simulator on $N_{\text{runs}}$ runs. Otherwise stated, we set $\theta = 1.5$, $N = 20$, and $N_{\text{runs}} = 10^6$.
\cblue{The complete derivations for building Gaussian majorant and minorant for the ratio $x \to p(x)^2/q(x)$ are provided in Appendix~B. The latter computation raises the condition $(\forall t \in \mathbb{R}) \; (\underline{\sigma}(t)^2,\overline{\sigma}(t)^2)\in (0,2\theta^2)$ limiting the range of the proposal standard deviation $\theta$, that is satisfied empirically in our experiences. This condition can be interpreted as a control on the relative tails of $p^2$ and $q$, ensuring the finiteness of the moments of the ratio $p^2/q$.}

%In our experiments, the variance condition required to derive Gaussian bounds for $p(x)^2/q(x)$ is satisfied. Indeed, the Gaussian upper bound of $p$ has variance $\overline{\sigma}(t)^2 = 1/\nu(t)=s^2$, while the Gaussian lower bound has variance $\underline{\sigma}(t)^2 = 1/\beta(t)\leq s^2$ since $\beta(t)\geq 1/s^2$. Hence, a sufficient condition is $s^2<2\theta$, where $\theta$ denotes the variance of the Gaussian proposal $q$. In our setting, $s=1.2$ and $\theta=1.5$, so that $s^2=1.44<3=2\theta$. Therefore, the variance assumptions used to bound $p(x)^2/q(x)$ are fulfilled.

% {\color{blue}To bound $\mathcal{J}$, we derive Gaussian bounds for $p(x)^2/q(x)$; see Appendix~\ref{app:bounds_p^2/q}.}
% for a range of $\theta$, such that all the involved integrals are well-defined

%These observations confirm both theoretically and empirically that the bounding accuracy increases with the number of tangency points, while the compact size—and thus the number of points—varies depending on the shape of the density.

%% file: 06b_description_experiments.tex
\subsection{Practical set-up}
\label{sec:practical}

\noindent\textbf{Dyadic sets}  
 As described in Sec.~\ref{sec:algopractical}, the pool of available tangency points, for Alg.~\ref{algo:tangency} is constructed by the procedure in Alg.~\ref{algo:dyadic_set}, starting from an interval defined in \eqref{eq_ab}. Setting the prescribed tail level to $\gamma = 10^{-6}$, \eqref{eq_ab} yields $[a_p,b_p] = [-5.8655,5.8744]$ when $\pi \equiv p$ (for \eqref{eq:bndI1} and \eqref{eq:bndZ}), and $[a_q,b_q] = [-5.9681,4.0988]$ when $\pi \equiv p^2/q$ (for \eqref{eq:bndI2}). Moreover, in all experiments, we set {$t_1 = 1$ and the density parameter for dyadic constructions $\ell = 10000$, which leads to $M_p = 6145$ (resp. $M_q = 5633$) candidate points in the dyadic sets obtained via Alg.~\ref{algo:dyadic_set}}

We structure the numerical study in three parts:

\begin{itemize}
  \item \textbf{Validation of the obtained bounds.}
  We first assess in Sec.~\ref{eq:expe1}, the accuracy of the upper/lower envelopes produced by Algorithm~\ref{algo:tangency}, by comparing the midpoint
  $$
    \frac{\overline{\mathcal{X}}+\underline{\mathcal{X}}}{2}
  $$
  against MATLAB's \texttt{integral()} for each target integral $\mathcal{X}\in\{\mathcal{Z},\mathcal{J},\mathcal{I}\}$. We report absolute and relative errors, together with the number of tangency points required.

  \item \textbf{Validation of the variance estimator.}
  Next, in Sec.~\ref{sec:expe2}, we compute the bounds on the variance of the Importance Sampling estimator for a range of values for the proposal variance $\theta$. This illustrates the validity of the bounds, compared to its empirical computation, and their applicative use to optimally set the proposal hyper-parameter $\theta$. 

  \item \textbf{Comparison with state-of-the-art.}  Finally, in Sec.~\ref{sec:expe3}, we compare our results with the envelope-based approximation of Evans~\emph{et al.} \cite{evans1997approximation}, in the case of the evaluation of bounds for the integrals $\mathcal{Z}$ and $\mathcal{I}$, using comparable computational budgets (same number of tangency points / evaluations), and discuss accuracy–cost trade-offs.
\end{itemize}

\subsubsection{Numerical validation of bounds}
\label{eq:expe1}

This section evaluates the tightness and accuracy of the envelope bounds produced by Algorithm~\ref{algo:tangency}. First, we report, in Tables~\ref{tab:I_results_alg4}, \ref{tab:J_results_alg4}, \ref{tab:Z_results_alg4}, the upper and lower bounds obtained when we stop as soon as a prescribed relative precision condition $G^{(n)}/\overline{\Ic}^{(n)} \leq \tau$, for some $\tau>0$, is reached. We also report the number of iterations $n_{\text{stop}}$ when reaching this stopping criterion. Note that, since $\mathbf{T}^{(1)}$ is reduced to the singleton $t_1$, we have $n_{\text{stop}}= \text{Card}(\mathbf{T}^{(n)})$, i.e., the number of evaluated tangency points. Finally, we display the difference between the estimated upper and lower bounds, as well as the absolute difference between the midpoint
$$
\frac{\overline{\mathcal{X}}+\underline{\mathcal{X}}}{2},
$$ with $\mathcal{X}\in\{\mathcal{I},\mathcal{J},\mathcal{Z}\}$, and
the MATLAB’s \texttt{integral()} result. As expected, the accuracy of the bounds increases when refining the stopping criterion, as this yields exploring more tangency points. The results are close to the ones obtained with the numerical integrator, the main difference is that our construction ensures that the bounds envelop the ground truth integral. 

%  

% for each target integral $\mathcal{X}\in\{\mathcal{I},\mathcal{J},\mathcal{Z}\}$, the upper and lower bounds obtained when we stop as soon as a prescribed \emph{relative precision} is reached; we compare the 

Second, we visualize on Figure~\ref{fig:Z_algo2} the convergence of the algorithm, until the relative-precision target $\tau = 10^{-4}$ is met. 

\sisetup{round-mode = places, round-precision = 3, round-integer-to-decimal}
\sisetup{detect-weight, group-minimum-digits = 6, mode=text}

% Requires: \usepackage{siunitx,booktabs}
\sisetup{round-mode = places, round-precision = 3, round-integer-to-decimal}
\sisetup{detect-weight, group-minimum-digits = 6, mode=text}

\begin{table}[H]
\centering
\footnotesize
\setlength{\tabcolsep}{3pt}
\begin{tabular}{@{}cccccc@{}}
\toprule
\multicolumn{6}{c}{Integral $\mathcal{I}$} \\
\midrule
$\tau$ & $n_{\text{stop}}$ &
$\overline{\mathcal{I}}$ & $\underline{\mathcal{I}}$ &
$\overline{\mathcal{I}}-\underline{\mathcal{I}}$ & Diff. vs MATLAB \\
\midrule
\num{1e-2} & 11  & \num{2.663e-3} & \num{2.638e-3} & \num{2.568e-5} & \num{9.271e-6} \\
\num{1e-3} & 34 & \num{2.643e-3} & \num{2.641e-3} & \num{2.605e-6} & \num{6.888e-7} \\
\num{1e-4} & 104 & \num{2.641e-3} & \num{2.641e-3} & \num{2.622e-7} & \num{7.237e-8} \\
\bottomrule
\end{tabular}
\caption{Obtained bounds for $\mathcal{I}$. The value computed with MATLAB's \texttt{integral()} is \num{2.641e-3}.}
\label{tab:I_results_alg4}
\end{table}

% Requires: \usepackage{siunitx,booktabs}
\sisetup{round-mode = places, round-precision = 3, round-integer-to-decimal}
\sisetup{detect-weight, group-minimum-digits = 6, mode=text}

\begin{table}[H]
\centering
\footnotesize
\setlength{\tabcolsep}{3pt}
\begin{tabular}{@{}cccccc@{}}
\toprule
\multicolumn{6}{c}{Integral $\mathcal{J}$} \\
\midrule
$\tau$ & $n_{\text{stop}}$&
$\overline{\mathcal{J}}$ & $\underline{\mathcal{J}}$ &
$\overline{\mathcal{J}}-\underline{\mathcal{J}}$ & Diff. vs MATLAB \\
\midrule
\num{1e-2} & 12  & \num{1.27083088954454e-05} & \num{1.25829342756408e-05} & \num{1.25374619804558e-07} & \num{4.11079293789589e-08} \\
\num{1e-3} & 36  & \num{1.26145835585724e-05} & \num{1.26025278716543e-05} & \num{1.20556869180853e-08} & \num{4.04205894922766e-09} \\
\num{1e-4} & 112  & \num{1.26055447734102e-05} & \num{1.26042887797313e-05} & \num{1.25599367892594e-09} & \num{4.03120406643914e-10} \\
\bottomrule
\end{tabular}
\caption{Obtained bounds for $\mathcal{J}$. The value computed with MATLAB's \texttt{integral()} is \num{1.260e-5}.}%Obtained bounds for $\mathcal{I}$. The value computed with MATLAB's \texttt{integral()} is \num{2.641e-3}.
\label{tab:J_results_alg4}
\end{table}

% Requires: \usepackage{siunitx,booktabs}
\sisetup{round-mode = places, round-precision = 3, round-integer-to-decimal}
\sisetup{detect-weight, group-minimum-digits = 6, mode=text}

\begin{table}[H]
\centering
\footnotesize
\setlength{\tabcolsep}{3pt}
\begin{tabular}{@{}cccccc@{}}
\toprule
\multicolumn{6}{c}{Integral $\mathcal{Z}$} \\
\midrule
$\tau$ & $n_{\text{stop}}$ &
$\overline{\mathcal{Z}}$ & $\underline{\mathcal{Z}}$ &
$\overline{\mathcal{Z}}-\underline{\mathcal{Z}}$ & Diff. vs MATLAB \\
\midrule
\num{1e-2} & 11  & \num{2.67031417564e-3} & \num{2.64478931216406e-3} & \num{2.55248634759460e-05} & \num{1.02697731105320e-05} \\
\num{1e-3} & 31  & \num{2.64954383679473e-3} & \num{2.64694189223841e-3} & \num{2.60194455631605e-06} & \num{9.60893725069148e-07} \\
\num{1e-4} & 101 & \num{2.64750273488882e-3} & \num{2.64723991394386e-3} & \num{2.62820944962189e-07} & \num{8.93536248407732e-08} \\
\bottomrule
\end{tabular}
\caption{Obtained bounds for $\mathcal{Z}$. The value computed with MATLAB's \texttt{integral()} is \num{2.647e-3}.}
\label{tab:Z_results_alg4}
\end{table}

\begin{figure}[H]
\centering
\begin{tabular}{@{}c@{}c@{}c@{}}
\includegraphics[width=5cm]{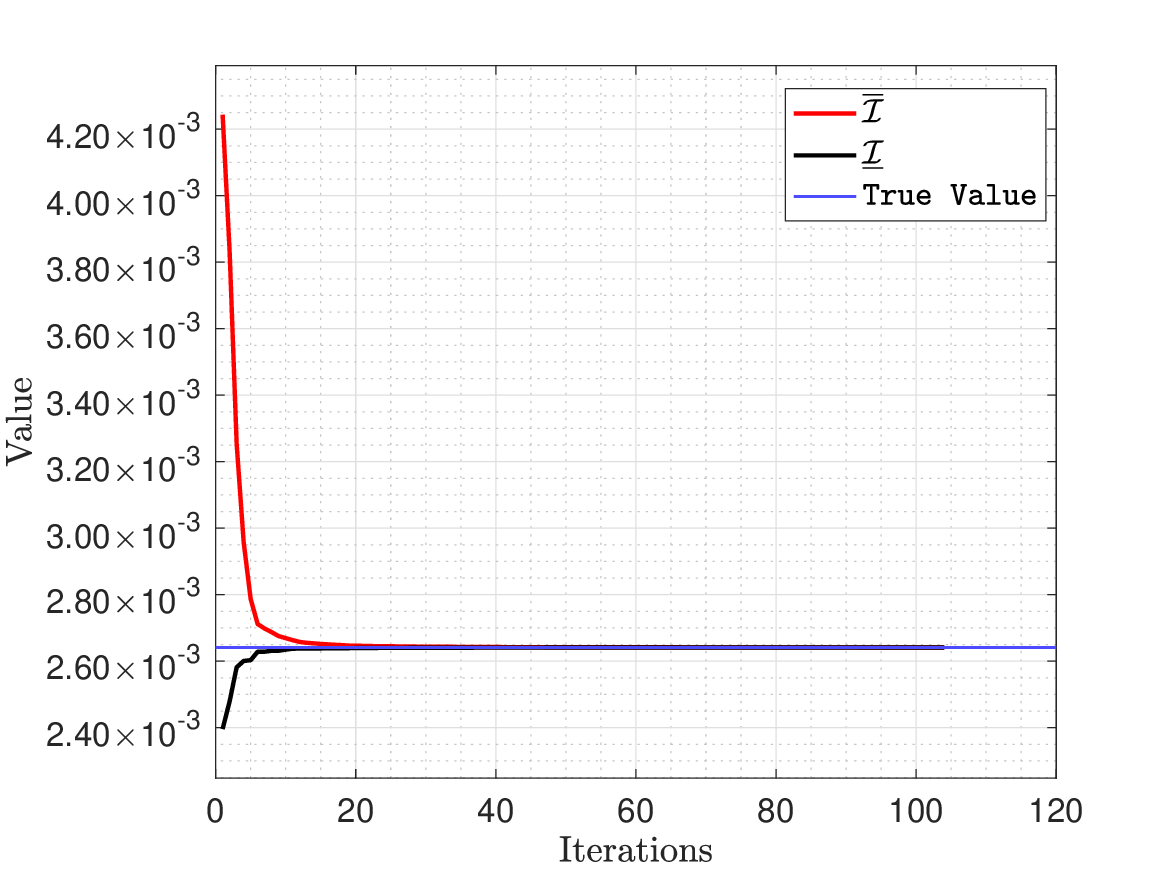} & 
\includegraphics[width=5cm]{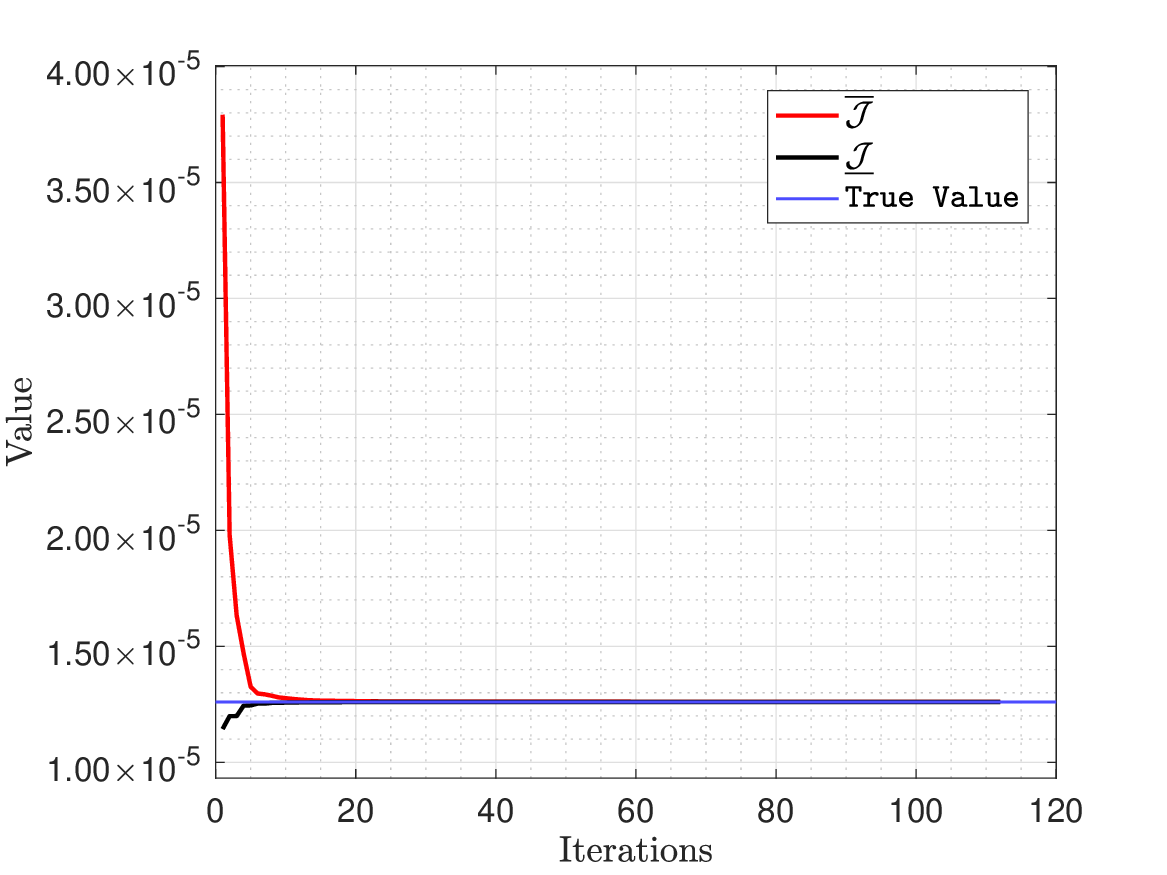}
& 
\includegraphics[width=5cm]{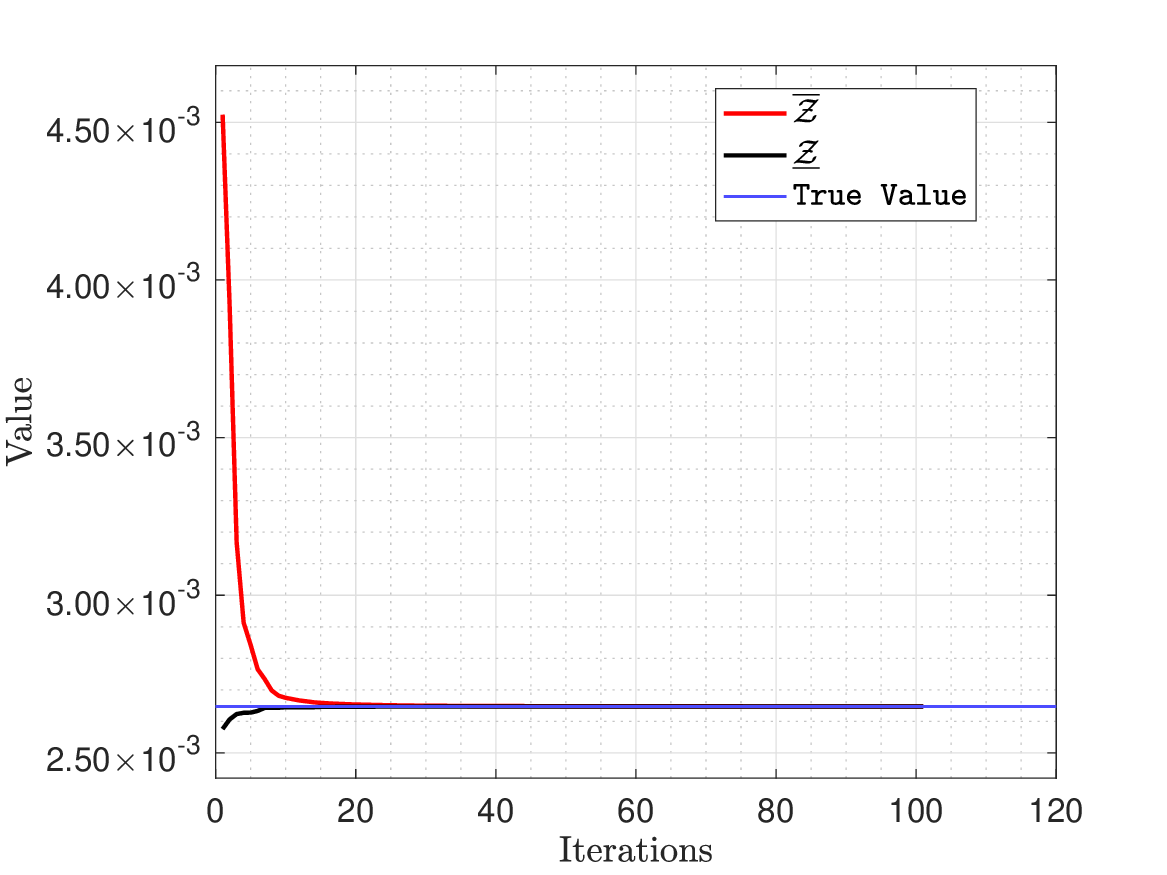} 
\end{tabular}
\caption{Evolution of the bounds along iterations $n$ of Alg.~\ref{algo:tangency}, until a relative precision of $\tau = 10^{-4}$ is reached, using Algorithm~\ref{algo:tangency}.}
\label{fig:Z_algo2}
\end{figure}

% \begin{figure}[H]
% \centering
% \includegraphics[width=5cm]{figures/J_algo2_new.eps}
% \caption{Evolution of $\underline{\mathcal{J}}$ and $\overline{\mathcal{J}}$ as tangency points are added until a relative precision of $10^{-4}$ is reached, using Algorithm~\ref{algo:two}.}
% \label{fig:J_algo2}
% \end{figure}

% \begin{figure}[H]
% \centering
% \includegraphics[width=5cm]{figures/I_algo2_new.eps}
% \caption{Evolution of $\underline{\mathcal{I}}$ and $\overline{\mathcal{I}}$ as tangency points are added until a relative precision of $10^{-4}$ is reached, using Algorithm~\ref{algo:two}.}
% \label{fig:I_algo2}
% \end{figure}

Across Tables~\ref{tab:I_results_alg4}–\ref{tab:Z_results_alg4}, Algorithm~\ref{algo:two} consistently achieves the prescribed relative precisions with a modest number of tangency points (e.g., $\tau =$ \num{1e-4} with 104 points for $\mathcal{I}$, 112 for $\mathcal{J}$, and 101 for $\mathcal{Z}$), while keeping both the absolute gap between bounds and the deviation from MATLAB extremely small (down to the order of $10^{-8}$).  
The convergence plots in Figure~\ref{fig:Z_algo2} confirm the fast monotone contraction of the envelopes and their stabilization around the reference as iterations increase, illustrating that the adaptive placement of tangency points effectively targets the regions that most reduce the bound gap. We observe faster convergence for the lower bound. This is expected, as it relies on the accurate, and adaptive, curvature $\beta(t)$.

\subsection{IS variance estimation}
\label{sec:expe2}

We present in Table~\ref{tab:variances} the obtained bounds on the variance of the IS method, computed as in \eqref{eq:variancebounds}, where the three intermediary bounds are computed with the proposed Alg.~\ref{algo:tangency} with relative precision $\tau = 10^{-4}$. We compare these bounds with the empirical variance $V_e$ of the unbiased IS estimator. Namely, we evaluate~\eqref{eq_IS_estimator_unnorm}, for $N_{\text{runs}}$ independent runs of IS. We use  $N=20$ samples, proposal \cblue{standard deviation} $\theta = 1.5$, and numerical intensive integration to compute the normalization constant of the target. This leads to $(\widehat{I}_\iota)_{\iota \in [1, \dots, N_{\text{runs}}]}$, from which we calculate the variance $V_e$. We also calculate the theoretical value $\text{Var}_q(\widehat{I})$ of the unbiased IS variance, whose expression is given in \eqref{var_IS_unnorm}, computed with the numerical integrator from Matlab. 

\sisetup{round-mode = places, round-precision = 3, round-integer-to-decimal}
\sisetup{detect-weight, group-minimum-digits = 4, mode=text}

\begin{table}[H]
\centering
\begin{tabular}{|c|c|c||c|c||c|}
\hline
  $V_e \,(N_{\text{runs}} = 10^2)$ & $V_e \,(N_{\text{runs}} = 10^4)$ & $V_e \,(N_{\text{runs}} = 10^6)$ & \textbf{$\underline{V}$} & \textbf{$\overline{V}$} & $\text{Var}_q(\widehat{I})$ \\
\hline
%\textbf{Empirical Variance of $\widehat{I}$}
 % & \num{4.010e-2} & \num{4.010e-2} & \num{4.010e-2} \\
%\hline
 \num{4.277e-2} & \num{4.005e-2} & \num{4.018e-2} & \num{4.013e-2} & \num{4.018e-2} & \num{4.016e-2}\\
\hline
\end{tabular}
\caption{Empirical variance $V_e$ of $(\widehat{I}_\iota)_{\iota \in [1, \dots, N_{\text{runs}}]}$, for different $N_{\text{runs}}$, estimated bounds $(\underline{V},\overline{V})$ by proposed approach, and theoretical value $\text{Var}_q(\widehat{I})$ evaluated by intensive numerical integration in \eqref{var_IS_unnorm} .}
\label{tab:variances}
\end{table}

As expected, as the number of runs $N_{\text{runs}}$ increases, $V_e$ converges towards the theoretical value \eqref{var_IS_unnorm} of the unbiased IS variance. The bounds computed with our method are valid and have good precision, while neither requiring any MC runs, nor intensive numerical integration.

Now, we run our approach to compute $(\underline{V},\overline{V})$, for different values of the proposal variance $\theta$, dataset size $J$ and prior hyper-parameter $s$. This experiment aims at illustrating a key motivation behind our work, i.e., to find the optimal proposal hyper-parameter, minimizing the variance of the estimator. The results are displayed in Figure~\ref{fig:var_vs_theta_s1.2}(left), along with the empirical variance $V_e$ evaluated for $N_{\text{runs}} = 10^6$ runs. We also display in Figure~\ref{fig:var_vs_theta_s1.2}(right) the Mean Square Error (MSE) of the IS estimator to compute the second order moment. Namely, the MSE is defined as
$$
\text{MSE} = \frac{1}{N_s} \sum_{\iota=1}^{N_s} \left( \frac{\mathcal{I}}{\mathcal{Z}} - \widehat{I}_\iota \right)^2,
$$
where $\mathcal{I}$ and $\mathcal{Z}$ are computed by numerical integration. Both plots are the same, as the variance and the MSE of the estimator are identical. This illustrates that the IS estimator is unbiased. An important remark is that the range $\theta$, used in all plots, is set so as to guarantee the well definition of all involved integrals.  

\begin{figure}[H]
\centering
\begin{tabular}{@{}c@{}c@{}}
\includegraphics[width=5cm]{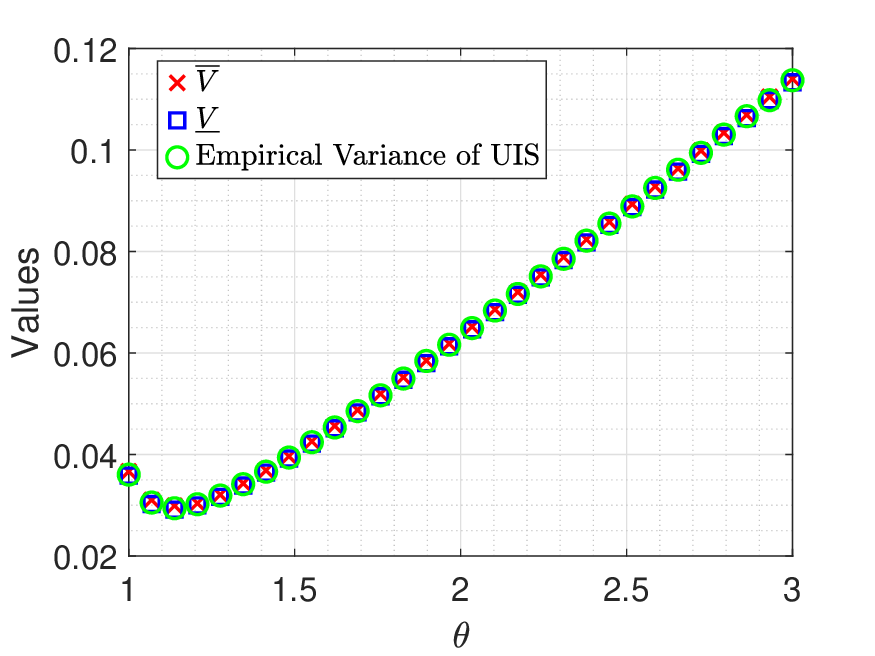} &
\includegraphics[width=5cm]{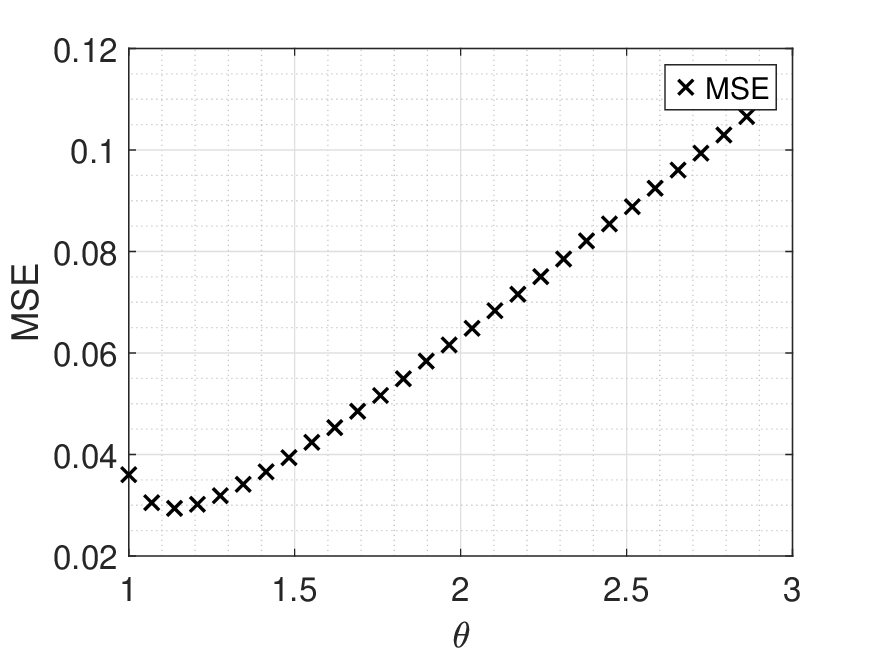}\\
\includegraphics[width=5cm]{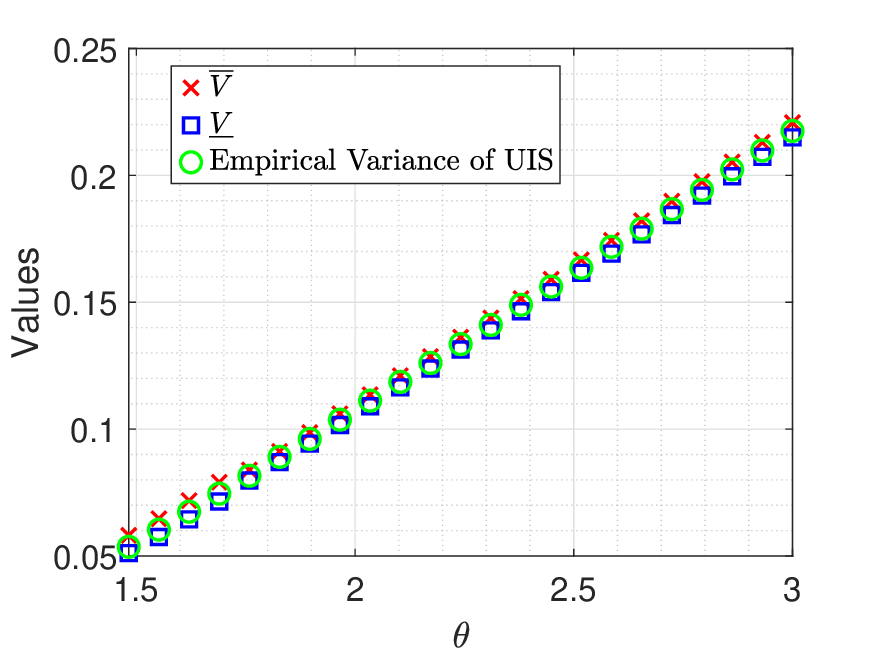} & \includegraphics[width=5cm]{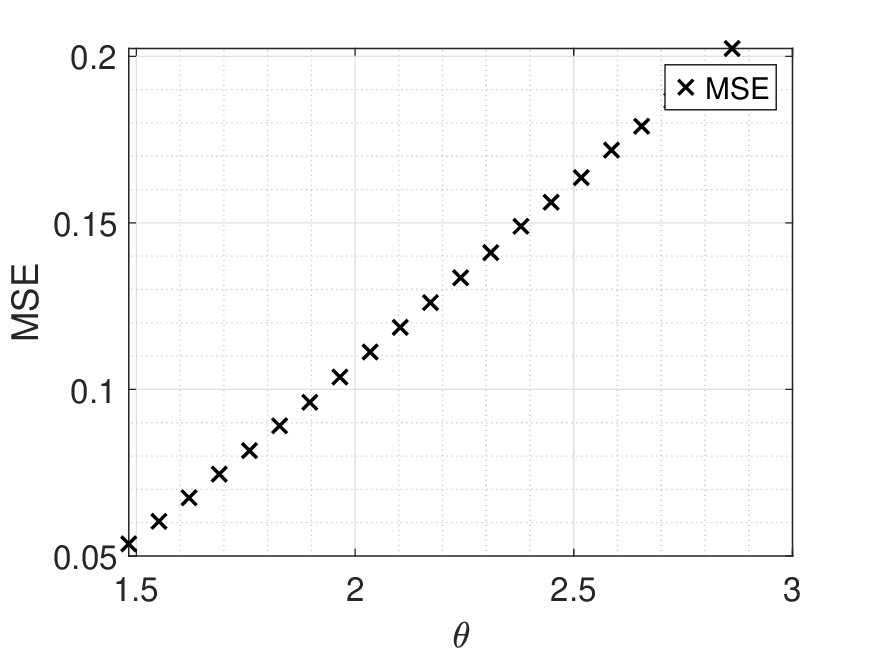}\\
\includegraphics[width=5cm]{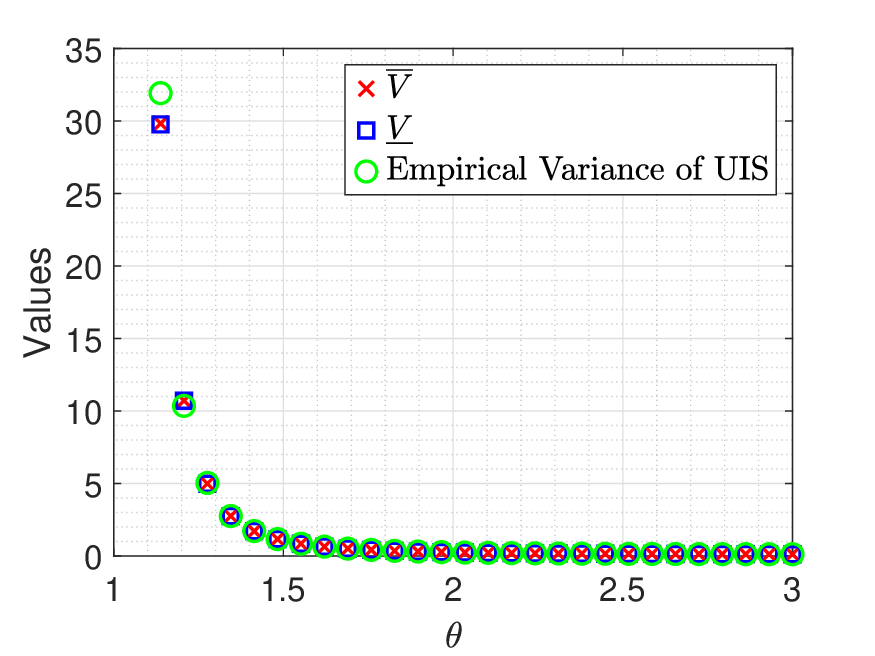} & \includegraphics[width=5cm]{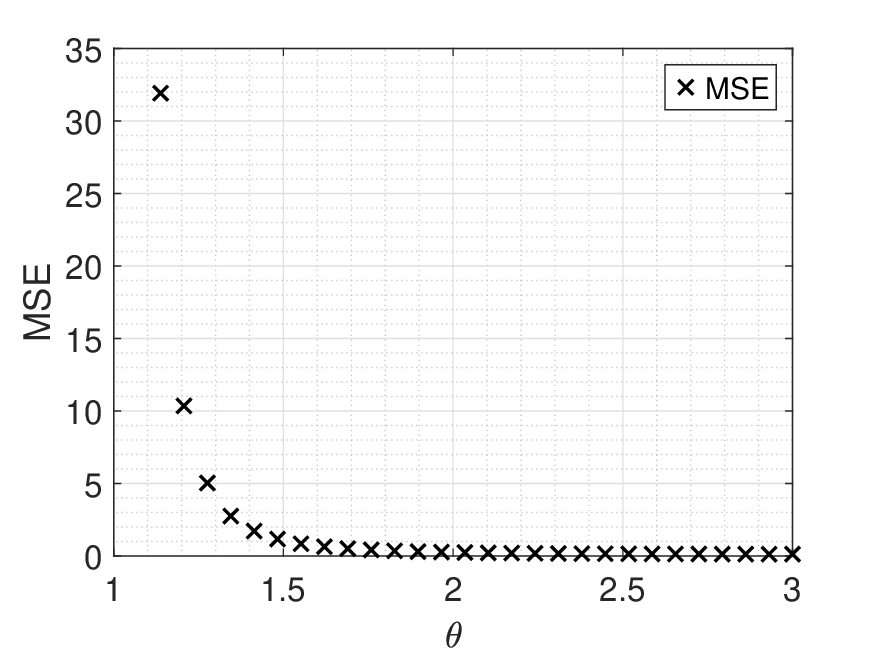}
\end{tabular}
\caption{Obtained bounds (red, blue) and empirical variance $V_e$ with $N_{\text{runs}} = 10^6$ (left) and MSE on second order moment (right) for different values of $\theta$. Top: $J=10$, $s=1.2$, middle: $J=10, s = 2$, bottom: $J=2$, $s=1.2$.}
\label{fig:var_vs_theta_s1.2}
\end{figure}

We can see in Figure~\ref{fig:var_vs_theta_s1.2} that, depending on the values of $(J,s)$ (resp., dataset size, and prior hyper-parameter), the plots have different aspects. In Figure~\ref{fig:var_vs_theta_s1.2}(top), a minimum of the variance is reached,  approximately at $\theta = 1.14$, which guides to tune the proposal in this case as \cblue{a Gaussian distribution with mean $\mu=2$ and standard deviation $\theta = 1.14$}. In Figure~\ref{fig:var_vs_theta_s1.2}(middle), the best proposal is the smallest $\theta$ value, in the range of well-defined integrals. In contrast, in
Figure~\ref{fig:var_vs_theta_s1.2}(bottom), a large $\theta$ leads to a reduced variance. In all cases, the computed bounds are valid and accurate, and provide a direct feedback on the quality of the proposal hyper-parameter in to get a minimal variance for the estimator, without requiring expensive MC runs nor MSE computation (impossible without access to ground truth).

\subsection{Comparison with state-of-the-art}
\label{sec:expe3}
Numerous algorithms are available to approximate integrals. The major part of them are constructed thanks to an interpolation method. This consists in approximating a function by an easier one where the two function are equal in a finite number of anchor points, for example one can use as an approximating family the well known Lagrange interpolator polynomial. However, only very few allow to compute exact bounds on integrals. We choose to compare our method to the one in \cite{evans1997approximation} computing bounds relying on polynomial piece-wise envelopes. Their main result is reminded below. 
\begin{lemma}{\cite[Lem.2]{evans1997approximation}} 
Let $d \geq 0$. If $f^{(d)}$ is concave on $[a,b]$ then $(\forall x \in [a,b])$
    \begin{equation}
\sum_{k=0}^{d} \frac{f^{(k)}(a)}{(k+1)!}(b-a)^{k+1} + \frac{f^{(d)}(b) - f^{(d)}(a)}{b-a}\frac{(b-a)^{d+2}}{(d+2)!}
\leq \int_a^b f(x) \, dx \leq \sum_{k=0}^{d+1} \frac{f^{(k)}(a)}{(k+1)!}(b-a)^{k+1}.
\end{equation}
This inequality is reverse if $f^{(d)}$ is convex.
\end{lemma}
Starting from this lemma, the authors propose a piece-wise (i.e., coumpounded) version of their approximation, by cutting the closed interval $[a,b]$ in $m$ sub-intervals, and derive a control of the precision of the obtained integral bounds, in \cite[Lem.3]{evans1997approximation}. In our practical implementation of this method, for a fair comparison, we use the same interval $[a,b]$ than the ones provided in Sec.~\ref{sec:practical} (obtained with \eqref{eq_ab}). We set a fixed number of iterations (i.e., cardinality of our pool of tangency points) for our method, namely $n=3$, $50$ or $100$. We set this same number for the coumpound points in Evans' method, and evaluate it with $d=0, 1, 2$ degree values. We only provide the results for $\mathcal{I}$ and $\mathcal{Z}$, as Evans' method complexity was too large, for handling the approximation of integral~$\mathcal{J}$. 

% They also derive the following convergence 

% \begin{lemma}{[Evans et al.]}
% \cite{evans1997approximation} in the paper is the lemma 3. \newline

% Let $E_m$  and  $E_m'$ denote the errors in the integrals  $\int_a^b u_m(x) \, dx$ and  $\int_a^b l_m(x) \, dx$,  respectively. Here,  $l_m$  represents the lower bound of the function  $f$ for a compound approximation with  $m$  points, while  $u_m$  represents the upper bound of  $f$ for the same compound with  $m $ points.
 
% Then, as $m \to \infty$, we have \\
% $
% E_m \sim \frac{1}{m^{(d+2)}} \left( f^{(d+1)}(b) - f^{(d+1)}(a) \right) \frac{(b-a)^{(d+2)}}{(d+3)!},
% $
% and similarly
% $
% E_m' \sim \frac{d+3}{m^{(d+2)}} \left( f^{(d+1)}(b) - f^{(d+1)}(a) \right) \frac{(b-a)^{(d+2)}}{(d+3)!}.
% $
    
% \end{lemma}.

% In the application of the Evan's method we will use the same interval $[a_\pi,b_\pi]$ than the one create by the matlab function $\texttt{norminv}$, used in our experiments, see (\ref{eq:intervals_pi_piq}). \\
% Here we will use this method to compute approximate $\mathcal{Z}$ and $\mathcal{{I}_1}$.

% To provide a focused comparison between our method and Evans' method, we present results for cases where the number of compound points is either close to 100 or minimal (3 points), for different values of \( n \). This allows us to directly compare the performance of the two methods under similar conditions.

\sisetup{round-mode = places, round-precision = 3, round-integer-to-decimal}
\sisetup{detect-weight, group-minimum-digits = 6, mode=text}

\begin{table}[H]
\centering
\small
\caption{Estimation of $\mathcal{Z}$ - Absolute and relative errors for proposed method and Evans', for different compound-point budget.}
\label{tab:comparison_Z_byCP}
\begin{tabular}{@{} l
                S[scientific-notation=true,retain-zero-exponent]
                S[scientific-notation=true,retain-zero-exponent] @{}}
\toprule
\multicolumn{3}{c}{Compound points = 3} \\
\midrule
Method  & {Absolute Error} & {Relative Error} \\
\midrule
Ours            & \num{1.94795802878871e-3} & \num{4.30600800297280e-1} \\
\hline
Evans (d=0)       & \num{1.351e-2}            & \num{1.910e+0} \\
Evans (d=1)      & \num{1.403e-2}            & \num{1.041e+0} \\
Evans (d=2)      & \num{2.772e-1}            & \num{2.475e+1} \\
\midrule
\multicolumn{3}{c}{Compound points = 50} \\
\midrule
Method  & {Absolute Error} & {Relative Error} \\
\midrule
Ours           & \num{1.18022352351647e-6} & \num{4.456539222218252e-4} \\
\hline
Evans (d=0)     & \num{6.8571e-5}           & \num{2.5571e-2} \\
Evans (d=1)      & \num{1.1823e-5}           & \num{4.4559e-3} \\
Evans (d=2)      & \num{4.2899e-7}           & \num{1.6204e-4} \\
\midrule
\multicolumn{3}{c}{Compound points = 100} \\
\midrule
Method & {Absolute Error} & {Relative Error} \\
\midrule
Ours           & \num{2.81552495653373e-7} & \num{1.06345713243234e-4} \\
\hline
Evans (d=0)    & \num{1.684e-5}            & \num{6.341e-3} \\
Evans (d=1)   & \num{1.433e-6}            & \num{5.413e-4} \\
Evans (d=2)     & \num{5.013e-9}            & \num{1.894e-6} \\
\bottomrule
\end{tabular}
\end{table}

% ===================== I =====================
\sisetup{round-mode = places, round-precision = 3, round-integer-to-decimal}
\begin{table}[H]
\centering
\small
\caption{Estimation of $\mathcal{I}$ - Absolute and relative errors for proposed method and Evans', for different compound-point budget.}
\label{tab:comparison_I_byCP}
\begin{tabular}{@{} l
                S[scientific-notation=true,retain-zero-exponent]
                S[scientific-notation=true,retain-zero-exponent] @{}}
\toprule
\multicolumn{3}{c}{Compound points = 3} \\
\midrule
Method  & {Absolute Error} & {Relative Error} \\
\midrule
Ours             & \num{1.84729047702401e-3} & \num{4.35333153035048e-1} \\
\hline
Evans (d=0)      & \num{1.0422e-2}           & \num{1.9543} \\
Evans (d=1)      & \num{8.8859e-3}           & \num{1.0551} \\
Evans (d=2)      & \num{3.6709e-1}           & \num{8.7630} \\
\midrule
\multicolumn{3}{c}{Compound points = 50} \\
\midrule
Method  & {Absolute Error} & {Relative Error} \\
\midrule
Ours           & \num{1.20745616811590e-6} & \num{4.56999292605061e-4} \\
\hline
Evans (d=0)    & \num{1.0606e-4}           & \num{3.9356e-2} \\
Evans (d=1)    & \num{2.2914e-5}           & \num{8.6380e-3} \\
Evans (d=2)    & \num{2.5161e-6}           & \num{9.5219e-4} \\
\midrule
\multicolumn{3}{c}{Compound points = 100} \\
\midrule
Method  & {Absolute Error} & {Relative Error} \\
\midrule
Ours          & \num{2.91645044722417e-7} & \num{1.10412927605317e-4} \\
\hline
Evans (d=0)    & \num{1.6487e-7}           & \num{6.242e-5} \\
Evans (d=1)   & \num{2.8051e-6}           & \num{1.061e-3} \\
Evans (d=2)     & \num{1.6487e-7}           & \num{6.2421e-5} \\
\bottomrule
\end{tabular}
\end{table}

Tables~\ref{tab:comparison_Z_byCP} and \ref{tab:comparison_I_byCP} highlight the comparative performance of our method and the method proposed by Evans et al. Our method achieves better accuracy than Evans' method with order $d=0$, when using a low number of compound points (3 points). This difference is particularly notable in the relative error, where Evans' method demonstrates lower accuracy due to the limited subdivision of the integration interval. However, when the number of compound points increases to 100, Evans' method closes the gap and provides results comparable to ours, though our method remains slightly more efficient in most cases.

As the order of differentiation increases to $d = 1$, Evans' method begins to show significant improvements, particularly for larger numbers of compound points. With 100 points, their method achieves smaller absolute and relative errors than ours, indicating its ability to leverage higher-order derivatives effectively to refine the bounds of the integral. This trend continues for $d = 2$, where Evans' method significantly outperforms ours in both absolute and relative errors, particularly when the interval is finely subdivided. These results emphasize the strength of Evans' method in producing highly accurate bounds for higher-order approximations.

However, this improved accuracy comes at the cost of an increased computational complexity. Evans’ method relies on higher-order derivatives and the determination of inflection points within the integration interval. Finding these inflection points involves solving root-finding problems, often using algorithms such as bisection or binary search. The complexity of Evans’ method depends on two parameters: the number of compound points and the number of differentiations. In contrast, our method avoids high order derivative-based computations, offering an alternative approach to achieving accurate results. For this reason, we could not apply Evans’ method to $\mathcal{J}$, as differentiating the integrand twice was overly cumbersome.

%% file: 07_multidimensionnal_case.tex
In this section, we now explore the generalization of the proposed approach to the multi-dimensional case, with dimension $d \geq 1$. Let $f:\mathbb{R}^d\to\mathbb{R}$ and $\phi:\mathbb{R}^d\to\mathbb{R}$, $d\geq1$. We aim at approximating
\begin{equation}
\label{eq:integral_interest_multidim}
\mathcal{I} \equiv \int_{\mathbb{R}^d} f(\x)\,\pi(\x)\,d\x,
\quad \text{where} \quad (\forall \x \in \mathbb{R}^d) \quad
\pi(\x) = \exp\big(-\phi(\x)\big),
\end{equation}
assuming that the above integral is finite. 
\subsection{Construction of the minorant/majorant functions}
Let us first generalize the Assumptions \ref{assump1} and \ref{assump2}, as follows:
\begin{assumption}
\label{assump3}
For every $\t\in\mathbb{R}^d$, there exist two matrices $\A(\t),\B(\t)\in\mathbb{S}^d_{++}$ (i.e., symmetric positive definite in $\mathbb{R}^{d \times d}$) such that, for all $\x\in\mathbb{R}^d$,
\begin{align}
\label{eq:ass_multidim}
\phi(\x)&\le \phi(\t)+ \nabla \phi(\t)^\top (\x-\t)+\frac12 (\x-\t)^\top \A(\t)(\x-\t),\\
\phi(\x)&\ge \phi(\t)+\nabla \phi(\t)^\top (\x-\t)+\frac12 (\x-\t)^\top \B(\t)(\x-\t).
\end{align}
\end{assumption}
The Assumption \ref{assump3} is typically reached when $\phi$ is Lipschitz-smooth, and strongly convex. It then allows to build Gaussian approximations for $\pi$, on $\mathbb{R}^d$, as shown in the following Lemma, following similar derivations to those in Section \ref{sec_gaussbounds}. We use the notation $g(\cdot;\boldsymbol{\mu},\boldsymbol{\Sigma})$ to denote the density of the multi-dimensional normal distribution $\mathcal{N}(\boldsymbol{\mu},\boldsymbol{\Sigma})$ on $\mathbb{R}^d$.
\begin{lemma}
\label{lem:tangent_gaussians_multidim}
Assume that $\phi$ satisfies~Assumption~\ref{assump3}.
Let, for every $(\x,\t)\in(\mathbb{R}^d)^2$,
\begin{equation}
\label{eq:minorant_form_multidim}
\underline{b}(\x;\t)=\underline{C}(\t)\, g\!\big(\x;\underline{\u}(\t),\underline{\boldsymbol{\Sigma}}(\t)\big),
\end{equation}
with parameters
\begin{equation}
\label{eq:minorant_params_multidim}
\begin{cases}
\underline{\boldsymbol{\Sigma}}(\t)=\A(\t)^{-1},
\\
\underline{\u}(\t)=\t-\underline{\boldsymbol{\Sigma}}(\t)\nabla \phi(\t),
\\
\underline{C}(\t)=(2\pi)^{\frac d2}\,|\underline{\boldsymbol{\Sigma}}(\t)|^{\frac12}\,
\exp\!\left(-\phi(\t)+\frac12 \nabla\phi(\t)^\top \underline{\boldsymbol{\Sigma}}(\t)\nabla\phi(\t)\right),
\end{cases}
\end{equation}
and, for every $(\x,\t)\in(\mathbb{R}^d)^2$,
\begin{equation}
\label{eq:majorant_form_multidim}
\overline{b}(\x;\t)=\overline{C}(\t)\, g\!\big(\x;\overline{\u}(\t),\overline{\boldsymbol{\Sigma}}(\t)\big),
\end{equation}
with
\begin{equation}
\label{eq:majorant_params_multidim}
\begin{cases}
\overline{\boldsymbol{\Sigma}}(\t)=\B(\t)^{-1},
\\
\overline{\u}(\t)=\t-\overline{\boldsymbol{\Sigma}}(\t)\nabla \phi(\t),
\\
\overline{C}(\t)=(2\pi)^{\frac d2}\,|\overline{\boldsymbol{\Sigma}}(\t)|^{\frac12}\,
\exp\!\left(-\phi(\t)+\frac12 \nabla\phi(\t)^\top \overline{\boldsymbol{\Sigma}}(\t)\nabla\phi(\t)\right).
\end{cases}
\end{equation}
Then, for every $\t\in\mathbb{R}^d$, 
\begin{equation}
\label{eq:tgt_min_maj_multidim}
\forall \x\in\mathbb{R}^d,\qquad
\underline{b}(\x;\t)\le \pi(\x)\le \overline{b}(\x;\t),
\qquad
\underline{b}(\t;\t)=\pi(\t)=\overline{b}(\t;\t).
\end{equation}
\end{lemma}
\begin{proof}
The proof follows the same lines as Lemmas~\ref{lem:ass1} and~\ref{lem:ass2}. Starting from~the first inequality of Assumption~\ref{assump3}, we rewrite
\begin{multline}
\phi(\t)+\nabla \phi(\t)^\top (\x-\t)+\frac12 (\x-\t)^\top \A(\t)(\x-\t)
= \\
\phi(\t)-\frac12 \nabla\phi(\t)^\top \A(\t)^{-1}\nabla\phi(\t)
+\frac12 (\x-\underline{u}(\t))^\top \A(\t)(\x-\underline{u}(\t)),
\end{multline}
and then compose with exponential function to obtain~\eqref{eq:minorant_form_multidim}--\eqref{eq:minorant_params_multidim}. The expressions
\eqref{eq:majorant_form_multidim}--\eqref{eq:majorant_params_multidim} are obtained similarly using the second inequality in~Assumption~\ref{assump3}.
\end{proof}
\subsection{Construction of the envelopes and integral computations}
We now move to the construction of the upper and lower piecewise envelopes, aiming at generalizing the strategy from Section~\ref{sec_compound}. Consider the set $\T=[\t_1,\dots,\t_M]\in \mathbb{R}^{d \times M}$ concatenating the $d$-dimensional coordinates of $M \geq 2$ tangency points. Following Lemma~\ref{lem:tangent_gaussians_multidim}, for every $m$, we build Gaussian minorant $\underline{b}(\cdot;\t_m)$ and majorant $\overline{b}(\cdot;\t_m)$ approximations of $\pi$, tangent of $\pi$ at each $t_m$. The lower and upper envelopes of this set of $M$ approximations, are given by:
\begin{equation}
\label{eq:envelopes_multidim}
\underline{C}(\x;\T) = \max_{1\le m\le M}\,\underline{b}(\x;\t_m),
\qquad
\overline{C}(\x;\T) = \min_{1\le m\le M}\,\overline{b}(\x;\t_m).
\end{equation}
By construction, for all $\x\in\mathbb{R}^d$,
\begin{equation}
\label{eq:envelope_bounds_multidim}
\underline{C}(\x;\T)\le \pi(\x)\le \overline{C}(\x;\T).
\end{equation}
Functions $\underline{C}(\cdot;\T)$ and $\overline{C}(\cdot;\T)$ are continuous on $\mathbb{R}^d$ as maxima/minima of finitely many continuous functions, and they interpolate $\pi$ on $\T$, in the sense that
\begin{equation}
\label{eq:interpolation_multidim}
\forall \t \in \{\t_1,\ldots,\t_M\},\qquad
\underline{C}(\t;\T)=\pi(\t)=\overline{C}(\t;\T).
\end{equation}
Similarly as in Section~\ref{sec_compound}, the sign changes of function $f$ can be handled as follows. Define, for every $\x \in \mathbb{R}^d$, $f^+(\x)=\max\{f(\x),0\}$ and $f^-(\x)=\max\{-f(\x),0\}$ so that $f=f^+-f^-$ and $|f|=f^++f^-$. For any finite $\T\in\mathbb{R}^{d \times M}$, let us define
\begin{equation}
\label{eq:Id_bounds_multidim}
\underline{\mathcal{I}}(\T) = \int_{\mathbb{R}^d} f^+(\x)\,\underline{C}(\x;\T)\,d\x
-\int_{\mathbb{R}^d} f^-(\x)\,\underline{C}(\x;\T)\,d\x,
\end{equation}
and
\begin{equation}
\label{eq:Id_bar_bounds_multidim}
\overline{\mathcal{I}}(\T) = \int_{\mathbb{R}^d} f^+(\x)\,\overline{C}(\x;\T)\,d\x
-\int_{\mathbb{R}^d} f^-(\x)\,\overline{C}(\x;\T)\,d\x.
\end{equation}
Then, \eqref{eq:envelope_bounds_multidim} implies
\begin{equation}
\label{eq:bracketing_multidim}
\underline{\mathcal{I}}(\T)\le \mathcal{I} \le \overline{\mathcal{I}}(\T),
\end{equation}
whenever the above integrals are finite.

When $d>1$, the resolution of the optimization problems in \eqref{eq:envelopes_multidim} is challenging. The recursive Algorithm \ref{algo_compund} based on Davenport-Schinzel sequences has been generalized in two-dimensions in~\cite{Schartz1990}, however its extension to higher dimension remains an open problem. We propose instead a novel approach, better exploiting the nice geometry of the problem. As explained hereafter, the proposed method requires an additional assumption of separability on the set of majorant/minorant functions, yielding a tractable and fast implementation of the piecewise envelope constructions.

%To design an efficient algorithm for approximating the integral $\mathcal{I}_d$, we restrict the integration domain to a hyper-rectangle $K \equiv \prod_{i=1}^d [a_i,b_i] \subset \mathbb{R}^d$. To ensure the computation of the envelopes is numerically tractable, we introduce two fundamental simplifying assumptions consistent with practical implementations.

\begin{assumption}
\label{ass:homosce}
Assumption~\ref{assump3} holds. Moreover, for every $\t \in \mathbb{R}^d$, 
\begin{align}
\label{eq:diag_BN_const}
\A(\t) &= \mathrm{Diag}(\boldsymbol{\beta}),\\%_{1}, \dots, \beta_{d}), \\ %\Sigma^{-1} =  
\B(\t) & = \mathrm{Diag}(\boldsymbol{\eta}),%_{d,1}, \dots, \eta_{d,d}), % = \Omega^{-1}
\end{align}
where $\boldsymbol{\beta} = (\beta_{i})_{1 \leq i \leq d}$ and $\boldsymbol{\eta} = (\eta_{i})_{1 \leq i \leq d}$ are positive vectors independant from $\t$. 
\end{assumption}
Under Assumption \ref{ass:homosce}, the Gaussian tangent bounds $\underline{b}(\cdot;\t)$ and $\overline{b}(\cdot;\t)$ factorize coordinate-wise, that is, for every $\x \in \mathbb{R}^d$,
\begin{equation}
\label{eq:bd_factorized_const}
\underline{b}(\x;\t) = \underline{C}(\t) \prod_{i=1}^d g(x_i; \underline{u}_{i}(\t), \beta_{i}^{-1}), \quad 
\overline{b}(\x;\t) = \overline{C}(\t) \prod_{i=1}^d g(x_i; \overline{u}_{i}(\t), \eta_{i}^{-1}).
\end{equation}
Let us emphasize that we do not assume here any homoscedastic/separability on the function $\pi$. Moreover, Assumption \ref{ass:homosce} holds trivially when the majorant and minorant functions are constructed from the calculation of the Lipschitz smoothness, and strong convexity, constants of $\phi$. The primary advantage of the constant curvature considered in Assumption \ref{ass:homosce} lies in the linearization of the problems \eqref{eq:envelopes_multidim} by passing to the log-space. Indeed, for the maximization problem (the reasoning being similar for the minimization one), we now have, for all $\x \in \mathbb{R}^d$, for all $\T \in \mathbb{R}^{d \times M}$,
\begin{equation}
\log \underline{C}(\x;\T) = -\frac{1}{2} \sum_{i=1}^d \beta_{i} x_i^2 + \max_{1 \le m \le M} \left(\boldsymbol{\omega}_m^\top \x + \alpha_m \right),
\label{eq:logmax}
\end{equation}
with, for every $m \in \{1,\ldots,M\}$, 
\begin{align}
\label{eq:affine}
\boldsymbol{\omega}_m & = \mathrm{Diag}(\boldsymbol{\beta}) \underline{\u}(\t_m)\\
\alpha_m & = \log \underline{C}(\t_m) - \frac{1}{2} \underline{\u}(\t_m)^\top \mathrm{Diag}(\boldsymbol{\beta}) \underline{\u}(\t_m) \nonumber\\
& - \frac{d}{2}\log(2\pi) + \frac{1}{2}\log|\mathrm{Diag}(\boldsymbol{\beta})|. \nonumber
\end{align}
Consequently, solving the maximization problem in \eqref{eq:logmax} now reduces to the comparison of $M$ affine functions with weight and offset terms $(\boldsymbol{\omega}_m,\alpha_m)_{1 \leq m \leq M}$. Geometrically, comparing affine functions results in building a partition of $\mathbb{R}^d$ into at most $M$ convex polyhedral cells $\{\underline{P}_1,\dots,\underline{P}_M\}$, each associated to a unique affine term \cite{deBerg2008,Ziegler1995}. Namely, for each $m\in\{1,\dots,M\}$, the associated dominant cell is defined by
\begin{equation}
\label{eq:power_cell}
\underline{P}_m
=
\left\{
\x\in \mathbb{R}^d
\;\middle|\;
\forall k\neq m,\;
\boldsymbol{\omega}_m^\top \x+\alpha_m
\ge
\boldsymbol{\omega}_k^\top \x+\alpha_k
\right\}.
\end{equation}
Each cell is obtained by intersecting $K$ with $M-1$ half-spaces, so that every $\underline{P}_m$ is a convex polyhedron (or a convex polytope if the computation is restricted to a bounded subset of $\mathbb{R}^d$). Some cells may be empty, so the partition consists of at most $M$ nonempty polyhedral regions. The resulting partition defines a so-called power diagram (also called weighted Voronoi diagram). Generalizing the standard Voronoi setting, the boundaries between adjacent cells of power diagrams are affine hyperplanes shifted by offset values $(\alpha_m)_{1 \leq m \leq M}$. The computation of a power diagram in dimension $d$ can be run efficiently, as it reduces to the construction of intersections of half-spaces, for instance via convex hull or polyhedral algorithms. In dimension $d=2$, this yields a decomposition of the plane into convex polygons. 

The integral of the obtained piecewise envelopes over any subset $K \subset \mathbb{R}^d$ then reduces to a sum of integrals of separable Gaussian functions over each corresponding cells. For instance, we have
\begin{equation}
\int_K f^+(\x)\,\underline{C}(\x;\T)\,d\x
=
\sum_{m=1}^M
\int_{\underline{P}_m} f^+(\x)\,\underline{b}(\x;\t_m)\,d\x.
\end{equation}
Similar formula hold for the other terms in \eqref{eq:Id_bounds_multidim} and \eqref{eq:Id_bar_bounds_multidim}, by introducing the minimization polyhedral cells $\{\bar{P}_1,\ldots,\bar{P}_M\}$. Efficient numerical quadrature (for instance, triangular quadrature) can be used to evaluate each term in a very accurate manner. To summarize, starting from $\T$ tangent points, we first solve the maximization/minimization problems \eqref{eq:envelopes_multidim} by converting them into a geometric partitioning starting from the affine parameters \eqref{eq:affine}, then we compute each resulting local integrations via quadrature over convex polyedral domains, and gather the results to obtain the integral bounds. To complete the practical implementation of the iterative bound construction, there remains to define how to refine the set of tangency points, to sequentially improve the bound accuracy and guarantee convergence. The complete convergence proof in multi-dimensional setting is omitted due to space constraints. The key ingredient to preserve the validity of the convergence results from the 1D setting, is to satisfy a generalized version of the compact refinement property introduced in Section~\ref{sec_convergence_analysis}. In our experiments in 2D case presented hereafter, we choose to add tangency points by fusing the minimum and maximum power diagrams, and computing the projection onto a pre-defined dyadic grid, of the vertices coordinates of the polygon leading to the largest integral gap. This leads by construction to a monotonic decrease of the approximation gap and can be implemented at low computational cost.

\subsection{Numerical Results}

Let us now present numerical results for the generalization of our approach in a multidimensional setting. We here consider the logistic regression example from Section~\ref{sec_logistic}, extended to the case $d=2$. We set
\begin{equation}
(\forall \x \in \mathbb{R}^2) \quad p(\x) = A \exp(-x_1^2/(2 s_1^2)) \exp(-x_2^2/(2 s_2^2)) \times \prod_{j=1}^J \exp(-\log(1 + \exp(y_j (\w_j^\top \x)))),
\end{equation}
with $A>0$, $y_j \in \{-1,1\}$, $\w_j \in \mathbb{R}^2$, and
\begin{equation}
(\forall \x \in \mathbb{R}^2) \quad f(\x) = x_1^2,
\end{equation}
so that $f^+ = f$ in that case. We set $s_1 = s_2 = 1.2$, $J=10$. The proposal distribution $q$ is an isotropic Gaussian with zero-mean and standard deviation $\theta = 2$ in both axis. The contour plots of $p$, $f p$, and $f^2 p^2/q$ are displayed in Fig.~\ref{fig:target2D}. We aim at integrating these three functions over $\mathbb{R}^2$.  

\begin{figure}[h!]
    \centering
    \begin{tabular}{ccc}
    \includegraphics[width=0.3\linewidth]{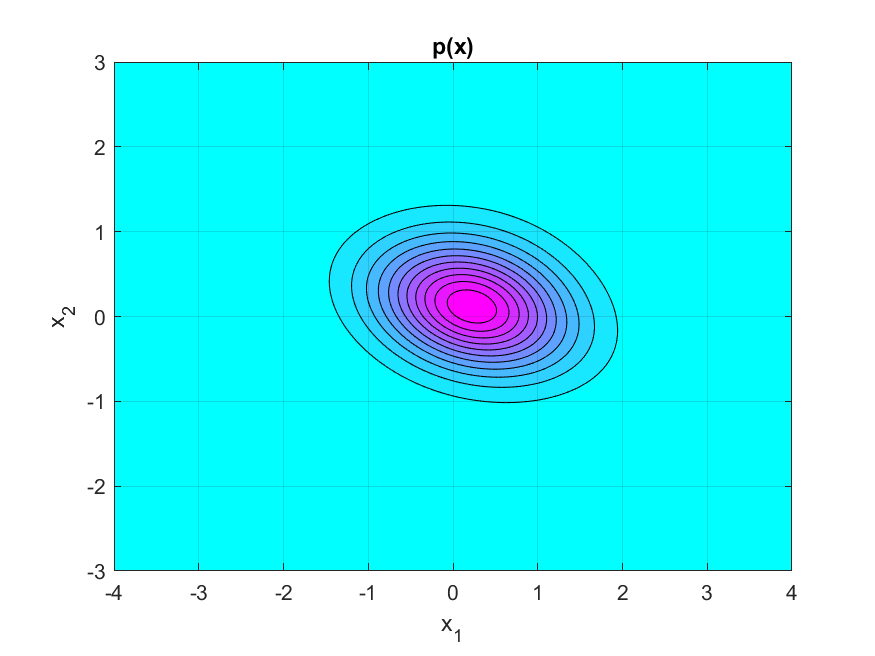}  &
    \includegraphics[width=0.3\linewidth]{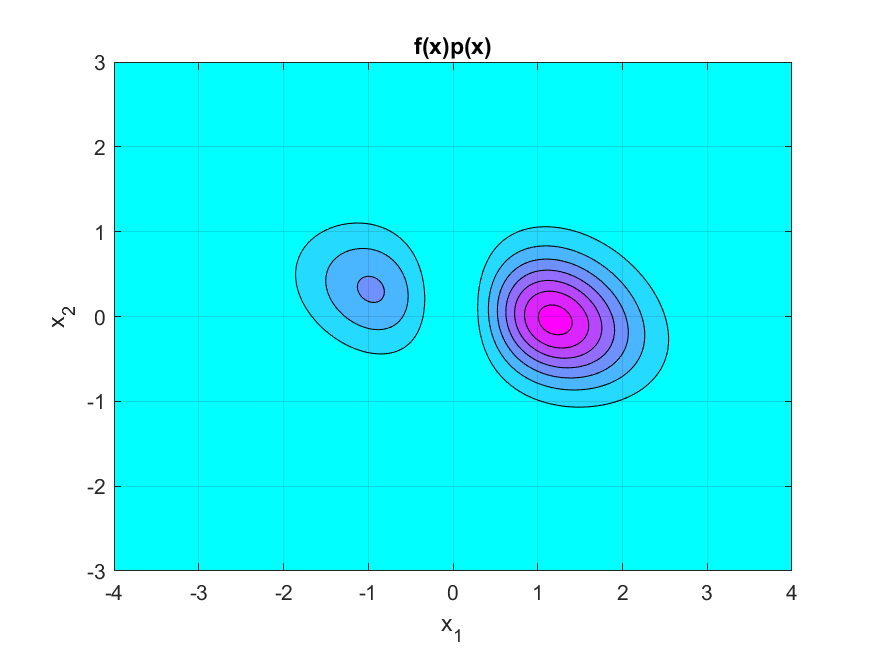} & \includegraphics[width=0.3\linewidth]{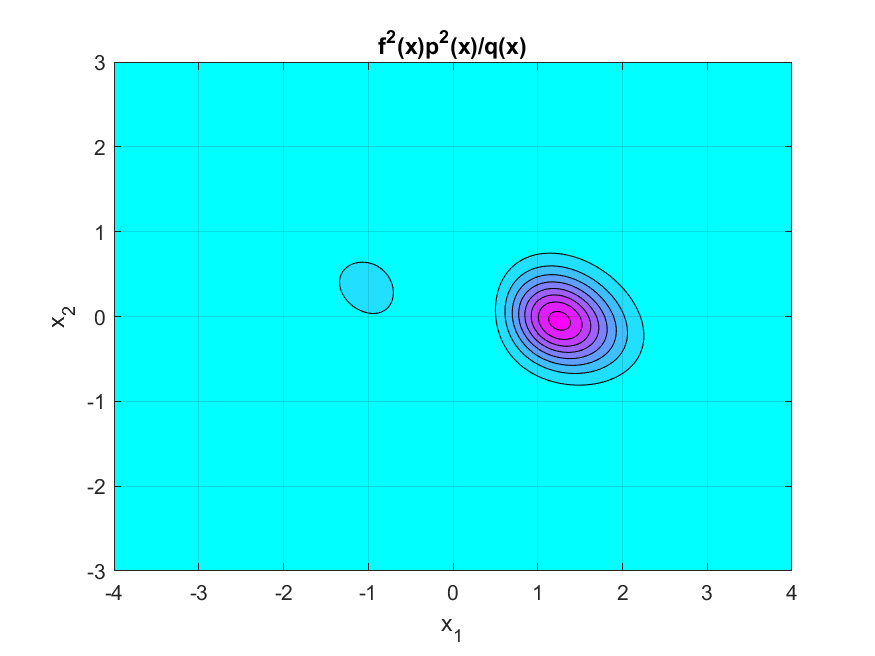}
    \end{tabular}
    \caption{Display of the integrand contour plots, for the 2D example with $f(x) = x_1^2$.}
    \label{fig:target2D}
\end{figure}
Assumptions \ref{assump3} and \ref{ass:homosce} hold with
\begin{equation}
\boldsymbol{\beta} = \|\mathrm{Diag}(\y) \W\|^2 \left[\begin{array}{c}1/4 \\1/4\end{array}\right] + \left[\begin{array}{c}1/s_1^2 \\ 1/s_2^2\end{array}\right],
\end{equation}                                                 
with $\y = (y_j)_{1 \leq j \leq J} \in \{-1,1\}^J$, $\W = \left[\w_1^\top, \ldots, \w_J^\top \right] \in \mathbb{R}^{J \times 2}$, and
\begin{equation}
\boldsymbol{\eta} = \left[\begin{array}{c}1/s_1^2 \\ 1/s_2^2\end{array}\right].
\end{equation} 
Similarly to in the 1D case (see Appendix B), we must check the following range conditions to ensure the validity of Gaussian bounds on the ratio $p^2/q$,
\begin{equation}
\begin{cases}
\beta_1^{-1} \in (0,\frac{\theta}{2}),\quad
\beta_2^{-1} \in (0,\frac{\theta}{2}),\\
\eta_1^{-1} \in (0,\frac{\theta}{2}),\quad
\eta_2^{-1} \in (0,\frac{\theta}{2}),
\end{cases}
\end{equation}
which are satisfied empirically in our practical example. We provide in Figure~\ref{fig:voronoif}, the evolution of the sets $\{\bar{P}_m\}$ and $\{\underline{P}_m\}$, along refinement iterations, where the next refined area is depicted in red color, for the integration of $p$. We also display in Figure~\ref{fig:Z_algo2_2D} the evolution of the integral approximation gap along iterations, showing the fast convergence of the algorithm. Similarly as in 1D case, we compute the bounds $(\underline{V},\overline{V})$ on the variance of the IS approach, using $20$ iterations (i.e., refinement steps) for our approach, and $N=200$ IS samples. The results are presented in Fig.~    \ref{fig:varUIS_2D}, showing that the obtained bounds offer a fast and accurate manner to decide a good setting of the proposal parameter $\theta$. %We obtain $(\underline{V},\overline{V}) = (1.69 \times 10^{-2},2.36 \times 10^{-2})$, to be compared with the empirical variance $V_e = 1.86 \times 10^{-2}$ evaluated for $N_{\text{runs}} = 10^5$ Monte-Carlo runs. The main source of discrepancy comes from the bounds of the ratio in \eqref{eq:bndI2}, which appears challenging to accurately approximation in 2D.
 
\begin{figure}[h!]
    \centering
    \begin{tabular}{ccc}
    \includegraphics[width=0.3\linewidth]{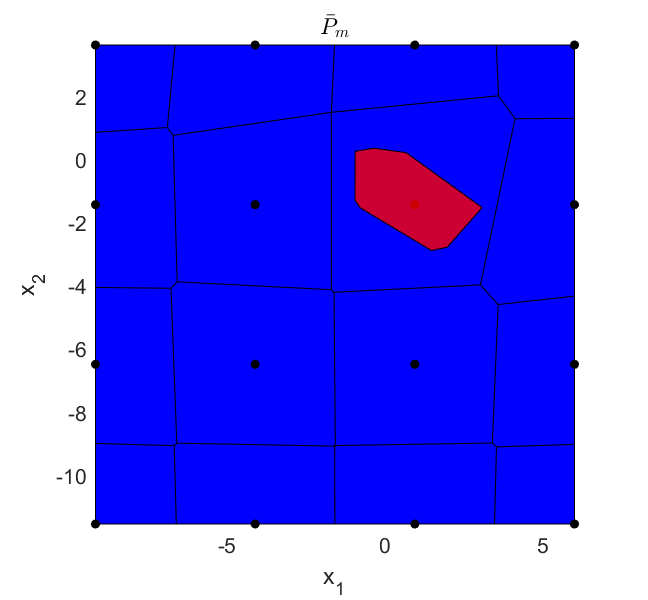}  &
    \includegraphics[width=0.3\linewidth]{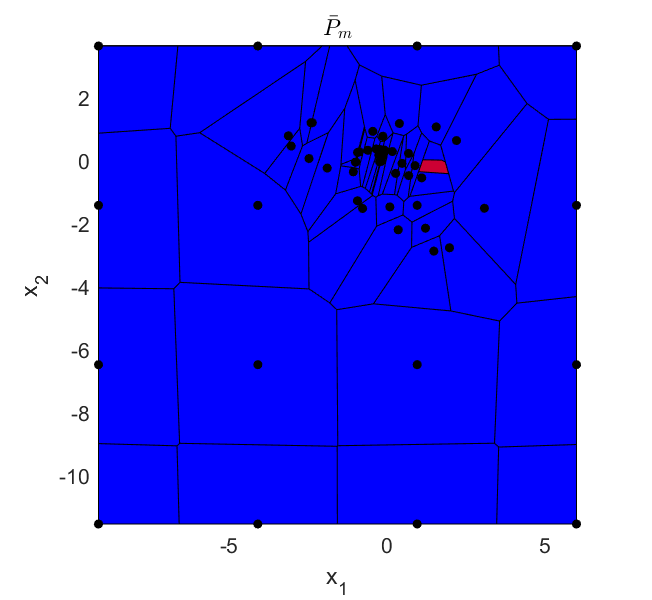} &
    \includegraphics[width=0.3\linewidth]{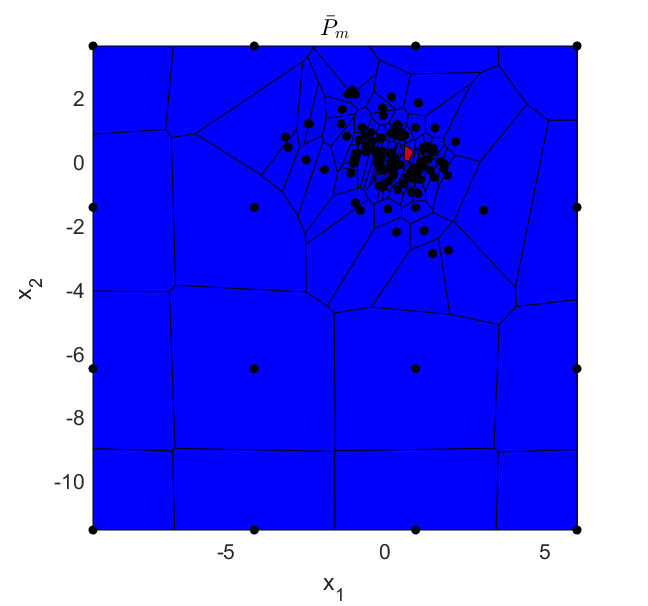}\\
        \includegraphics[width=0.3\linewidth]{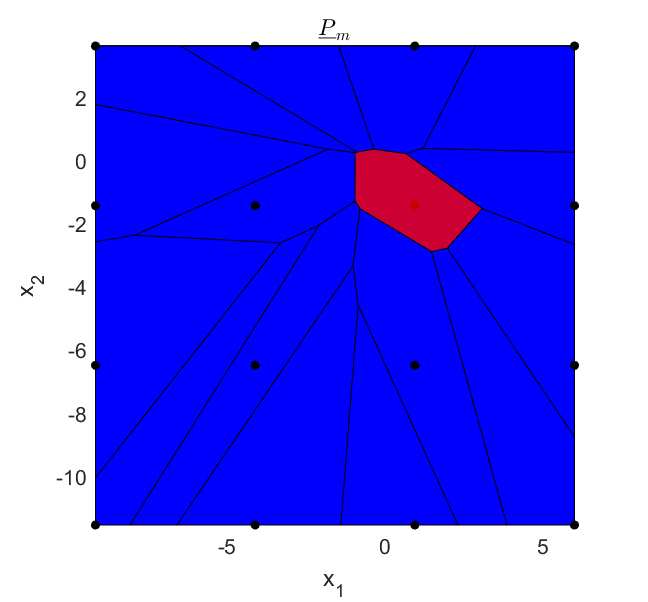}  &
    \includegraphics[width=0.3\linewidth]{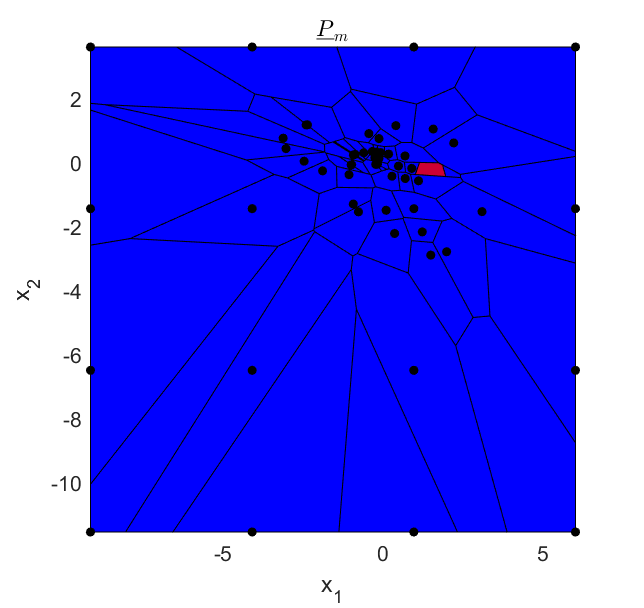} &
    \includegraphics[width=0.3\linewidth]{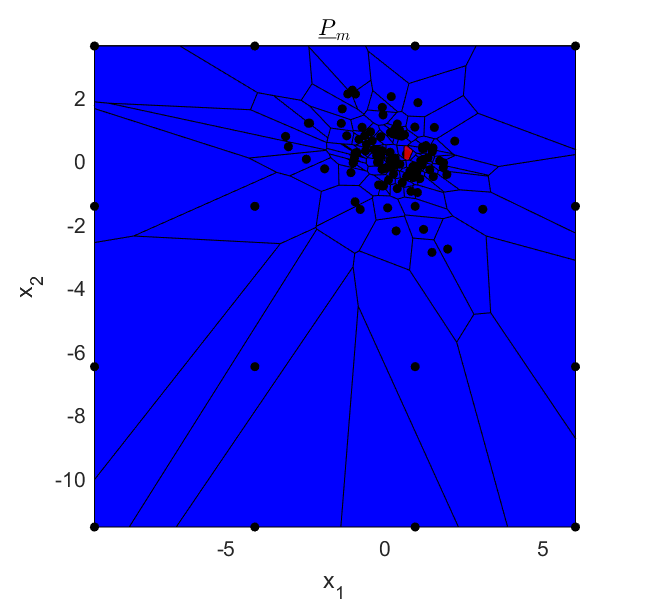}\\
    $n=1, m=16, G^{(n)} = 5.2 \times 10^{-4}$ & $n=5, m=64, G^{(n)} = 2.7 \times 10^{-5}$ & $n=10, m=143, G^{(n)} = 1.2 \times 10^{-5}$
    \end{tabular}
    \caption{\cblue{2D example: Evolution of the convex polyhedra $\{\bar{P}_m\}$ and $\{\underline{P}_m\}$, along refinement iterations, for the computation of the bounds \eqref{eq:bndZ}. In red color, we highlight the polyhedra used for the next refinement step. Black dots represent the tangency points. We display the cardinality $m$ as well as the integral approximation gap.}}
    \label{fig:voronoif}
\end{figure}

\begin{figure}[H]
\centering
\begin{tabular}{@{}c@{}c@{}c@{}}
\includegraphics[height=3.8cm]{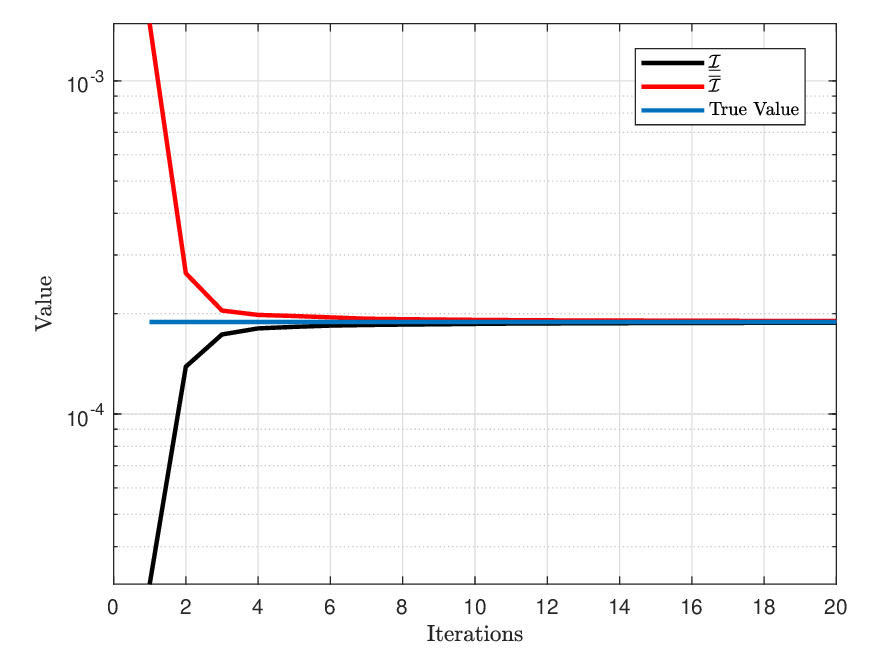} & 
\includegraphics[height=4cm]{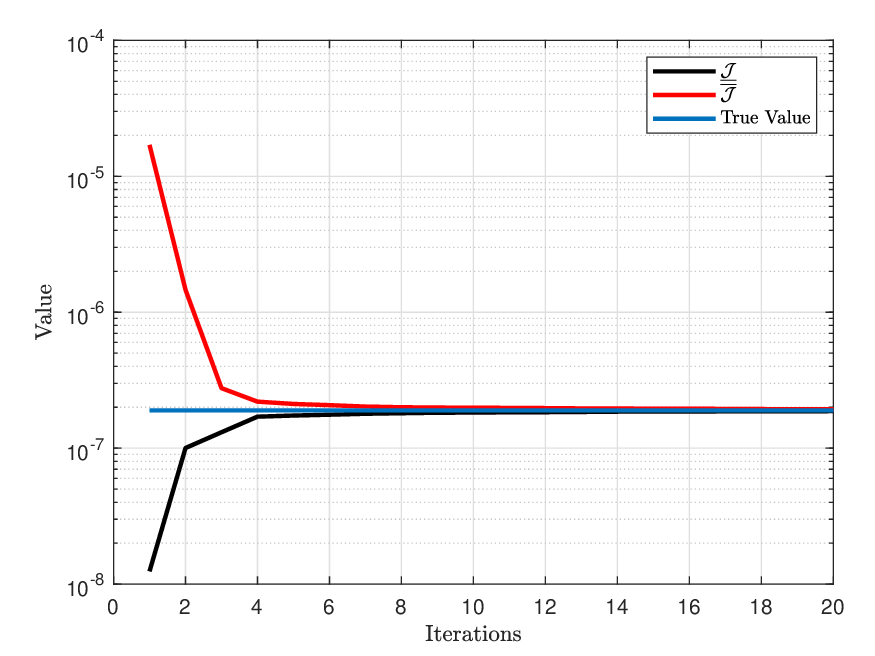}
& 
\includegraphics[height=4cm]{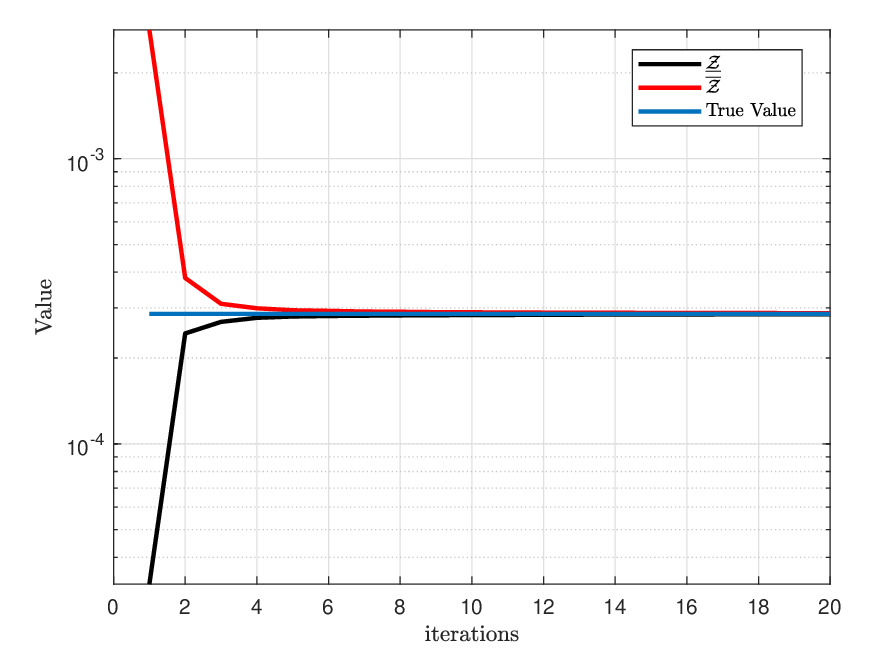} 
\end{tabular}
\caption{\cblue{Evolution of the integral bounds \eqref{eq:bndI1}, \eqref{eq:bndI2}, \eqref{eq:bndZ} along iterations $n$ of the proposed method generalized to the 2D case.}}
\label{fig:Z_algo2_2D}
\end{figure}

\begin{figure}
    \centering    \includegraphics[width=0.7\linewidth]{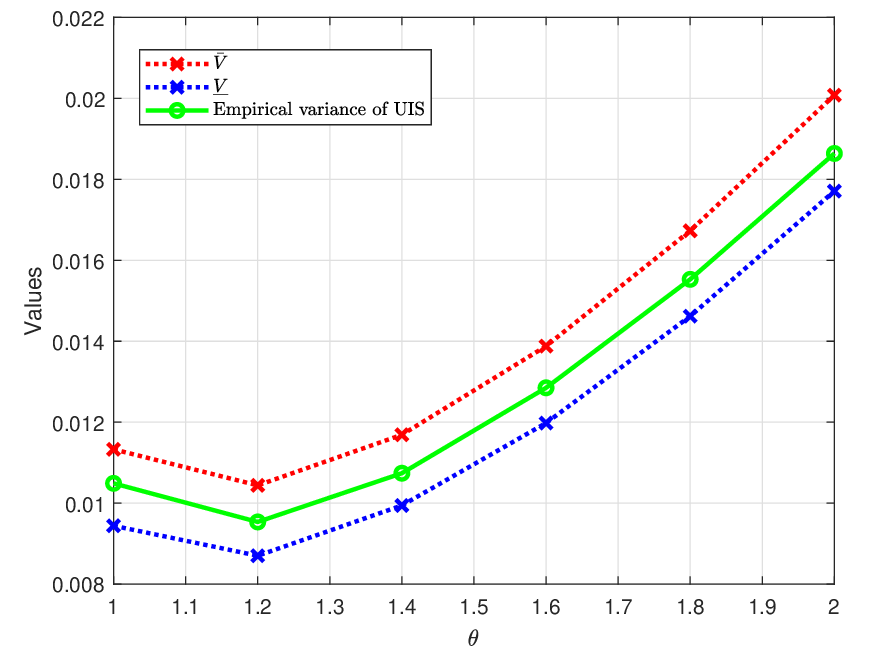}
    \caption{\cblue{Proposed bounds $(\underline{V},\overline{V})$, and empirical variance of UIS for $N_{\text{runs}} = 10^4$, computed for different values of the proposal parameter $\theta$ in the 2D example.}}
    \label{fig:varUIS_2D}
\end{figure}

%% file: Appendix_truncated_gaussian.tex
In this appendix, we provide explicit formulas for the computation of monomial moments with respect to Gaussian and truncated Gaussian distributions, which will be essential tools to practically implement the proposed approach. %evaluate bounds such as $\underline{\Ic}(t)$ and $\overline{\Ic}(t)$.

\subsection*{A.1. Moments of the Gaussian distribution}

Let $X \sim \mathcal{N}(\mu, \sigma^2)$ and $k \in \mathbb{N}$. The $k$-th moment of $X$ can be computed explicitly:

\begin{itemize}
    \item If $\mu = 0$, then odd moments vanish and the even moments are given by
    $$
    \mathbb{E}[X^{2k}] = \sigma^{2k} (2k - 1)!! = \sigma^{2k} \cdot \frac{(2k)!}{2^k k!}.
    $$
    \item For general $\mu \in \mathbb{R}$, the $k$-th moment is
    $$
    \mathbb{E}[X^k] = \sum_{j = 0}^{\lfloor k/2 \rfloor} \binom{k}{2j} \mu^{k - 2j} \sigma^{2j} (2j - 1)!!,
    $$
    see e.g.~\cite{jorge2018}.
\end{itemize}

\subsection*{A.2. Moments of the truncated Gaussian distribution}

Let $X \sim \mathcal{N}(\mu, \sigma^2)$ truncated to the interval $[a, b]$. The $k$-th conditional moment is defined as
$$
m_k := \mathbb{E}[X^k \mid a \leq X \leq b].
$$
These moments satisfy the recursive formula~\cite{Orjebin2014}:
$$
m_k = (k-1)\sigma^2 m_{k-2} + \mu m_{k-1} - \sigma \cdot \frac{b^{k-1} \phi\left(\frac{b - \mu}{\sigma}\right) - a^{k-1} \phi\left(\frac{a - \mu}{\sigma}\right)}{\Phi\left(\frac{b - \mu}{\sigma}\right) - \Phi\left(\frac{a - \mu}{\sigma}\right)},
$$
for $k \geq 1$, with base cases $m_{-1} = 0$, $m_0 = 1$.

Hereabove, we denoted $(\forall x \in \mathbb{R}) \;\phi(x) = \frac{1}{\sqrt{2\pi}} e^{-x^2/2}$, the standard normal probability distribution function, and $(\forall x \in \mathbb{R}) \; \Phi(x) = \int_{-\infty}^x \phi(t)\,dt$, the standard normal cumulative distribution function.

In particular, we have~(\cite{Orjebin2014,jorge2018}):

\begin{align*}
m_1 &= \mu - \sigma \cdot \frac{\phi(\beta) - \phi(\alpha)}{\Phi(\beta) - \Phi(\alpha)}, \\
m_2 &= \mu^2 + \sigma^2 - \sigma \cdot \frac{(\mu + b)\phi(\beta) - (\mu + a)\phi(\alpha)}{\Phi(\beta) - \Phi(\alpha)},
\end{align*}

where $\alpha = \frac{a - \mu}{\sigma}$ and $\beta = \frac{b - \mu}{\sigma}$. When $\alpha$ and $\beta$ are both large (in absolute value), it is recommended to use numerically stabilized variants involving the scaled complementary error function \texttt{erfcx}, as detailed in~\cite{jorge2018,chevillard2012}.

%% file: Appendix_bounds_p2q.tex
Let, for all $x \in \mathbb{R}$, $p(x) = \exp(-\phi(x))$, with $\phi$ satisfying Assumptions \ref{assump1} and \ref{assump2}. Let, for all $x \in \mathbb{R}$, $q(x)=g(x;\mu,\theta)
=
\frac{1}{\sqrt{2\pi\theta^2}}
\exp\left(-\frac{(x-\mu)^2}{2\theta^2}\right),
$
a Gaussian proposal distribution, with mean $\mu \in \mathbb{R}$, and standard deviation $\theta>0$. According to Lemmas~\ref{lem:ass1} and \ref{lem:ass2}, 
for every $(x,t)\in\mathbb{R}^2$,
\begin{equation}
\label{bounds_over_p}
\underline{C}(t)\, g(x;\underline{u}(t),\underline{\sigma}(t)) \leq p(x) 
\leq \overline{C}(t)\, g(x;\overline{u}(t),\overline{\sigma}(t))
\end{equation}
with
\begin{equation}
\begin{cases}
\underline{\sigma}(t)=1/\sqrt{\beta(t)},\\
\underline{u}(t)=t-\underline{\sigma}(t)^2\dot{\phi}(t),\\
\underline{C}(t)=\sqrt{2\pi \underline{\sigma}(t)^2}
\exp\left(-\phi(t)+\frac{\underline{\sigma}(t)^2}{2}(\dot{\phi}(t))^2\right),\\
\overline{\sigma}(t)=1/\sqrt{\nu(t)},\\
\overline{u}(t)=t-\overline{\sigma}(t)^2\dot{\phi}(t),\\
\overline{C}(t)=\sqrt{2\pi \overline{\sigma}(t)^2}
\exp\left(-\phi(t)+\frac{\overline{\sigma}(t)^2}{2}(\dot{\phi}(t))^2\right).
\end{cases}
\end{equation}
Let us show that, under suitable assumption on the proposal standard deviation $\theta$, we can also derive Gaussian bounds for the ratio
$
x \to \frac{p(x)^2}{q(x)}.
$

\begin{lemma}
\label{lem:bounds_p2_over_q}
Assume that, for all $t \in \mathbb{R}$, $(\underline{\sigma}(t)^2,\overline{\sigma}(t)^2)\in (0,2\theta^2)$, then, for every $(x,t) \in \mathbb{R}^2$,
\begin{equation}
\underline{D}(t)\, g(x;\underline{v}(t),\underline{\rho}(t)) \leq
\frac{p(x)^2}{q(x)}\leq \overline{D}(t)\, g(x;\overline{v}(t),\overline{\rho}(t)),
\label{eq:boundp2q}
\end{equation}
where
\begin{equation}
\begin{cases}
\underline{\rho}(t)^2
=
\left(\dfrac{2}{\underline{\sigma}(t)^2}-\dfrac{1}{\theta^2}\right)^{-1},\\[2ex]
\underline{v}(t)
=
\underline{\rho}(t)^2
\left(
\dfrac{2\underline{u}(t)}{\underline{\sigma}(t)^2}
-\dfrac{\mu}{\theta^2}
\right),\\[2ex]
\underline{D}(t)
=
\underline{C}(t)^2
\dfrac{\underline{\rho}(t)}{\underline{\sigma}(t)^2}\theta
\exp\left(
-\dfrac{\underline{u}(t)^2}{\underline{\sigma}(t)^2}
+\dfrac{\mu^2}{2\theta^2}
+\dfrac{\underline{v}(t)^2}{2\underline{\rho}(t)^2}
\right),\\
\overline{\rho}(t)^2
=
\left(\dfrac{2}{\overline{\sigma}(t)^2}-\dfrac{1}{\theta^2}\right)^{-1},\\[2ex]
\overline{v}(t)
=
\overline{\rho}(t)^2
\left(
\dfrac{2\overline{u}(t)}{\overline{\sigma}(t)^2}
-\dfrac{\mu}{\theta^2}
\right),\\[2ex]
\overline{D}(t)
=
\overline{C}(t)^2
\dfrac{\overline{\rho}(t)}{\overline{\sigma}(t)^2}\theta
\exp\left(
-\dfrac{\overline{u}(t)^2}{\overline{\sigma}(t)^2}
+\dfrac{\mu^2}{2\theta^2}
+\dfrac{\overline{v}(t)^2}{2\overline{\rho}(t)^2}
\right).
\end{cases}
\end{equation}
\end{lemma}

\begin{proof}
We have
$$
(\forall x \in \mathbb{R}) \quad q(x)=\frac{1}{\sqrt{2\pi \theta^2}}
\exp\left(-\frac{(x-\mu)^2}{2\theta^2}\right),
$$
hence
\begin{equation}
(\forall x \in \mathbb{R}) \quad
\frac{p(x)^2}{q(x)}
=
\exp(-2\phi(x))\sqrt{2\pi \theta^2}
\exp\left(\frac{(x-\mu)^2}{2\theta^2}\right).
\end{equation}
Using the left inequality in~\eqref{bounds_over_p}, and the above expression of $q$ leads, for every $t \in \mathbb{R}$, to
% Under Assumption~\ref{assump1}, Lemma~\ref{lem:ass1} yields
% \begin{equation}
% \phi(x)
% \leq
% \frac{1}{2\underline{\sigma}(t)^2}(x-\underline{u}(t))^2
% +\phi(t)
% -\frac{\underline{\sigma}(t)^2}{2}(\dot{\phi}(t))^2.
% \end{equation}
% Multiplying by $-2$, we get
% \begin{equation}
% -2\phi(x)
% \geq
% -\frac{1}{\underline{\sigma}(t)^2}(x-\underline{u}(t))^2
% -2\phi(t)
% +\underline{\sigma}(t)^2(\dot{\phi}(t))^2.
% \end{equation}
% Hence
\begin{align}
(\forall x \in \mathbb{R}) \quad \frac{p(x)^2}{q(x)}
&\geq
\sqrt{2\pi \theta^2}
\exp\left(
-\frac{1}{\underline{\sigma}(t)^2}(x-\underline{u}(t))^2
+\frac{(x-\mu)^2}{2\theta^2}
-2\phi(t)
+\underline{\sigma}(t)^2(\dot{\phi}(t))^2
\right).
\end{align}
Expanding the quadratic terms gives, for every $(x,t) \in \mathbb{R}^2$,
\begin{multline}
-\frac{1}{\underline{\sigma}(t)^2}(x-\underline{u}(t))^2
+\frac{(x-\mu)^2}{2\theta^2}
=
-\left(
\frac{1}{\underline{\sigma}(t)^2}
-\frac{1}{2\theta^2}
\right)x^2
+
\left(
\frac{2\underline{u}(t)}{\underline{\sigma}(t)^2}
-\frac{\mu}{\theta^2}
\right)x  
-\frac{\underline{u}(t)^2}{\underline{\sigma}(t)^2}
+\frac{\mu^2}{2\theta^2}.
\label{eq_170}
\end{multline}
Since we assumed that, for every $t \in \mathbb{R}$, $\underline{\sigma}(t)^2<2\theta^2$, then
$$
\frac{1}{\underline{\sigma}(t)^2}-\frac{1}{2\theta^2}>0.
$$
Therefore, by completing the square, we can rewrite \eqref{eq_170} for every $(x,t) \in \mathbb{R}^2$ as
\begin{equation}
-\frac{1}{\underline{\sigma}(t)^2}(x-\underline{u}(t))^2
+\frac{(x-\mu)^2}{2\theta^2}
=
-\frac{1}{2\underline{\rho}(t)^2}(x-\underline{v}(t))^2
-\frac{\underline{u}(t)^2}{\underline{\sigma}(t)^2}
+\frac{\mu^2}{2\theta^2}
+\frac{\underline{v}(t)^2}{2\underline{\rho}(t)^2},
\end{equation}
with the notation
\begin{equation}
(\forall t \in \mathbb{R}) \quad \underline{\rho}(t)^2
=
\left(\frac{2}{\underline{\sigma}(t)^2}-\frac{1}{\theta^2}\right)^{-1}
\end{equation}
and
\begin{equation}
(\forall t \in \mathbb{R}) \quad \underline{v}(t)
=
\underline{\rho}(t)^2
\left(
\frac{2\underline{u}(t)}{\underline{\sigma}(t)^2}
-\frac{\mu}{\theta^2}
\right).
\end{equation}
Thus, for every $(x,t) \in \mathbb{R}^2$,
\begin{align}
\frac{p(x)^2}{q(x)}
&\geq
\sqrt{2\pi \theta^2}
\exp\left(
-2\phi(t)
+\underline{\sigma}(t)^2(\dot{\phi}(t))^2
-\frac{\underline{u}(t)^2}{\underline{\sigma}(t)^2}
+\frac{\mu^2}{2\theta^2}
+\frac{\underline{v}(t)^2}{2\underline{\rho}(t)^2}
\right)
\exp\left(
-\frac{(x-\underline{v}(t))^2}{2\underline{\rho}(t)^2}
\right).
\end{align}
    Using the expression of $g(x;\underline{v}(t),\underline{\rho}(t))$ proves the left inequality in \eqref{eq:boundp2q}. The right inequality in \eqref{eq:boundp2q} is proved in a similar manner, using the right inequality in \eqref{bounds_over_p} and the assumption on the range of $\overline{\sigma}(t)$.

\end{proof}

%% file: Appendix_useful_lemmas.tex
\begin{lemma}[Pointwise extrema preserve a common Lipschitz constant]
\label{lemm:max_min_lip_function}

Let $(f_i)_{i\in I}$ be a family of real-valued functions on a set $E\subset \mathbb{R}^d$, where $I$ is an arbitrary index set. Assume that there exists $L\geq 0$ such that, for every $i\in I$, the function $f_i$ is $L$-Lipschitz on $E$. Then the pointwise supremum and infimum,
$$
f_{\sup}(x):=\sup_{i\in I} f_i(x),
\qquad
f_{\inf}(x):=\inf_{i\in I} f_i(x),
$$
are also $L$-Lipschitz on $E$, provided they are finite-valued.
\end{lemma}

\begin{proof}
We first prove the result for $f_{\max}$. Let $x,y\in\mathbb{R}^n$. For every $i\in I$, since $f_i$ is $L$-Lipschitz, we have
\begin{equation*}
f_i(x)\leq f_i(y)+L\|x-y\|.
\end{equation*}
Taking the supremum over $i\in I$, we obtain
\begin{equation*}
f_{\max}(x)
=\sup_{i\in I} f_i(x)
\leq \sup_{i\in I}\bigl(f_i(y)+L\|x-y\|\bigr)
= \sup_{i\in I} f_i(y)+L\|x-y\|
= f_{\max}(y)+L\|x-y\|.
\end{equation*}
Exchanging the roles of $x$ and $y$ gives
\begin{equation*}
f_{\max}(y)\leq f_{\max}(x)+L\|x-y\|.
\end{equation*}
Hence,
\begin{equation*}
|f_{\max}(x)-f_{\max}(y)|\leq L\|x-y\|,
\end{equation*}
so $f_{\max}$ is $L$-Lipschitz.

We now prove the result for $f_{\min}$. For every $i\in I$, we also have
\begin{equation*}
f_i(x)\leq f_i(y)+L\|x-y\|.
\end{equation*}
Taking the infimum over $i\in I$ yields
\begin{equation*}
f_{\min}(x)
=\inf_{i\in I} f_i(x)
\leq \inf_{i\in I}\bigl(f_i(y)+L\|x-y\|\bigr)
= \inf_{i\in I} f_i(y)+L\|x-y\|
= f_{\min}(y)+L\|x-y\|.
\end{equation*}
Permutting the roles of $x$ and $y$, we get
\begin{equation*}
f_{\min}(y)\leq f_{\min}(x)+L\|x-y\|.
\end{equation*}
Therefore,
\begin{equation*}
|f_{\min}(x)-f_{\min}(y)|\leq L\|x-y\|,
\end{equation*}
which concludes the proof.
\end{proof}